\documentclass[12pt]{article}
\usepackage{amsmath, amsthm, amsfonts, amssymb,mathrsfs}
\usepackage{mathtools} 
\usepackage[utf8]{inputenc}
\usepackage{tikz}
\usepackage{tikz-cd}
\usetikzlibrary{decorations.markings}
\usepackage[all,cmtip,pdf]{xy}
\usepackage{url}
\usepackage{enumitem}
\usepackage[normalem]{ulem}
\newcommand{\itemref}[1]{(\ref{#1})}
\usepackage{comment}
\usepackage{eucal}
\usepackage{CJKutf8}

\theoremstyle{plain}
\newtheorem{thm}{Theorem}[section]
\newtheorem{lemma}[thm]{Lemma}
\newtheorem{prop}[thm]{Proposition}
\newtheorem{cor}[thm]{Corollary}

\theoremstyle{definition}
\newtheorem{defn}[thm]{Definition}

\newtheorem{rem}[thm]{Remark}

\numberwithin{equation}{section}

\advance \topmargin by -\headheight
\advance \topmargin by -\headsep
\evensidemargin \oddsidemargin
\newcommand{\W}{\mathscr{W}}
\newcommand{\mf}{\mathfrak}
\newcommand{\mc}{\mathcal}
\newcommand{\+}{\oplus}
\newcommand{\on}{\operatorname}

\DeclareMathOperator{\ad}{ad}
\DeclareMathOperator{\aut}{Aut}

\DeclareMathOperator{\en}{End}

\DeclareMathOperator{\Id}{Id}
\DeclareMathOperator{\im}{Im}

\DeclareMathOperator{\Jac}{Jac}
\DeclareMathOperator{\lie}{\mathbf{Lie}}
\DeclareMathOperator{\Mat}{Mat}

\DeclareMathOperator{\Pic}{Pic}
\DeclareMathOperator{\res}{Res}

\DeclareMathOperator{\spec}{Spec}
\DeclareMathOperator{\specm}{Specm}

\DeclareMathOperator{\tr}{Tr}

\newcommand{\sll}{\mathfrak{sl}}
\newcommand{\al}{\alpha}
\newcommand{\la}{\lambda}
\newcommand{\La}{\Lambda}

\newcommand{\D}{\Delta}

\newcommand{\C}{\mathbb{C}}

\newcommand{\GG}{\mathbb{G}}

\newcommand{\R}{\mathbb{R}}
\newcommand{\Z}{\mathbb{Z}}

\newcommand{\g}{\mathfrak{g}}
\newcommand{\h}{\mathfrak{h}}
\newcommand{\n}{\mathfrak{n}}

\newcommand{\CA}{\mathcal{A}}
\newcommand{\CC}{\mathcal{C}}
\newcommand{\CD}{\mathcal{D}}
\newcommand{\CE}{\mathcal{E}}
\newcommand{\CF}{\mathcal{F}}

\newcommand{\CH}{\mathcal{H}}
\newcommand{\CL}{\mathcal{L}}

\newcommand{\CN}{\mathcal{N}}
\newcommand{\CP}{\mathcal{P}}

\newcommand{\CV}{\mathcal{V}}

\newcommand{\OO}{\mathcal{O}}

\newcommand{\ZZ}{\mathbb{Z}}
\newcommand{\vac}{\left|0\right\rangle}

\newcommand{\Top}{\on{top}}
\newcommand{\peak}{\on{peak}}
\newcommand{\stab}{\on{stab}}
\renewcommand{\mod}{\text{\textendash}\mathsf{mod}}

\newcommand{\bracsmall}[2]{\left[\begin{smallmatrix} #1 \\ #2 \end{smallmatrix}\right]}

\newcommand{\wtil}[1]{\widetilde#1}
\newcommand{\what}[1]{\widehat#1}

\begin{document}

\begin{center}

{\LARGE\bfseries
Modular invariance of characters\\
of quasi-lisse vertex algebras
\par}

\vspace{1.2em}

{\normalsize
Tomoyuki Arakawa\footnote{Okinawa Institute of Science and Technology (OIST),
\texttt{tomoyuki.arakawa@oist.jp}},
Jethro van Ekeren\footnote{Instituto de Matem\'{a}tica Pura e Aplicada (IMPA), \texttt{jethro@impa.br}}\quad and\quad
Hao Li\footnote{Okinawa Institute of Science and Technology (OIST), \texttt{hao-li@oist.jp}}
\par}

\vspace{1.5em}

\end{center}

\vspace*{5mm}

\noindent
{\small
\textbf{Abstract.} 
We study spaces of conformal blocks associated with line bundles over elliptic curves, with coefficients in a vertex algebra. For vertex algebras satisfying suitable finiteness and semisimplicity conditions, which are met by all admissible affine vertex algebras as well as admissible $W$-algebras associated with nilpotent elements of standard Levi type, we prove the holonomicity of the sheaf of conformal blocks over the moduli space of bundles. Furthermore, we show that the space of flat sections of the associated Jacobi-invariant connection is spanned by trace functions on modules. This result provides a substantial generalization of the celebrated theorem of Yongchang Zhu to quasi-lisse vertex algebras. As a special case, we deduce that for affine vertex algebras at admissible level, the dimension of the space of conformal blocks coincides with the number of admissible weights at that level.
}


\section{Introduction}

Let $(V,\omega)$ be a conformal vertex algebra of central charge $c$. In the seminal work \cite{Z96} Zhu has proved that if $V$ is lisse and rational, then the trace functions $S_W$ on irreducible $V$-modules $W$, defined by
\begin{align}\label{eq:intro-zhu-trace}
S_W(u,\tau)=\tr_W u_0 q^{L_0-c/24},
\end{align}
span the space of genus-one conformal blocks of $V$, and constitute a finite dimensional vector valued modular form on $SL_2(\Z)$. Zhu's theorem is a key part of the mathematical description of conformal field theories and has important applications to the structure theory of vertex algebras.

The purpose of this article is to establish an analogue of Zhu's theorem for a broad class of {\emph{quasi-lisse}} vertex algebras
\cite{Arakawa-Kawasetsu}
 by replacing ordinary conformal blocks on elliptic curves by conformal blocks associated with pairs $(E,\CL)$ consisting of an elliptic curve $E$ together with a line bundle $\CL$ over $E$. The corresponding trace functions depend on both the elliptic moduli parameter $\tau$ and another parameter $\alpha$ corresponding to the line bundle. The trace functions are equivariant under an action of the Jacobi group $SL_2(\Z) \ltimes \Z^2$ extending that of the modular group $SL_2(\Z)$.

Let $h$ be an element of the weight-one space $V_1$ of $V$, and suppose that $h_0$ acts on $V$ semisimply with integer eigenvalues. We refer to such a vector as a current, and to the eigenspace decomposition with respect to $h_0$ as the charge grading of $V$. If $X$ is a smooth complex algebraic curve and $\CL$ a holomorphic line bundle over $X$, one may associate to $(V,\omega,h)$ a twisted chiral algebra over $X$, and a sheaf over the moduli space of pairs $(X, \CL)$ which we refer to as 
\emph{charged conformal blocks} \cite{BD}\cite[Chapter 7]{FBZ}. In the genus-one case we consider the family $T \rightarrow B$ of pairs $(E,\CL)_{(\alpha,\tau)}$ parametrised by $(\alpha,\tau) \in B= \C \times\CH$, where $\CH$ is the complex upper half plane, $E=E_\tau$ is the elliptic curve with modulus $\tau$ and $\CL$ is the holomorphic line bundle of degree zero determined by $\alpha$. See Section \ref{sec:mod.ell} below for details.

To formulate our results we introduce a charged analogue of the quasi-lisse condition. Let $R_V$ denote Zhu's Poisson algebra and let $R_V^h$ be the Poisson algebra obtained as the quotient of $R_V$ by the commutative algebra ideal generated by vectors of nonzero charge. We say that $(V,\omega,h)$ is \emph{stably quasi-lisse} if 
{the Poisson variety $\on{Spec}R_V^h$ has finitely many symplectic leaves}.
This condition is satisfied in many important classes of examples which we discuss below, including 
{all quasi-lisse affine vertex algebras and all quasi-lisse $W$-algebras for an appropriate choice of current $h$.}

Our first result concerns the geometric structure of the sheaf of charged conformal blocks.
\begin{thm}\label{thm:intro.1}
Let $(V,\omega,h)$ be a charged conformal vertex algebra which we assume finitely strongly generated and stably quasi-lisse. Then the sheaf $\CC((E,\CL),V)$ of charged conformal blocks over $B$ carries the structure of a holonomic $\CD$-module with regular singularities on the lattice $N\alpha\in \Z+\Z\tau$, for some positive integer $N$. In particular, away from this lattice $\CC$ is a vector bundle with integrable connection $\nabla$.
\end{thm}
When $h=0$ the charge grading is trivial and the condition of being stably quasi-lisse reduces to the quasi-lisse condition. In this case Theorem~ \ref{thm:intro.1} recovers the finite dimensionality result of \cite{Arakawa-Kawasetsu}.

Our second result identifies charged conformal blocks with trace functions on an appropriate category of $V$-modules. Since quasi-lisse vertex algebras generally possess non ordinary modules, the traces \eqref{eq:intro-zhu-trace} are not well defined in general. Instead we consider the trace functions
\begin{align}\label{eq:trace.def.intro}
S_W(u,\alpha,\tau)=\tr_W u_0 y^{h_0}q^{L_0-c/24},
\end{align}
where $y = e^{2\pi i \al}$ and, as above, $q = e^{2\pi i \tau}$. These trace functions converge for a much larger class of $V$-modules $W$. Indeed in Section \ref{sec:stable.rat} below we introduce the notion of a \emph{stable} $V$-module and a corresponding notion of \emph{stable rationality}.
\begin{thm}\label{thm:intro.2}
Let $(V, \omega, h)$ be as in Theorem~\ref{thm:intro.1} and assume furthermore that $V$ is stably rational. Let $W^1,\ldots,W^k$ be the complete set of irreducible stable $V$-modules up to isomorphism. Then the trace functions $S_{W^i}(u,\alpha,\tau)$ converge on the domain
\[
0<-N\im(\alpha)<\im(\tau),
\]
span the space $\CC((E,\CL)_{(\alpha,\tau)},V)$ of charged conformal blocks at each point of this domain, and are flat sections of $\nabla$.
\end{thm}
The restriction on the domain of convergence is quite natural and reflects the fact that the trace functions generally possess poles along the lattice $N\alpha\in\Z+\Z\tau$. Nevertheless it is easy to see that every point in the complement of this lattice is equivalent under the action of the Jacobi group on $B$ to a point in the domain of convergence, see Lemma \ref{lem:G.action}.

Our third result establishes the Jacobi invariance of charged trace functions.
\begin{thm}\label{thm:intro.3}
Let $(V, \omega, h)$ be as in Theorem~\ref{thm:intro.2}. The connection $\nabla$ is equivariant with respect to the natural action of the Jacobi group $SL_2(\Z)\ltimes\Z^2$ on $B$. More precisely the sheaf of flat sections of $\nabla$, and in particular the span of the trace functions $S_{W^i}$, is invariant under the Jacobi transformations
\begin{align}\label{eq:Jacobi.xform.intro}
\begin{split}
[S \cdot A]\left(u, \frac{\al}{c\tau+d}, \frac{a\tau+b}{c\tau+d} \right) &= e^{2\pi i \kappa \frac{c}{c\tau+d} \al^2} S( e^{2\pi i\frac{c\al}{c\tau+d} h_{[1]}}(c\tau+d)^{L_{[0]}} u, \al, \tau), \\
[S \cdot (m, n)](u, \al + m\tau+n, \tau) &= e^{ -2\pi i \kappa (m^2 \tau + 2m \al) } S(e^{-2\pi i m h_{[1]}} u, \al, \tau),
\end{split}
\end{align}
where the index $\kappa$ is defined by $h_{[1]}h = 2\kappa \vac$.
\end{thm}
A few words about the proof of Theorem~\ref{thm:intro.2}. Using the integrability of the connection $\nabla$ we prove an isospectrality result in Section \ref{sec:holonomic}, which in turn guarantees that flat sections of $\nabla$ possess series expansions in $y$ and $q$ of a particular form. Then in Section \ref{sec:stable.rat} we introduce a charged analogue of the Zhu algebra $A(V)$, which we call the \emph{stable Zhu algebra} $A^{\stab}(V)$. This algebra plays for stable modules the same role that $A(V)$ plays for positive energy modules. Assuming semisimplicity of $A^{\stab}(V)$, we show that charged conformal blocks are exhausted by trace functions.

Our results apply in particular to admissible affine vertex algebras {and admissible $W$-algebras associated with nilpotent elements of standard Levi type.}
Let $\g$ be a simple Lie algebra and $L_k(\g)$ the simple affine vertex algebra at admissible level $k$. Taking for $h$ the Weyl vector $\rho \in \g = V_1$, {one finds that the stable Zhu algebra $A^{\mathrm{stab}}(L_k(\g))$ is a finite dimensional semisimple quotient of the polynomial algebra $\C[\mf{h}]$ on the Cartan subalgebra $\mf{h}$ of $\g$.} In particular $(L_k(\g),\omega,\rho)$ is stably quasi-lisse and stably rational. Hence the corresponding trace functions span the space of charged conformal blocks. Restricting to $u = \vac$ we recover the fact, which follows easily from results of Kac and Wakimoto \cite{KW89}, that the graded dimensions of admissible $\what\g$-modules of admissible level $k$ constitute a vector valued Jacobi form of index $\kappa = kh^\vee\dim(\g) / 24$. {In the case of the universal $W$-algebra $\W^k(\g,f)$, whose Zhu algebra is the finite $W$-algebra $U(\g,f)$, the stable Zhu algebra $A^{\mathrm{stab}}(\W^k(\g,f))$, for an appropriately chosen current $h$, is isomorphic to the finite $W$-algebra $U(\mf{l},f)$ associated with the minimal Levi subalgebra $\mf{l}$ containing $f$, and the stable Zhu algebra $A^{\mathrm{stab}}(\W_k(\g,f))$ of the simple quotient $\W_k(\g,f)$ of $\W^k(\g,f)$ is a quotient of $U(\mf{l},f)$. Similarly, $X^h_{W_k(\g,f)}$ is the closed subvariety of the Slodowy slice $f+\mf{l}^e$ at $f$ in $\mf{l}$. We thus show that all quasi-lisse $W$-algebras are stably quasi-lisse, and all quasi-lisse admissible $W$-algebras of standard Levi type are stably rational as well, see Theorem \ref{thm:stably-quasi-lisse-W} and  Theorem \ref{thm:W-algebra} below. }
This represents a large extension of {Zhu's modularity results  \cite{Z96} to quasi-lisse vertex algebras}.

The appearance of modular and Jacobi forms in representation theory has a long history. In \cite{KP84} Kac and Peterson determined modular properties of characters of integrable representations of affine Kac-Moody algebras by expressing them in terms of theta functions, and deduced consequences such as asymptotic growth of representations. In \cite{FF84} Feigin and Fuchs developed analogous character formulas for minimal representations of the Virasoro algebra. Further impetus came from the role of the moonshine module vertex algebra $V^\natural$ in monstrous moonshine \cite{FLM,B92}.

These phenomena are naturally explained by Zhu's theorem. Furthermore Zhu's theorem plays an important role in conformal field theory, and the resulting representation of $SL_2(\Z)$ on the space $\CC^\nabla(V,\omega)$ of global flat sections, appears in Huang's proof of the Verlinde formula \cite{Huang2008}. Modularity also has important structural applications for vertex algebras. For example, \cite{EMS-dimfla,MS23} use properties of weight $2$ cusp forms, including Deligne's bound on Fourier coefficients \cite{Deligne}, to obtain estimates on the weight-one spaces of orbifolds of vertex algebras. These estimates were used in \cite{ELMS21} to show that every holomorphic vertex algebra of central charge $24$ with nontrivial weight-one space is an orbifold of the Leech lattice vertex algebra.

Jacobi invariance is similarly important. Schellekens used it to constrain the weight-one Lie algebra of a holomorphic conformal vertex algebra of central charge $24$, leading to his list of $71$ possibilities \cite{Schell} \cite[Section 6]{EMS20}; see also \cite{DM04}. For lisse rational vertex algebras the Jacobi property was established in \cite{Miyamoto-theta,KM2012} by expanding \eqref{eq:trace.def.intro} as a Taylor series in $\alpha$, identifying the coefficients with $n$-point functions, and applying Zhu's results. Closer in spirit to the present work is \cite{EH2016}, where supertrace functions of topological $N=2$ vertex algebras were shown to be flat sections of superconformal blocks, see also \cite{GK2009}. The lisse condition is assumed in all these works.

It has long been expected that analogues of Zhu's theorem should hold under somewhat weaker hypotheses. Indeed the space $\CC(E_\tau,V)$ of ordinary conformal blocks is finite dimensional not only for lisse vertex algebras, but for all quasi-lisse vertex algebras \cite{Arakawa-Kawasetsu}. This class includes the simple affine vertex algebras $L_k(\g)$ at admissible levels \cite{Ara09b}. Although the representation category of admissible $L_k(\g)$ is very large, it was shown in \cite{AM95} for $\g=\sll_2$ and in \cite{A12-2} in general that the $L_k(\g)$-modules in category $\OO$ for $\what\g$ are precisely the admissible modules of Kac and Wakimoto, who had shown that the characters of the latter are modular invariant \cite{KW88,KW89}.

The connection $\nabla$ on genus-one conformal blocks gives rise to modular linear differential equations (MLDE) satisfied by vertex algebra characters. These equations have been studied since the early days of conformal field theory \cite{MMS88} \cite{EO88}, and later in the context of vertex algebras, providing useful constraints, even without assumptions of rationality \cite{KNS2013, Arakawa-Kawasetsu, Li23, LLY2026}.

The 4d/2d duality of \cite{Beem.et.al} relating $4$-dimensional $N=2$ SUSY theories with vertex algebras provides further motivation. Modular properties of Schur indices were observed in \cite{R12}, and in view of the 4d/2d duality, would be natural consequences of modularity for the corresponding vertex algebras. The associated MLDEs have been investigated in detail in \cite{BR}, and flavoured analogues have been studied in \cite{PW2023}. Here the inclusion of ``flavour'' roughly corresponds to the passage from trace functions of the form \eqref{eq:intro-zhu-trace} to that of \eqref{eq:trace.def.intro}. In particular, the vertex algebras associated with $\mathcal S$-class theories \cite{Beem.et.al,BR,ClassS} in genus $0$ have been constructed in \cite{Arakawa.S.class} as chiral quantisations of Moore-Tachikawa symplectic varieties \cite{BFN}. These vertex algebras are known to be quasi-lisse in type $A$ and believed to be so in general, and their characters are conjecturally quasi-modular and have interesting relations to multiple $q$-zeta values \cite{Milas.MZV}. Many non-admissible affine vertex algebras are also quasi-lisse \cite{AM.Joseph}.

Thus motivated, it is reasonable to seek a description of conformal blocks in terms of trace functions on appropriate modules for the whole class of quasi-lisse vertex algebras. By including moduli of line bundles in our approach, trace functions on a sufficiently large category of modules are rendered convergent and, under a reasonable rationality condition, span the corresponding spaces of conformal blocks.

There are many generalisations of Zhu's theorem in the literature. Dong, Li and Mason proved a generalisation to twisted modules \cite{DLM00}, a result fundamental to orbifold constructions of holomorphic vertex algebras \cite{LS.framed,Miya.Z3,EMS20,Jethro.expos}. Generalisations to vertex superalgebras with ``good'' and ``bad'' statistics were obtained in \cite{DZ05} and \cite{DZ10}, respectively. A version for vertex superalgebras graded by rational conformal weights was established in \cite{VE13}, and used in \cite{AE2019} to lift the Kac--Wakimoto modular transformation formulas for admissible affine vertex algebras from characters to trace functions; see also \cite{DLMadm,MTZ08}. Modularity of trace functions involving intertwining operators has been applied by Huang to rigidity of tensor categories \cite{Huang.DE,Huang-rigidity}. There has also been substantial interest in modularity for lisse but non-rational vertex algebras \cite{Miyamoto-C2}. Miyamoto showed that modularity is restored by adjoining suitable ``pseudotrace functions'', and the analysis of these functions was a key ingredient in McRae's proof that semisimplicity of Zhu's algebra implies rationality for lisse vertex algebras \cite{McRae2026,CM2016}. We note that while relatively few non-rational lisse vertex algebras are known \cite{Kausch-symplectic,AM.triplet,Sugimoto}, there is by contrast a rich supply of non-rational quasi-lisse examples \cite{Arakawa.S.class,AM.Joseph,AM.U-class}. We expect extensions of our results in all these directions to be possible.

Finally, our results describe charged conformal blocks in terms of stable $V$-modules, characterised among weight modules by a support condition. Other categories of modules, such as spectral flows of stable modules, arise by modifying this condition. It would be interesting to relate charged conformal blocks associated with these different categories, perhaps in connection with the proposal of Creutzig and Ridout to recover Verlinde-type formulas from distributional characters \cite{CR2013}, see also \cite{Ridout.sl3mod, Ridout.BPmod}.

\emph{Acknowledgements:} The authors would like to thank R. Heluani, Y.-Z. Huang, V. Kac and H. Movasati for discussions. JvE gratefully acknowledges RIMS in Kyoto for excellent working conditions during his stay as visiting professor in 2024, during which time work on this project began.
TA was partially supported  by JSPS KAKENHI Grant Numbers 21H04993, 25K21659 and 26H01997.
HL was supported by JSPS KAKENHI Grant Number 21H04993.
JvE was supported by grants Serra-2023-0001, CNPq 306498/2023-5 and FAPERJ 204.322/2025.

\section{Elliptic curves, line bundles, and modular forms}\label{sec:mod.ell}

In this section we recall background material on modular forms and elliptic functions, and line bundles over elliptic curves. Throughout the paper $\CH$ denotes the upper half complex plane, and we reserve the symbols $\tau \in \CH$ for a complex variable and $q = e^{2\pi i \tau}$, and $\al \in \C$ and $y = e^{2\pi i \al}$.

\subsection{Modular forms and elliptic functions}\label{sec:modular.functions}


For each integer $k \geq 1$ the Eisenstein series is defined as
\begin{align}\label{eq:Eisenstein.def}
\wtil{G}_{2k}(q) = (2\pi i)^{2k} \left[ -\frac{B_{2k}}{(2k)!} + \frac{2}{(2k-1)!} \sum_{n=1}^\infty n^{2k-1} \frac{q^n}{1-q^n} \right],
\end{align}
where the $B_{2k}$ are the Bernoulli numbers: $B_2 = 1/6$, $B_4 = -1/30$, $B_6 = 1/42$, etc. This series is convergent for $|q|<1$, and we set $G_{2k}(\tau) = \wtil{G}_{2k}(q)$ for $\tau \in \CH$. The holomorphic function $G_{2k}(\tau)$ is also referred to as an Eisenstein series.

For all $k \geq 1$ the Eisenstein series satisfies
\begin{align*}
G_{2k}\left(\frac{a\tau+b}{c\tau+d}\right) = (c\tau+d)^{2k} G_{2k}(\tau) - 2\pi i c (c\tau+d) \delta_{k,1} \qquad \text{for all $\left(\begin{array}{cc}a&b\\c&d\\  \end{array}
\right) \in SL_2(\Z)$,}
\end{align*}
i.e., $G_{2k}(\tau)$ is a modular form of weight $2k$ for $k \geq 2$. The Eisenstein series $G_2(\tau)$ is sometimes referred to as a quasimodular form.

For each integer $k \geq 1$ we consider the Laurent series
\begin{align}\label{eq:wp.def}
\wtil\wp_k(z, q) = z^{-k} + \pi i \delta_{k,1} + (-1)^k \sum_{n=0}^\infty \binom{2n+1}{k-1} \wtil{G}_{2n+2}(q) z^{2n+2-k}.
\end{align}
At $q = e^{2\pi i \tau}$ the series is convergent for $0 < |z| < r$ where $r$ is the minimal norm of a nonzero point in the lattice $\Lambda = \Z + \Z \tau \subset \C$. In this domain the limit coincides with a meromorphic function, defined on all of $\C$ with poles on $\La$, which we denote $\wp_k(z, \tau)$.

We remark that \eqref{eq:wp.def} differs slightly from the convention used by Zhu in {\cite[equation (3.6)]{Z96}}. Indeed $\wp_1(z, \tau) = \wp_1^{\text{Zhu}}(z, \tau) + \pi i - G_2(\tau)z$ and $\wp_2(z, \tau) = \wp_2^{\text{Zhu}}(z, \tau) + G_2(\tau)$, while $\wp_k(z, \tau) = \wp_k^{\text{Zhu}}(z, \tau)$ for $k \geq 3$.

We consider, for each integer $k \geq 1$, the series
\begin{align}\label{eq:zhu.3.9}
P^+_k(w, q) = \frac{(2\pi i)^k}{(k-1)!} \sum_{n=1}^\infty \left[ \frac{n^{k-1} w^n}{1-q^n} + (-1)^k \frac{n^{k-1} w^{-n}q^n}{1-q^n} \right],
\end{align}
convergent in the domain $|q| < |w| < 1$. Notice that $2\pi i w \partial_w P^+_k(w, q) = k P^+_{k+1}(w, q)$. This is precisely the series $P_k(w,q)$ introduced in {\cite[equation (3.9)]{Z96}}.

If we denote the limit of $P^+_k(w, q)$ also by the same symbol, then we have $P^+_k(e^{2\pi i z}, e^{2\pi i \tau}) = (-1)^k {\wp}_k(z, \tau)$ {\cite[equations (3.11)--(3.13)]{Z96}}. Let us be careful to note, however, that substitution of $w = e^{2\pi i z}$ into the formal series defining $P^+_k(w, q)$ does not yield the formal series defining $(-1)^k \wtil\wp_k(z, q)$. Indeed \eqref{eq:zhu.3.9} involves the summation $w + w^2 + w^3 + w^4 + \cdots$, and the substitution $w = e^{2\pi i z}$ here is ill defined.

It will also be convenient for us to have the following variant:
\begin{align}\label{eq:zhu.3.9.minus}
P^-_k(w, q) = \frac{(2\pi i)^k}{(k-1)!} \sum_{n=1}^\infty \left[ \frac{n^{k-1} w^n q^n}{1-q^n} + (-1)^k \frac{n^{k-1} w^{-n}}{1-q^n} \right],
\end{align}
a series in $\C((w^{-1}))[[q]]$ convergent in the domain $|q| < |w^{-1}| < 1$. As before we have $2\pi i w \partial_w P^-_k(w, q) = k P^-_{k+1}(w, q)$, and $P^-_k(e^{2\pi i z}, e^{2\pi i \tau}) = (-1)^k {\wp}_k(z, \tau)$ when the left hand side converges.

We recall the theta function {\cite[p. 17]{Mumford.Tata.1}}
\begin{align*}
\theta(z, \tau) = \theta_{11}(z, \tau) = i \sum_{n \in \Z} (-1)^n q^{(n+1/2)^2/2} e^{2\pi i (n+1/2) z},
\end{align*}
which is holomorphic on $z \in \C$ and vanishes only at $z \in \Lambda$. It satisfies the relations (see the table at {\cite[p. 19]{Mumford.Tata.1}})
\begin{align}\label{eq:theta.xforms}
\theta(z+1, \tau) &= -\theta(z, \tau), \\
\theta(z+\tau, \tau) &= -q^{-1/2} e^{-2\pi i z} \theta(z, \tau).
\end{align}
It is known that $\theta'(0, \tau) = -2\pi \eta(\tau)^3$, where $\theta'(z, \tau)$ denotes the derivative in $z$ of $\theta$, and $\eta(\tau)$ is the Dedekind eta function
\[
\eta(\tau) = q^{1/24} \prod_{n=1}^\infty (1-q^n).
\]

Now we define
\begin{align}\label{eq:psi.def}
\psi(z, \al, \tau) = \frac{\theta'(0, \tau)\theta(z+\alpha, \tau)}{\theta(z, \tau)\theta(\alpha, \tau)}
\end{align}
(see {\cite[A.9]{EH2016}}) which we see satisfies
\begin{align}\label{eq:psi.transf}
\psi(z+1, \al, \tau) &= \psi(z, \al, \tau), \\
\psi(z+\tau, \al, \tau) &= e^{-2\pi i\alpha} \psi(z, \al, \tau).
\end{align}

We now consider the Laurent series expansion of $\psi(z, \al, \tau)$ in powers of $z$. We define
\begin{align}\label{eq:Ppm.def}
\begin{split}
P_1^+(z, y, q) &= \sum_{n \geq 0} y^{-n} - \frac{e^{2\pi i z}}{e^{2\pi i z}-1} + \sum_{n = 1}^\infty \left( \frac{y^{-1}q^n}{1-y^{-1}q^n} e^{2\pi i n z} - \frac{y q^n}{1-yq^n} e^{-2\pi i n z} \right), \\
P_1^-(z, y, q) &= -\sum_{n \geq 1} y^{n} - \frac{e^{2\pi i z}}{e^{2\pi i z}-1} + \sum_{n = 1}^\infty \left( \frac{y^{-1}q^n}{1-y^{-1}q^n} e^{2\pi i n z} - \frac{y q^n}{1-yq^n} e^{-2\pi i n z} \right).
\end{split}
\end{align}
These series converge in the region $|q|<|y^{\mp 1}|<1$. The only difference between $P_1^+$ and $P_1^-$ is the first summation, and in both cases the summation is an expansion of the same rational function $y/(y-1)$. Now \cite[Proposition 5]{MTZ08} asserts that
\[
\psi(z; -\al, \tau) = -2\pi i P_1^{\pm}(z; e^{2\pi i \al}, e^{2\pi i \tau}).
\]
On the other hand we write
\begin{equation}\label{eq:P1pm-Laurent}
P_1^\pm(z,y,q)
=
-\frac{1}{2\pi i z}
+
\sum_{m= 1}^\infty E_m^\pm(y,q)\,(2\pi i z)^{m-1}.
\end{equation}
The series $P_1^\pm$ differ only in the constant coefficient, so we have
\[
E_m(y,q) := E_m^+(y,q) = E_m^-(y,q) \in \C[y^{\pm 1}][[q]],\qquad m\ge 2,
\]
while
\[
E_1^\pm(y,q) = -\frac{1}{2\pi i}P_1^\pm(y,q)-\frac{1}{2}.
\]

A more general class of \emph{twisted Weierstrass functions} has been considered in \cite{MTZ08} (see also \cite{DLM00})
\[
P_k\bracsmall{\theta}{\phi}(z, \tau)
= \frac{(-2\pi i)^k}{(k-1)!}
\sum_{n \in \mathbb{Z} + \lambda}\!\!\!{}'\,\,\, \frac{n^{k-1} e^{n z}}{1 - \theta^{-1} q^n},
\]
where $\la \in \R$ is chosen so $e^{2\pi i \la} = \phi$, and the prime means omit the summand $n=0$ if $(\theta, \phi) = (1, 1)$. Then, our $P_1^\pm(z,y, q)$ are particular expansions of $P_1\bracsmall{y}{1}(z, \tau)$. These have series expansions
\begin{align*}
P_k\bracsmall{\theta}{\phi}(z, \tau) = z^{-k} + (-1)^k \sum_{m=k}^\infty \binom{m-1}{k-1} G_m\bracsmall{\theta}{\phi}(\tau) z^{m-k},
\end{align*}
in terms of \emph{twisted Eisenstein series} $G_m\bracsmall{\theta}{\phi}(\tau)$. The functions $E_m$ above are now recovered as $E_m(y,q) = G_m\bracsmall{y}{1}(\tau)$. The results of {\cite[Section 5.2]{Li23}} assert that the ring generated by twisted Eisenstein series is Noetherian. In Section \ref{sec:finiteness} below we will require the following special case:
\begin{prop}\label{prop:Li23}
The ring $\CE \subset \C[y^{\pm 1}][[q]]$ generated by $E_k(y,q)$ for $k \geq 2$ is Noetherian.
\end{prop}
In fact, according to \cite[Proposition 2.10]{Libgober} the functions $E_m^\pm(y,q)$ are polynomials in the functions $P_m^\pm(y,q)$, $m \leq 3$.

\subsection{Elliptic curves and line bundles}\label{sec:curves.bundles}

{{This section establishes the connection between the analytic functions in Section~\ref{sec:modular.functions} and the algebro-geometric framework of elliptic curves and degree-zero line bundles, which form the foundation for charged conformal blocks.}}

We introduce a family $T \rightarrow B$ of pairs $(E, \CL)$ consisting of a smooth elliptic curve $E$ and a holomorphic line bundle $\CL$ of degree $0$ over $E$. The base of the family will be $B = \C \times \CH$, points of which we will denote $(\alpha, \tau)$. We now describe the fibres. We consider the following group action of $\Z^2$ on $B \times \C^2$:
\begin{align}\label{eq:elliptic.family}
(m, n) \cdot (\al, \tau, v, z) = (\al, \tau, e^{-2\pi i m \alpha} v, z + m\tau + n).
\end{align}
The quotient $T = B \times \C^2 / \Z^2$ is the total space of our family, and in the fibre over the point $(\al, \tau)$ the projection $\CL \rightarrow E$ is $(v, z) \mapsto z$.

As above we write $q = e^{2\pi i \tau}$, let us here and throughout also write $y = e^{2\pi i \alpha}$. The fibre $(E, \CL)_{(\al, \tau)}$ can also be presented as the quotient of $\C \times \C^\times$ by the following action of $\Z$:
\[
m \cdot (v, w) = (y^{-m} v, q^m w).
\]
\begin{rem}
The parameter space $(\al, \tau) \in \C \times \CH$ carries an action of the Jacobi group $SL_2(\Z) \ltimes \Z^2$, pairs $(E_1, \CL_1)$, $(E_2, \CL_2)$ identified under this action are isomorphic, and the isomorphisms can be given in general by explicit formulas (see Section \ref{sec:modularity} below and {\cite[Proposition 3.2]{E2017}}). In particular, if we write $\CL_\alpha$ for the line bundle associated with the parameter $\al$, then $\CL_{\alpha + m\tau + n} \cong \CL_\al$.
\end{rem}

We now describe meromorphic sections of $\CL_\al$ and relate them to the functions introduced in Section \ref{sec:modular.functions} above.

Let $X$ be a smooth complex curve of any genus $g$. The exponential exact sequence of sheaves
\[
0 \rightarrow 2\pi i \underline{\Z} \rightarrow \OO_X \rightarrow \OO_X^\times \rightarrow 0,
\]
yields a long exact sequence in cohomology and in particular
\[
\Pic(X) = H^1(X, \OO_X^\times) \rightarrow H^2(X, \underline{\Z}) \cong \Z
\]
takes a line bundle $\CL \in \Pic(X)$ to its degree. The kernel of the degree map is isomorphic to the quotient
\[
\Jac(X) = H^1(X, \OO_X) / H^1(X, \underline{\Z}) \cong \C^g / \Z^g
\]
i.e., a complex $g$-dimensional torus. The other characterisation of degree is in terms of meromorphic sections. Namely a section $\sigma$ of $\CL$ defines a divisor $\sum_{p \in X} s_p [p]$ where $s_p$ is the degree of vanishing of $\sigma$ at $p$, and $\deg(\CL)$ is $\sum_{p \in X} s_p$. From this description of the degree, it follows that for $\deg(\CL) < 0$ one has $H^0(X, \CL) = 0$.

Let $\CL$ be a line bundle on $X$ of degree $d$. We write $h^0(X, \CL)$ for the dimension of $H^0(X, \CL)$ the vector space of global sections. The Riemann-Roch theorem asserts that
\[
h^0(X, \CL) - h^0(X, \CL^{-1} \otimes \omega_X) = d + 1 - g,
\]
where $\omega_X$ is the canonical bundle, i.e., the bundle of $1$-forms since $X$ is a curve. For a point $p \in X$ and $k \in \Z$ we can consider the line bundle $\OO_X(k p)$ associated to the divisor $k p$, characterised by its space $H^0(X, \OO_X(kp))$ of global sections being the set of meromorphic functions on $X$, regular outside $p$, and with pole at $p$ of order not greater than $k$. The degree of $\OO_X(kp)$ is $k$. We can similarly form $\CL(kp) \cong \CL \otimes \OO_X(kp)$. Note also that $\CL(kp)^{-1} \cong \CL^{-1}(-kp)$.

If $E$ is a curve of genus $1$ and $\CL$ a line bundle on $E$, then $\omega_E \cong \OO_E$ and we obtain
\[
h^0(E, \CL(kp)) - h^0(E, \CL^{-1}(-kp)) = k.
\]
For $k > 0$ the second term on the LHS vanishes so $h^0(E, \CL(kp)) = k$. For $k=0$ all we can conclude is that $h^0(E, \CL) = h^0(E, \CL^{-1})$. But if $\CL = \OO_E$ we know both sides here are $1$. If $\CL$ is nontrivial then, since it has degree $0$, any nonzero global meromorphic section of $\CL$ must have as many poles as it has zeroes. So if such a nonzero section has no poles, it has no zeroes either, which is impossible for $\CL$ nontrivial.

In short, for a smooth elliptic curve $E$ the dimensions $h^0(E, \CL(kp))$ are given for $k \geq 0$ by
\begin{align*}
1, 1, 2, 3, 4, \ldots \quad \text{if $\CL \cong \OO_E$} \\
0, 1, 2, 3, 4, \ldots \quad \text{if $\CL \not\cong \OO_E$}.
\end{align*}
These spaces of sections are described in terms of Weierstrass elliptic functions and theta functions. For $\CL = \OO_E$ the sum of all $H^0(E, \CL(kp))$ is spanned by the functions $1$ and $\wp_k(z, \tau)$ for $k \geq 2$. Let $(E, \CL)$ be the elliptic curve and line bundle corresponding to $(\al, \tau)$ with $\al \notin \La$ now. If we define
\begin{align}\label{eq:varphi.def}
\varphi(z) = \frac{\theta(z-\alpha, \tau)}{\theta(z, \tau)},
\end{align}
then from \eqref{eq:theta.xforms} we see that
\begin{align*}
\varphi(z+1) &= \varphi(z), \\
\varphi(z+\tau) &= e^{2\pi i\alpha} \varphi(z).
\end{align*}
Hence $\varphi(z) v$ defines a meromorphic section of the line bundle $\CL_\alpha$ with a simple pole at $z=0$ (and a simple zero at $z=\alpha$).

Note that the function $\psi(z, -\al, \tau)$, defined in \eqref{eq:psi.def} above, is just $\varphi(z)$ up to a normalisation which depends on $\al$ and $\tau$ but not $z$. The normalisation $\psi$ is more convenient for our purposes than $\varphi$, because the coefficients $E_m(y, q)$ of its expansion \eqref{eq:P1pm-Laurent} lie in a Noetherian ring, Proposition {\ref{prop:Li23}}.

\section{Vertex algebras and chiral algebras}

We refer to \cite{KacVA} for background on the theory of vertex algebras. For the material on chiral algebras constructed from vertex algebras via associated bundle constructions we follow \cite[Chapters 6 and 7]{FBZ}. The definitive reference on chiral algebras is \cite{BD}.

{{This section has two parts. Section~\ref{sec:vertex} fixes notation and conventions for vertex algebras and conformal vertex algebras used throughout the paper. Section~\ref{sec:chiral} recalls the associated-bundle construction of chiral algebras and discusses the notion of charged conformal block on curves, providing the geometric framework used in later sections.}}

We now recall a few pieces of notation. For a Laurent series $f(z) = \sum_n f_n z^n \in \C(\!(z)\!)$ the formal residue $\res_z f(z)$ is just the coefficient $f_{-1}$. If $\rho(w) = \sum_{n=1}^{\infty}a_nw^n\in \mathbb{C}[\![w]\!]$ is an invertible power series, i.e.,  $a_1\neq 0$, the following formal change of coordinates formula holds:
\begin{align}\label{change_of_variable}
\res_{z}f(z) = \res_{w}\left( f(\rho(w)) \rho'(w) \right).
\end{align}
One also has the formal integration by parts formula:
\begin{align*}
\res_z (f(z) g'(z)) = -\res_z(f'(z) g(z)).
\end{align*}
When we write $(1+t)^n$ for $n\in \Z$ in general, we mean
\begin{align}\label{binom}
(1+t)^n=\sum_{j\geq 0}\binom{n}{j} t^j, \quad \text{where} \quad \binom{n}{j}=\frac{n(n-1)\cdots (n-j+1)}{j!}.
\end{align}

\subsection{Vertex algebras}\label{sec:vertex}

\begin{defn}
A \emph{vertex algebra} consists of a vector space $V$, a vacuum vector $|0\rangle\in V$, a linear map
\begin{align*}
V \rightarrow \text{End}(V)[\![z, z^{-1}]\!], \quad
a \mapsto Y(a,z)=\sum_{n\in \ZZ}a_{(n)}z^{-n-1}
\end{align*}
(so $a_{(n)}b = \res_z z^n Y(a,z)b$ for all $a, b \in V$ and $n \in \Z$), satisfying the following axioms:
\begin{itemize}
\item{(Quantum field property)} For all $a, b \in V$ we have $Y(a,z)b \in V(\!(z)\!)$.
\item{(Vacuum axiom)} $Y(|0\rangle,z) = \Id_V$, and for all $a \in V$ one has $Y(a, z)\vac \in V[\![z]\!]$ and $a_{(-1)}|0\rangle=a$;
\item{(Borcherds identity)}
\begin{align}
\label{eq:Borcherds_identity}
\begin{split}
\sum_{j\geq 0}&\binom{m}{j}(a_{(n+j)}b)_{(m+k-j)}c\\ &=\sum_{j\geq 0}(-1)^j\binom{n}{j}a_{(m+n-j)}(b_{(k+j)}c)-\sum_{j\geq 0}(-1)^{j+n}\binom{n}{j}b_{(n+k-j)}(a_{(m+j)}c)
\end{split}
\end{align}
for all $a,b,c \in V$, and $k,m,n\in \ZZ$.
\end{itemize}
\end{defn}
The \emph{translation operator} $T\in \text{End}(V)$ is defined by $T(a)=a_{(-2)}|0\rangle$ for all $a \in V$, and \eqref{eq:Borcherds_identity} implies that $[T, Y(a, z)] = Y(Ta,z)=\partial_zY(a,z)$, for all $a\in V$. A useful special case of the Borcherds identity is the \emph{commutator formula}
\begin{align}\label{eq:va.comm.fla}
[a_{(m)}, b_{(n)}] = \sum_{j \geq 0} \binom{m}{j} (a_{(j)}b)_{(m+n-j)}.
\end{align}

\begin{defn}
A \emph{conformal vector} in a vertex algebra $V$ is a vector $\omega\in V$ with $Y(\omega,z)=\sum_{n\in \ZZ}L_{n}z^{-n-2}$, such that $L_{-1}=T$, $L_{0}$ is diagonalizable on $V$, and $\{L_{n}\}_{n\in \ZZ}$ satisfies the Virasoro relations:
\begin{align}\label{eq:vir.comm}
[L_{n},L_{m}]=(n-m)L_{n+m}+\frac{n^3-n}{12}\delta_{n,-m}c \Id_V,
\end{align}
for some constant $c\in \C$. The invariant $c$ is called the \emph{central charge} of $(V, \omega)$. A \emph{conformal vertex algebra} is a vertex algebra endowed with a {conformal vector} for which the eigenvalues of $L_0$ are integers bounded below.
\end{defn}
If $a \in V$ is an eigenvector of $L_0$ then we refer to the eigenvalue as its conformal weight and we denote it by $\Delta_a$. Given $a,b\in V$ and $n\in \Z$, we have
\[
\Delta_{a_{(n)}b}=\Delta_{a}+\Delta_{b}-n-1.
\]
We denote by $V_\D$ the vector subspace of vectors of conformal weight $\D$. It is useful to introduce the following notation for $a\in V_{\Delta}$:
\begin{align}\label{eq:conf.modes}
Y(a,z)=\sum_{n\in \ZZ} a_nz^{-n-\Delta}.
\end{align}
In particular we have $a_{(n)}=a_{n-\Delta+1}$, and $a_0 : V_{\D} \rightarrow V_\D$ for all $a \in V$ and $\D \in \Z$.

A \emph{current} is a vector $h \in V_1$ of conformal weight $1$, such that $h_0 = h_{(0)}$ acts semisimply on $V$ and
\begin{align*}
h_{n}h = 0 \quad \text{for $n \geq 2$} \quad \text{and} \quad h_1h = 2\kappa  \vac
\end{align*}
for some constant $\kappa \in \C$ which we call the \emph{index}. In OPE notation
\[
h(z) h(w) \sim \frac{2\kappa \vac}{(z-w)^2}.
\]
We call the eigenvalues of $h_{0}$ the \emph{charge}; we denote by $V^{(k)}$ the space of vectors in $V$ with charge $k$.

Let $\rho(t) = \sum_{n=1}^\infty \rho_n t^n = \exp(\sum_{n \geq 1} a_n t^{n+1} \partial_t) (a_1)^{t\partial_t} t$ be an invertible power series, i.e., $\rho_1 = a_0 \neq 0$. Huang established the following formula {\cite[Section 7.4]{Huang.Book.cft}}
\begin{align}\label{eq:huang.fla}
Y(a, z) = R(\rho) Y(R(\rho_z)^{-1}a, \rho(z)) R(\rho)^{-1},
\end{align}
where $\rho_z$ is the series defined by $\rho_z(t) = \rho(z+t)-\rho(z)$, and $R(\rho) \in \en(V)$ is a linear automorphism defined by the following procedure:
\begin{align}\label{eq:R(rho).def}
R(\rho) = \exp\left( -\sum_{n \geq 1} a_n L_n \right) a_0^{-L_0}.
\end{align}
This formula for $R(\rho)$ makes sense because $L_0$ acts on $V$ with integer eigenvalues, and the expression is well-defined because each $L_{>0}$ acts locally nilpotently.

We recall Zhu's modified vertex algebra structure, which is useful in formulating the modular transformation properties of trace functions. For all $a \in V$ one defines \cite{Z96}
\begin{align*}
Y[a,z] = Y(e^{2\pi i z L_0} a, \phi(z)),
\end{align*}
where $\phi(t) = e^{2\pi i t}-1$. If we write $R = R(\phi)$ then $R(\phi_z) = R(\phi) e^{-2\pi i z L_0}$ since $\phi_z(t) = e^{2\pi i z}\phi(t)$. Now we have
\begin{align*}
Y[a, z] = R^{-1} Y(Ra, z) R
\end{align*}
because of \eqref{eq:huang.fla}. So in fact $R : (V, Y[-,z]) \rightarrow (V, Y(-,z))$ is an isomorphism of vertex algebras.

The vertex algebra $(V, Y[-,z])$ is conformal with conformal vector $\wtil{\omega} = R^{-1}\omega = (2\pi i)^2(\omega - \tfrac{c}{24}\vac)$.  We write $L[z] = Y[\wtil{\omega}, z]$. One also checks that $R^{-1}(h) = \wtil{h} = 2\pi i (h + \frac{1}{2}L_1 h)$. Let's assume $h$ is primary for now. We write $h[t] = Y[\wtil h, t] = \sum_n h_{[n]} t^{-n-1}$.

We note that $h_{[1]}$ is related to the notation $I(h)$ used in \cite{AE2019}. Specifically {\cite[equation (2.14)]{AE2019}}
\[
(2\pi i)^2 h_{([1])} = I(h) = \sum_{j \geq 1} \frac{(-1)^j}{j} h_j.
\]

\subsection{Associated chiral algebras}\label{sec:chiral}

In this section we follow closely the presentation in \cite{FBZ}. It is possible to view invertible power series as continuous automorphisms of the topological algebra $\OO = \C[[t]]$. Let $V$ be a conformal vertex algebra. Then the group $\aut(\OO)$ of all continuous automorphisms of $\OO$ has a representation on $V$ via the construction $\rho \mapsto R(\rho)$ defined in \eqref{eq:R(rho).def}.

If $X$ is a smooth complex curve and $p \in X$ a point, the local ring $\OO_p$ is naturally an $\aut(\OO)$-torsor, i.e., carries a simple transitive right $\aut(\OO)$-action. Informally elements of $\OO_p$ are choices of local coordinate in a neighbourhood of $p$, more precisely in the formal disc $D_p = \spec{\OO_p}$, and $\aut(\OO)$ is the group of all changes of coordinate. Through an associated bundle construction one obtains an infinite dimensional vector bundle $\CV$ over $X$ whose fibre over $p$ is $\OO_p \times_{\aut(\OO)} V$. The bundle $\CA = \CV \otimes \omega_X$ carries the structure of a chiral algebra \cite{BD}. We do not explain precisely what a chiral algebra is here, except to say that there is an action
\[
\mu_p : \Gamma(D_p^\times, \CA) \times \CV_p \rightarrow \CV_p,
\]
where $D_p^\times$ is the punctured formal disc, and that a choice of coordinate $z \in \OO_p$ at the point $p$ induces identifications $\CV_p \rightarrow V$ and $\Gamma(D_p^\times, \CA) \rightarrow V[\![z]\!]$ upon which the action becomes
\begin{align}\label{eq:chiral.action.explicit}
\mu_p(f(z) a, b) = \res_z f(z) Y(a, z)b.
\end{align}
The associated bundle construction and Huang's formula \eqref{eq:huang.fla} together guarantee that this action is well-defined, independently of the choice of coordinate $z$.

Within this context we have the definition of the space of conformal blocks on $X$ at $p$:
\begin{defn}
A conformal block is a linear functional $\Phi : \CV_p \rightarrow \C$ which annihilates the image $\Gamma(X\backslash p, \CA) \cdot \CV_p$. The vector space of conformal blocks is denoted $\CC(X, p, (V, \omega))$ or just $\CC(X, p, V)$.
\end{defn}
Although we will not require it in this work, the construction can be adapted to include multiple points $p$ on $X$ and the insertion of nontrivial $V$-modules at the points instead of $V$.

We now describe an enhancement of these constructions in which the geometric input consists of an algebraic curve $X$ as well as a holomorphic line bundle $\CL$ over $X$. We shall be brief, as the theory we require is treated in {\cite[Chapter 7]{FBZ}} in greater generality.

Let $\GG$ denote the algebraic group characterised by $\GG(R) = R^\times$. In particular $\GG(\OO)$ may be identified with formal power series $a_0 + a_1 t + \cdots$ in which $a_0 \in \C^\times$ and $a_n \in \C$ for $n \geq 1$. The group $\aut(\OO)$ of continuous automorphisms of $\OO$ acts on $\GG(\OO)$ and we may form the semidirect product $G = \aut(\OO) \ltimes \GG(\OO)$.

Let $\CL$ be a holomorphic line bundle over $X$ and $p \in X$, then $\GG(\OO)$ acts on the set of trivialisations of $\CL|_{D_p}$ via multiplication. This action combines with the action of $\aut(\OO)$ on $\OO_p$, so that the semidirect product $G$ acts on the set $\CP_p$ of pairs $(s, z)$ consisting of $z \in \OO_p$ a coordinate on $D_p$ and $s$ a trivialisation of the restriction $\CL|_{D_p}$. In this way $\CP_p$ acquires the structure of a $G$-torsor.

At the level of Lie algebras $\lie\aut(\OO)$ is spanned by $L_m = -z^{m+1} \partial_z$ for $m \geq 0$. These elements, of course, satisfy the commutator relations of the Virasoro algebra \eqref{eq:vir.comm}. The Lie algebra $\lie\GG(\OO)$ is abelian, spanned by $h_n$ for $n \geq 0$, where (after a choice of local coordinate $z$) $h_n$ corresponds to multiplication by $z^n$. In the semidirect product $\lie(G)$ we have the Lie bracket
\begin{align}\label{eq:L.h.rel}
[L_m, h_n] = -n h_{m+n}.
\end{align}

\begin{defn}
A \emph{charged conformal vertex algebra} is a triple $(V, \omega, h)$ in which $(V, \omega)$ is a conformal vertex algebra and $h \in V$ is a current, such that $L_1h = 0$ (thus $h$ is said to be quasi-primary), and $h_0$ acts on $V$ with integer eigenvalues. Since $L_0$ and $h_0$ commute, a charged conformal vertex algebra has a decomposition $V = \bigoplus_{\D, k} V_\D^{(k)}$ by conformal weight $\D$ and eigenvalue $k$ of $h_0$. For $a \in V_\D^{(k)}$ we refer to the integer $k$ as the \emph{charge} of $a$.
\end{defn}


Let $(V, \omega, h)$ be a charged conformal vertex algebra. Since $h$ is quasi-primary, the commutator formula \eqref{eq:va.comm.fla} gives the same relation \eqref{eq:L.h.rel} as above. Thus $V$ carries a representation of $\lie(G)$. Because of the integral charge condition, this representation exponentiates to a representation of $G$. After a choice of local coordinate $z$, the elements $e^a$ and $e^{a z}$ of $\GG(\OO)$ are mapped to $e^{a h_0}$ and $e^{a h_1}$ in $\en(V)$, respectively.

Let $(V, \omega, h)$ be a charged conformal vertex algebra, $X$ a smooth complex curve and $\CL$ a holomorphic line bundle over $X$. From these data we obtain, as above, a $G$-action on $V$ and the $G$-torsors $\CP_p$ for each $p \in X$. Through an associated bundle construction one obtains an infinite dimensional vector bundle $\CV$ over $X$ whose fibre over $p$ is $\CP_p \times_{G} V$. As before $\Gamma(D_p^\times, \CV \otimes \omega_X)$ acts on $\CV_p$, and each choice of local coordinate $z$ at $p$ and trivialising section $s$ of $\CL|_{D_p}$ yields an isomorphism which identifies the action with the residue \eqref{eq:chiral.action.explicit} {\cite[Theorem 7.1.6]{FBZ}}.

\begin{defn}\label{defn:charged.CB}
A charged conformal block is a linear functional $\Phi : \CV_p \rightarrow \C$ which annihilates the image $\Gamma(X \backslash p, \CA) \cdot \CV_p$. The vector space of charged conformal blocks is denoted $\CC((X, \CL), p, (V, \omega, h))$ or just $\CC((X, \CL), p, V)$.
\end{defn}

Through the formalism of Virasoro localisation, the spaces of conformal blocks assemble to form a quasicoherent sheaf, in fact a $\CD$-module, as we vary the moduli parameters of the curve $X$ \cite{BS88} \cite[Chapter 17]{FBZ}. In the following, we shall restrict attention to the case of $X = E_\tau$ smooth elliptic curves and write $\CC(E_\tau, V)$ generally in place of $\CC(E_\tau, 0, V)$. The connection on the sheaf of conformal blocks comes from an instance of the notion of \emph{Ward identity} in physics, and is interpreted as a Gauss-Manin connection in the chiral algebra picture \cite[Section 4.5]{BD} and is given explicitly in the following definition, which is essentially the definition of $1$-point function in \cite{Z96}.
\begin{defn}
Let $(V, \omega)$ be a conformal vertex algebra. A genus-one conformal block is a function
\[
S : V \times \CH \rightarrow \C,
\]
written $S(u, \tau)$, linear in $u$ and holomorphic in $\tau$, such that for each $\tau$ we have $S(-, \tau) \in \CC(E_\tau, V)$, and $S$ satisfies the differential equation
\begin{align*}
2\pi i \frac{d}{d\tau} S(u, \tau) = S(\res_t \wp_1(t, \tau) L(t) u, \tau).
\end{align*}
We shall denote the vector space of genus-one conformal blocks by $\CC^\nabla(V, \omega)$.
\end{defn}
In a similar way, we introduce a connection on charged conformal blocks as follows.
\begin{defn}
Let $(V, \omega, h)$ be a charged conformal vertex algebra. A genus-one charged conformal block is a function
\[
S : V \times \C \times \CH \rightarrow \C,
\]
written $S(u, \al, \tau)$, linear in $u$ and holomorphic in $(\al, \tau)$, such that for each $(\al, \tau)$ we have $S(-, \al, \tau) \in \CC((E, \CL)_{(\al, \tau)}, V)$, and $S$ satisfies the differential equations
\begin{align}
\frac{\partial}{\partial\al} S(u, \al, \tau) &= S(\res_t \wp_1(t, \tau) h(t) u, \al, \tau) \label{eq:DE.y} \\
\text{and} \quad
2\pi i \frac{\partial}{\partial\tau} S(u, \al, \tau) &= S(\res_t \wp_1(t, \tau) L(t) u, \al, \tau). \label{eq:DE.q}
\end{align}
We shall denote the vector space of genus-one charged conformal blocks by $\CC^\nabla(V, \omega, h)$.
\end{defn}
Using the material on elliptic functions in Section \ref{sec:mod.ell} we now cast the definition of charged conformal blocks more explicitly. First we briefly explain for the case of ordinary conformal blocks.

We consider the elliptic curve $E_\tau = \C / \La$ and the local coordinate at the point $0$ induced by the standard coordinate $z$ on $\C$. We denote this local coordinate also by $z$. Using this choice the fibre $\CV_0$ is identified with $V$. The bundle $\CV$ is trivial, and elements of $\Gamma(E_\tau \backslash 0, \CV \otimes \omega_E)$ are linear combinations of $a \, dz$ and $\wp_k(z, \tau) a \, dz$ for $k \geq 2$ and $a \in V$. The space $\CC(E_\tau, V)$ is thus identified with the space of linear functionals $\Phi : V \rightarrow \C$ such that
\begin{align}\label{eq:CB.analytic}
\begin{split}
\Phi(a_{(0)}b) &= 0 \quad \text{for all $a, b \in V$}, \\
\text{and} \quad \Phi(\res_t \wp_2(t, \tau) Y(a, t) b) &= 0 \quad \text{for all $a, b \in V$}.
\end{split}
\end{align}
We note that
\[
\Phi(\res_t \wp_k(t, \tau) Y(a, t) b) = 0
\]
for all $k \geq 2$ follows from the case $k=2$ by substituting $Ta$ in place of $a$ and transfering the derivative to the elliptic function using the integration by parts formula.

We now describe charged conformal blocks. It is instructive to work over $\C$ and then pass to the quotient as in \eqref{eq:elliptic.family}. Consider the map $f : \C \rightarrow \C$ defined by $f(z) = z + \tau$. Denote by $z_p$ the local coordinate at $p$ induced by the global coordinate $z$. Let $a \in V^{(k)}$ and consider the local sections $(v, z_p, a)$ and $(v, z_{f(p)}, a)$. The quotient by the $\Z^2$-action identifies $(v, p)$ with $(e^{-2\pi i \al}v, f(p))$, so
\[
f^*((v, z_{f(p)}, a)) = (e^{-2\pi i \al}v, z_p, a) = (v, z_p, e^{-2\pi i \al h_0} a) = (v, z_p, e^{-2\pi i k \al} a).
\]
If $k\al \notin \La$ then the function $\psi(z, -k\al, \tau)$ is well-defined and meromorphic in $z$ and, because of the transformation formula \eqref{eq:psi.transf}, we obtain a section
\[
\psi(z, -k\al, \tau) (v, z, a) \, dz \in \Gamma(E \backslash 0, \CV \otimes \omega_E).
\]
In fact, this section has a simple pole at $0$. We now have a characterisation of charged conformal blocks similar to (\ref{eq:CB.analytic}):
\begin{align}\label{eq:twisted.CB.analytic}
\begin{split}
\Phi(a_{(0)}b) &= 0 \quad \text{for $b \in V$ and $a \in V^{(k)}$ if $k\al \in \La$}, \\
\Phi(\res_t \wp_2(t, \tau) Y(a, t) b) &= 0 \quad \text{for $b \in V$ and $a \in V^{(k)}$ if $k\al \in \La$}, \\
\text{and} \quad \Phi(\res_t \psi(t, -k\al, \tau) Y(a, t) b) &= 0 \quad \text{for $b \in V$ and $a \in V^{(k)}$ if $k\al \notin \La$}.
\end{split}
\end{align}

\begin{lemma}
If $\Phi$ is a charged conformal block on $V = \bigoplus_{k \in \Z} V^{(k)}$ then $\Phi(V^{(k)}) = 0$ for each $k \neq 0$.
\end{lemma}

\begin{proof}
We have $h \in V^{(0)}$, so if $b \in V^{(k)}$ we have that $\Phi$ annihilates $h_{(0)}b = k b$. For $k \neq 0$ we have immediately that $\Phi(b) = 0$.
\end{proof}

\section{Finiteness of charged conformal blocks}\label{sec:finiteness}

In this section we prove finite dimensionality of spaces of charged conformal blocks under appropriate hypotheses on the vertex algebra $V$. In order to do this, it is convenient to first formulate a notion of \emph{formal coinvariants} inspired by \eqref{eq:twisted.CB.analytic}, but expressed purely in terms of series in $q$ and $y$.

\begin{defn}
Let $N$ be a positive integer, and let $R^{(N)}\subset \C((y^{-1}))[[q]]$ denote the subring generated by the series $\widetilde{G}_{2k}(q)$ for $k\geq 1$, by $E_k(y^c,q)$ for $k\geq 2$ and $|c|\leq N$, and by $E_1^-(y^c,q)$ for $0 < c \leq N$. (Note that $E_1^\pm$ differs from $P_1^\pm$ by a constant, and $P_1^+(y^{-1},q) = -P_1^-(y,q)$). We also denote
\[
R=R^{(1)},\qquad \widetilde{R}=\bigcup_{N\geq 1} R^{(N)}.
\]
\end{defn}

\begin{lemma}
For each fixed $N \geq 1$, the ring $R^{(N)}$ is Noetherian.
\end{lemma}

\begin{proof}

For each integer $c$ with $|c|\leq N$ we consider the ring
\[
\CE^{(c)}
:=
\C\bigl[\widetilde{G}_{2k}(q),\, E_m(y^c,q)\ \big|\ k\geq 1,\ m\geq 2\bigr]
\subset \C((y^{-1}))[[q]].
\]
By Proposition \ref{prop:Li23} the ring $\mathcal E^{(c)}$ is Noetherian.


Since $N$ is fixed, there are only finitely many integers $c$ with $|c|\leq N$. The image of the morphism $\bigotimes_{|c|\leq N}\mathcal E^{(c)} \rightarrow
\C((y^{-1}))[[q]]$ is Noetherian and $R^{(N)}$ is the subring generated, over this image, by the finite set of elements $E_1^-(y^c,q)$ for $0 < c \leq N$. Therefore $R^{(N)}$ is Noetherian.
\end{proof}


{{We say that a conformal vertex algebra $V$ is \emph{strongly generated} by a subset $U \subset V$ if $V$ is spanned by vectors of the form $u^{(1)}_{(-n_1)} \cdots u^{(r)}_{(-n_r)} \vac$ with $u^{(j)} \in U$ and $n_j \geq 1$. We call $V$ \emph{finitely strongly generated} if it is strongly generated by a finite set $\{u_1,\ldots,u_r\}$.}}

Inspired by \eqref{eq:twisted.CB.analytic}, we formulate a definition of \emph{formal coinvariants}.
\begin{defn}
Let $(V, \omega, h)$ be a charged conformal vertex algebra. Assume $V$ is finitely strongly generated by elements $\{u_1, \ldots, u_r\}$, and let $N = \max\{|c_i|\}$ where $c_i$ is the charge of $u_i$. The space of \emph{formal coinvariants} $\wtil{H}$ is the quotient ${R^{(N)}} \otimes_\C V^{(0)} / S$, where $S$ is the ${R^{(N)}}$-submodule spanned by elements
\begin{align*}
\begin{split}
a_{(0)}b \quad &\text{for $a, b \in V^{(0)}$}, \\
\res_t \wtil\wp_2(t, q) Y(a, t) b \quad &\text{for $a, b \in V^{(0)}$}, \\
\res_t P_1^+(t, y^{c}, q) Y(u_i, t) b \quad &\text{for $u_i \in V^{(c)}$ and $b \in V^{(-c)}$ where $c \in \Z_{> 0}$}, \\
\text{and} \quad \res_t P_1^-(t, y^{c}, q) Y(u_i, t) b \quad &\text{for $u_i \in V^{(c)}$ and $b \in V^{(-c)}$ where $c \in \Z_{< 0}$}.
\end{split}
\end{align*}
\end{defn}
Technically this definition of the space of formal coinvariants depends on the auxiliary datum of the choice of strong generators. However this choice is not essential, nor is the value of $N$ appearing in the definition, so long as it is finite.

We recall that the vector space quotient $R_V = V / V_{(-2)}V$ carries the structure of a Poisson algebra, when endowed with commutative product $a \cdot b = a_{(-1)}b$ and Poisson bracket $\{a, b\} = a_{(0)}b$. In \cite{Z96} Zhu introduced this algebra, as well as the finiteness condition $\dim(R_V) < \infty$ which serves as a sufficient condition guaranteeing $\dim \CC(E_\tau, V) < \infty$ {\cite[Lemma 4.4.1]{Z96}}. In fact the much weaker condition $\dim R_V / \{R_V, R_V\} < \infty$ (implied in particular by the \emph{quasi-lisse} condition) is already sufficient {\cite[Proposition 5.2]{Arakawa-Kawasetsu}}.

We form an analogous definition in the charged case.
\begin{defn}\label{def:stably.quasi-lisse}
Let $(V, \omega, h)$ be a charged conformal vertex algebra with associated Zhu Poisson algebra $R_V$. Write $I$ for the commutative algebra ideal in $R_V$ generated by $\bigoplus_{c \neq 0} V^{(c)}$. The quotient $R_V / I$ is a Poisson algebra which we denote $R_V^h$. We refer to the corresponding affine Poisson scheme $X_V^h = \specm{R_V^h}$ as the \emph{relative associated variety}.
\end{defn}
It is easy to see, but not completely obvious, that $R_V^h$ is a Poisson algebra. It is clear that $R_V^h$ is naturally a quotient of $V^{(0)}$, and any element of $I \cap R_V^{(0)}$ is expressed as a linear combination of terms of the form $a \cdot b$ where $a \in R_V^{(c)}$ and $b \in R_V^{(-c)}$ for some $c \neq 0$. But then for any $x \in R_V^{(0)}$ we have $\{a \cdot b, x\} = a \cdot \{b, x\} + b \cdot \{a, x\} \in I$.

For convenience let us denote
\[
C(V) = V^{(0)}_{(-2)}V^{(0)} + \sum_{c \neq 0} V^{(c)}_{(-1)}V^{(-c)},
\]
so that $R_V^h = V^{(0)} / C(V)$.
\begin{defn}
We shall say that $V$ is \emph{stably quasi-lisse} if $X_V^h$ has finitely many symplectic leaves.
\end{defn}
In particular if $V$ is stably quasi-lisse then $\dim R_V^h / \{R_V^h, R_V^h\} < \infty$ \cite{Arakawa-Kawasetsu, ES2010}. Notice that if we take $h = 0$ then the charge grading collapses to $V = V^{(0)}$, and $R_V^h = R_V$ so that we recover the usual quasi-lisse condition. On the other hand if $V$ is quasi-lisse, and $h \neq 0$, it is not clear if $V$ is stably quasi-lisse. The issue is that $I \subset R_V$ is not necessarily an ideal for $\{\cdot, \cdot\}$. Nevertheless, we shall examine several classes of examples below, and in these cases the stable quasi-lisse condition indeed follows from the quasi-lisse condition.
\begin{prop}\label{prop:formal.finiteness}
Let $V = \bigoplus_{\D, c} V_\D^{(c)}$ be a charged conformal vertex algebra. Suppose that
\begin{itemize}
\item $V$ is finitely strongly generated by elements $\{u_1, \ldots, u_r\}$, and let $N = \max\{|c_i|\}$ where $c_i$ is the charge of $u_i$,

\item $V$ is stably quasi-lisse.
\end{itemize}
Then the space of formal coinvariants $\wtil{H}$ is finitely generated as an $R^{(N)}$-module.
\end{prop}

\begin{proof}
Let $\D_0$ be chosen so that $W = \bigoplus_{\D \leq \D_0} V^{(0)}_\D$ contains a vector subspace of $V^{(0)}$ complementary to $C(V)$. Note that $W$ is finite dimensional.

The vertex algebra $V$ is spanned by the elements
\begin{align}\label{eq:PBW}
u^{i_1}_{(-n_1)} u^{i_2}_{(-n_2)} \cdots u^{i_s}_{(-n_s)} \vac,
\end{align}
where $n_j \geq 1$ for each $j$. 
Any element of $V_{(-2)}V$ can be written, by repeated application of the Borcherds identity if necessary, as a linear combination of terms of the form \eqref{eq:PBW} in which $n_1 \geq 2$. Set $a = u^{i_1}$ and $b = u^{i_2}_{(-n_2)} \cdots u^{i_s}_{(-n_s)} \vac$.

If $a \in V^{(0)}$ then $\res_t \wtil\wp_2(t, q) Y(a, t) b \in S$ and so $a_{(-2)}b$ is equal to a linear combination of elements of strictly smaller conformal weight modulo $S$. If $a \in V^{(c)}$ for some $c \neq 0$ then $\res_t P_1^\pm(t, y^c, q) Y(a, t) b \in S$ and the same conclusion holds for $a_{(-1)}b$. Thus all elements of ${R}^{(N)} \otimes V^{(0)}$ are congruent modulo $S$ to elements of ${R}^{(N)} \otimes W$.
\end{proof}
%
In the following corollary, and in many places throughout the paper, we will require that our vertex algebra $V$ be finitely strongly generated, and the number $N$ defined as in the statement of Proposition \ref{prop:formal.finiteness} will play a role. On the understanding that $N$ will always be as in Proposition \ref{prop:formal.finiteness}, we shall not repeat its definition in full. The following result is analogous to {\cite[Lemma 10.8]{EH2024.chiral.H1}}.
\begin{cor}\label{cor:informal.finiteness}
Let $V$ be a charged conformal vertex algebra. Suppose that $V$ is finitely strongly generated and stably quasi-lisse. Then $\CC((E, \CL)_{(\al, \tau)}, V)$ is finite dimensional whenever
\[
0 < -N \im(\al) < \im(\tau).
\]
\end{cor}

We will strengthen this result in Corolloary \ref{cor:finite.everywhere} below: in fact $\CC((E, \CL)_{(\al, \tau)}, V)$ is finite dimensional whenever $N\al \notin \Z + \Z\tau$.

\begin{proof}
Fix $\tau$ and $\al$ lying in the prescribed domain. Then the series defining the generators of $R^{(N)}$ converge at $q = e^{2\pi i \tau}$ and $y = e^{2\pi i \al}$, so there is a well defined evaluation homomorphism $R^{(N)} \rightarrow \C$. Proposition \ref{prop:formal.finiteness} implies that the vector subspace $S \otimes_{R^{(N)}} \C$ of $V$ is of finite codimension. Let $\Phi$ be a charged conformal block, in the sense of Definition \ref{defn:charged.CB}. Then $\Phi \in V^*$ annihilates $S \otimes_{R^{(N)}} \C$ by equation \eqref{eq:twisted.CB.analytic}. It follows that the vector space
$\CC((E, \CL)_{(\al, \tau)}, V)$ is finite dimensional.
\end{proof}

In the remainder of this section we establish the stable quasi-lisse condition for several classes of affine vertex algebras and $W$-algebras. The following result and its proof are closely related to {\cite[Lemma 6.2]{AE2019}}.
\begin{prop}\label{prop:quasi-lisse.enough}
Let $\g$ be a simple Lie algebra and $k$ a level for which the simple affine vertex algebra $L_k(\g)$ is quasi-lisse. Let $\omega$ be the usual Sugawara conformal vector and $\rho$ the Weyl vector. Then $(L_k(\g), \omega, \rho)$ is stably quasi-lisse. Furthermore the space of formal coinvariants $\wtil H$ in this case is finitely generated over $R^{(1)}$.
\end{prop}

\begin{proof}
Let $V = V^k(\g)$ denote the universal affine vertex algebra, and $L = L_k(\g)$ the simple quotient. If $E^{\al_1}, \ldots, E^{\al_\ell} \in \g$ are the simple root vectors and $F^{\al_1}, \ldots, F^{\al_\ell} \in \g$ the corresponding negative root vectors, then $V$ and $L$ are both strongly generated by the $E^{\al_i}_{(-1)}\vac$ and $F^{\al_i}_{(-1)}\vac$ which have charge, respectively, $+1$ and $-1$.

It is well known that $R_V \cong \C[\g^*]$, i.e., $\spec(R_V)$ is $\g^*$ as an affine space, which may be identified with $\g$ via the Killing form. The surjection $R_V \rightarrow R_V / I$ corresponds geometrically to restriction to the locus where $E^{\al} = F^\al = 0$, in particular $\spec(R_V^h) = \h$.

Passage to the quotient $L$ of $V$ corresponds geometrically to intersection with the associated variety $X_L = \spec(R_L)$. For a simple affine vertex algebra $L = L_k(\g)$, it is known that $L$ is quasi-lisse if and only if $X_L \subset \CN$, where $\CN \subset \g$ is the nilpotent cone \cite{Ara09b}. Thus the spectrum of $R_L^h$ is the intersection $\h \cap X_{L} \subset \h \cap \CN = \{0\}$. It follows that $R_L^h$ is finite dimensional, which implies in particular that $L$ is stably quasi-lisse.

The second statement follows from the choice of $h$, so that our strong generators all have charge $\pm 1$.
\end{proof}
Proposition~\ref{prop:quasi-lisse.enough} applies, in particular, to all simple affine vertex algebras $L_k(\g)$ of admissible level $k$, since these are quasi-lisse {\cite[Theorem 5.14]{Ara09b}}. We recall briefly that an admissible level for the simple Lie algebra $\g$ \cite{KW89} is a rational number of the form $k = -h^\vee + p/u$ where $p, u \in \Z_{\geq 1}$ are coprime, and $p \geq h^\vee$ if $(r^\vee, u) = 1$ in which case $k$ is said to be a principal admissible level, and $p \geq h$ if $(r^\vee, u) = r^\vee$ in which case $k$ is said to be a coprincipal admissible level.

The first examples of quasi-lisse $L_k(\g)$ of \emph{non-admissible} level $k$ were discovered in \cite{AM.Joseph} in connection with the Deligne exceptional series. In particular $L_k(\g)$ is quasi-lisse for $\g = D_4, E_6, E_7, E_8$ and $k$ integral in the range $-h^\vee/6-1 \leq k \leq -1$.

We next discuss the stable quasi-lisse condition for affine $W$-algebras. Let $\W^k(\g,f)=H_{DS,f}^0(V^k(\g))$ be the universal affine $W$-algebra associated with the simple Lie algebra $\g$, nilpotent element $f \in \g$, and level $k\in \C$. Here $H_{DS,f}(?)$ denotes the functor of quantized Drinfeld-Sokolov reduction associated with $f$ \cite{FF90,KacRoaWak03}. We also denote by $\W_k(g,f)$ the unique simple graded quotient of $\W^k(\g,f)$.

Let $\mf{s}=\{e,h,f\}$ be an $\mf{sl}_2$-triple in $\g$ associated with $f$. The adjoint action of $h$ induces a grading $\g=\bigoplus_{j\in \frac{1}{2}\Z}\g_j$ of $\g$ by eigenvalues, i.e., $\g_j=\{x\in \g\mid [h,x]=2j x\}$. If we denote by $\g^f$ the centraliser of $f$ in $\g$, and $\g_j^f=\g^f\cap \g_j$, then we have the induced grading $\g^f=\bigoplus_{j\in \frac{1}{2}\Z_{\leq 0}} \g^f_{j}$. According to \cite{De-Kac06} the associated variety of $\W^k(\g,f)$ is the Slodowy slice at $f$, that is
\begin{align}\label{eq:DSK.iso}
R_{\W^k(\g,f)} \cong \C[S_f]\cong \C[\g^f],
\end{align}
where $S_f=f+\g^e$ is the Slodowy slice.

We denote by $\g^{\natural}=\g^f_0$ the centralizer of $\mf{s}$ in $\g$.  Let $\mf{t}$ be a maximal toral subalgebra of $\g^\natural$ and $\mf{h}$ a Cartan subalgebra of $\g$ containing $h$ and $\mf{t}$. Note that $\mf{t}$ is also a maximal toral subalgebra of $\g^f$. We denote by $\Delta \subset \h^*$ the root system of $\g$ with respect to $\mf{h}$, and by $\Delta^f \subset \mf{t}^*$ the restricted root system, that is, the set of nonzero weights of $\mf{t}$ acting on $\g^f$ \cite{BruGoo07,BruGooKle08}. The centralizer of $\mf{t}$ in $\g$ is a minimal Levi subalgebra $\mf{l}$ of $\g$ containing $f$, and we clearly have the following restricted root space decomposition
\begin{align}
\g^f = \mf{l}^f \+ \bigoplus_{\alpha \in \Delta^f} \g^f_{\alpha}
\end{align}
Here the zero weight space $\mf{l}^f = \mf{l} \cap \g^f$ is the centralizer of $f$ in $\mf{l}$. Let $\Pi^f$ be a base of the restricted root system $\Delta^f$, so that any element of $\Delta^f$ can be written as $\sum_{\alpha\in \Delta^f} a_\alpha \alpha$ with either all $a_{\alpha}\in \Z_{\geq 0}$ or all $a_{\alpha}\in \Z_{\leq 0}$. Finally set
\begin{align}\label{eq:rho_f}
\rho_f=\frac{1}{2}\sum_{\alpha\in \Delta^f} \alpha,
\end{align}

The weight-one space of $\W^k(\g,f)$ generates the universal affine vertex algebra associated with $\g^{\natural}$ and we have a vertex algebra embedding
\begin{align}\label{eq:weight-one-subalgenra}
V^{k^{\natural}}(\g^\natural)\hookrightarrow \W^k(\g,f),
\end{align}
where $k^{\natural}$ is some complex number that depends on $k$.

Let $\{h_1,\dots, h_{\dim \mf{t}}, x_1,\dots, x_r\}$ be a basis of
$\mf{l}^f$ such that $\{h_1,\dots, h_{\dim \mf{t}}\}$ is a basis of $\mf{t}$. Let us also choose a basis $\{y_{\alpha,i}\}$ of each restricted weight space $\g_{\alpha}^f$, so that $\C[\g^f] = \C[h_i, x_i, y_{\alpha,i}]$. Considering the surjection $\W^k(\g,f)\twoheadrightarrow R_{\W^k(\g,f)}\cong \C[\g^f]$, we may choose lifts of the listed generators to a set of strong generators of $\W^k(\g,f)$, whose associated quantum fields we denote
\begin{align*}
h_i(z), \quad x_i(z),\quad \text{and} \quad y_{\alpha,i}(z),
\end{align*}
respectively, with the properties that $\{h_i(z)\}$ generates the image of the Heisenberg algebra associated with $\mf{t}$ by \eqref{eq:weight-one-subalgenra},
\begin{align*}
h_i(z)x_j(w)\sim 0, \quad \text{and} \quad h_i(z)y_{\alpha,j}(w)\sim \frac{\alpha(h_i)}{z-w}y_{\alpha,j}(w).
\end{align*}
We denote the element in the image of \eqref{eq:weight-one-subalgenra} corresponding to the element $\rho_f$ also by ${\rho_f}$.

Notice that the charge grading of $\W^k(\g,f)$ by $\rho_f$ runs over $\Z_{\geq 0}$ and so, from Definition \ref{def:stably.quasi-lisse}, we have $R_{\W^k(\g,f)}^{\rho_f} = R_{\W^k(\g,f)}^{(0)} = \C[h_i, x_i]$, and the isomorphism \eqref{eq:DSK.iso} yields
\begin{align}
R_{\W^k(\g,f)}^{\rho_f} \cong \C[S^{\mf{l}}_f]
\end{align}
at the level of Poisson algebras, where $S^{\mf{l}}_f$ is the Slodowy slice $f+\mf{l}^e$ at $f$ in the Levi subalgebra $\mf{l}$ (cf. \cite[Theorem 4.3]{BruGooKle08}).

\begin{thm}\label{thm:stably-quasi-lisse-W}
A quasi-lisse $W$-algebra $\mathscr{W}_k(\g,f)$ is stably quasi-lisse with respect to $\rho_f$, defined as above. In particular, the simple $W$-algebra $\mathscr{W}_k(\g,f)$ is stably quasi-lisse whenever $L_k(\g)$ is quasi-lisse (equivalently $X_{L_k(\g)} \subset \mathcal{N}$), and $f\in X_{L_k(\g)}$.
\end{thm}

\begin{proof}
Let $G$ be the adjoint group of $\g$, and $L$ the Levi subgroup of $G$ whose Lie algebra is $\mf{l}$. We have the restriction map
\begin{align}\label{eq:restricion-invariants}
\C[\g^*]^G\hookrightarrow \C[\mf{l}^*]^L.
\end{align}
We choose $\mf{h}$ a Cartan subalgebra of $\g$ contained in $\mf{l}$, and denote by $W$ and $W_L \subset W$ the Weyl groups of $G$ and $L$, respectively. By the Chevalley restriction theorem, the embedding \eqref{eq:restricion-invariants} is identified with the embedding $\C[\h^*]^W \hookrightarrow \C[\h^*]^{W_L}$ which clearly presents the latter algebra as a finitely generated module over the former, and hence we see that the morphism $\mf{l}^*/\!/L\to \g^*/\!/G$ is finite.

By \cite{Pre07}, the restriction maps $\C[\g^*]^G \to  \C[S_f]$ and $\C[\mf{l}^*]^L\to  \C[S_f^{\mf{l}}]$ are embeddings, and their images are identified with Poisson centers of $\C[S_f]$ and $\C[S_f^{\mf{l}}]$, respectively. Let us denote by $Y$ and $Y_L$ the spectra of the images
of $\C[\g^*]^G $ and $\C[\mf{l}^*]^L$ under the morphisms
\begin{align*}
\C[\g^*]^G \to  \C[S_f] &\cong R_{\W^k(\g,f)}\to R_{\W_k(\g,f)} \\
\text{and} \quad \C[\mf{l}^*]^L\to  \C[S_f^{\mf{l}}] &\cong R_{\W^k(\g,f)}^{\rho_f} \to R_{\W_k(\g,f)}^{\rho_f},
\end{align*}
respectively. Then we have
\begin{align*}
Y_L\cong Y\times_{\g^*/\!/G}\mf{l}^*/\!/L,
\end{align*}
and it therefore follows that
$Y_L$ is finite over $Y$.

Since $\W_k(\g,f)$ is quasi-lisse, or equivalently $Y=\{0\}$ as topological spaces, we deduce that $Y_L=\{0\}$ as topological spaces as well. Hence
$\on{Specm}(R_{\W_k(\g,f)})$ is contained in $S^{\mf{l}}_f \cap \mc{N}^{\mf{l}}$, where $\mc{N}^{\mf{l}}$ is the nilpotent cone of $\mf{l}$. Whence the first assertion.

Next we suppose that $L_k(\g)$ is quasi-lisse, or equivalently that
$X_{L_k(\g)}\subset \mc{N}$. By \cite[Theorem 4.21]{Ara09b}, we have $X_{H_{DS,f}(L_k(\g))}= S_f\cap X_{L_k(\g)}\subset S_f\cap \mc{N}$. Hence,
for any $f\in X_{L_k(\g)}$, we have that $H_{DS,f}(L_k(\g))$ is nonzero and quasi-lisse. Since $\W_k(\g,f)$ is a quotient of $H_{DS,f}(L_k(\g))$ by \cite[Theorem 4.15]{Ara09b}, we deduce that it is quasi-lisse as well, and thus stably quasi-lisse by the first assertion.
\end{proof}

\section{Trace functions and charged conformal blocks}\label{sec:circulating.trace}

Let $(V, \omega, h)$ be a charged conformal vertex algebra. In this section we prove that certain trace functions on an appropriate category of $V$-modules, which we call \emph{stable $V$-modules}, are flat sections of charged conformal blocks in a formal sense. The precise statement is Proposition \ref{prop:conformalblock} below.

\begin{defn}
Let $M$ be a $V$-module. We shall say that $M$ is a weight module if it possesses a decomposition $M = \bigoplus_{\D, k} M_{\D}^{(k)}$ into eigenspaces for $L_0$ and $h_0$. Here we denote by $M_{\D}^{(k)}$ the eigenspace in which $L_0=\D$ and $h_0=k$. We shall say that $M$ is a stable $V$-module if
\begin{itemize}
\item $\dim M_{\D}^{(k)} < \infty$ for all $\D$ and $k$,

\item there exists $h$ such that $M_{\D}^{(k)} = 0$ for all $\D < h$,

\item for each $\D$ there exists $K$ such that $M_{\D}^{(k)} = 0$ for all $k > K$.
\end{itemize}
If $V = V^k(\g)$ then $V$-modules in category $\OO$ are examples.
\end{defn}

\begin{prop}\label{prop:conformalblock}
{{Let $(V,\omega,h)$ be a charged conformal vertex algebra and let $u\in V$.}} Let $M$ be a stable $V$-module. Then the trace function
\[
\tr_M u_0 y^{h_0} q^{L_0},
\]
viewed as a formal series in powers of $y$ and $q$, satisfies the following identities:
\begin{align}
\tr_M (\res_z Y[a, z]b \, dz)_0 y^{h_0} q^{L_0} &= 0 \quad \text{for $a, b \in V^{(0)}$} \label{eq:tr.formal.1} \\
\tr_M (\res_z \wtil\wp_2(z, q) Y[a, z]b \, dz)_0 y^{h_0} q^{L_0} &= 0\quad \text{for $a, b \in V^{(0)}$} \label{eq:tr.formal.2} \\
\tr_M (\res_z P_1^+(z, y^c, q) Y[a, z]b \, dz)_0 y^{h_0} q^{L_0} &= 0 \quad \text{for $a \in V^{(c)}$, $b \in V^{(-c)}$ where $c \in \Z_{> 0}$} \label{eq:tr.formal.3+} \\
\tr_M (\res_z P_1^-(z, y^c, q) Y[a, z]b \, dz)_0 y^{h_0} q^{L_0} &= 0 \quad \text{for $a \in V^{(c)}$, $b \in V^{(-c)}$ where $c \in \Z_{< 0}$}.  \label{eq:tr.formal.3-}
\end{align}
\end{prop}


In short, the proposition identifies trace functions on stable $V$-modules as charged conformal blocks relative to Zhu's vertex algebra structure $Y[-, z]$ on $V$.

\begin{proof}
Recall
\[
Y[u, z] = e^{2\pi i \D_u z} Y(u, e^{2\pi i z} - 1).
\]
From the commutator formula and Borcherds identity, we have the following useful formulas:
\begin{align}\label{eq:comm.[]}
\begin{split}
2\pi i (\res_z e^{-2\pi i nz} Y[a, z]b dz)_0
&= (\res_w (1+w)^{-n+\D_a-1} Y(a, w)b dw)_0 \\
&= \sum_{j =0}^\infty \binom{\D-1-n}{j} (a_{(j)}b)_0 \\
&= [a_{-n}, b_n],
\end{split}
\end{align}
for any $n \in \Z$, and
\begin{align}\label{eq:bor.[]}
\begin{split}
2\pi i\left(\res_z \frac{e^{2\pi i z}}{e^{2\pi i z}-1} Y[a, z]b dz\right)_0
&= (\res_w w^{-1} (1+w)^{\D_a} Y(a, w)b dw)_0 \\
&= \sum_{j =0}^\infty \binom{\D_a}{j} (a_{(j-1)}b)_0 \\
&= \sum_{n =0}^\infty \left( a_{-n} b_n + b_{-1-n} a_{1+n} \right).
\end{split}
\end{align}

We prove \eqref{eq:tr.formal.1}. By the commutator identity \eqref{eq:comm.[]} we have
\begin{align*}
2\pi i (\res_z Y[a, z]b \, dz)_0 = [a_0, b_0]
\end{align*}
as endomorphisms of $M$. 
Since $[L_0, a_0] = [h_0, a_0] = 0$, and by symmetry of the trace, we have
\begin{align*}
\tr_M a_0 b_0 y^{h_0} q^{L_0} = \tr_M b_0 y^{h_0} q^{L_0} a_0 = \tr_M b_0 a_0 y^{h_0} q^{L_0}
\end{align*}
and so we are done.

Before proving \eqref{eq:tr.formal.2}, \eqref{eq:tr.formal.3+} and \eqref{eq:tr.formal.3-}, we make some preparations. Let $c \in \Z$ arbitrary and suppose $a \in V^{(c)}$ and $b \in V^{(-c)}$. For any $n \in \Z$ we have 
\begin{align*}
\tr_{M_\D^{(k)}} a_{-n} b_n
= \tr_{M_{\D-n}^{(k-c)}} b_n a_{-n}
\end{align*}
and hence
\begin{align*}
\tr_{M_\D^{(k)}} a_{-n} b_n y^{h_0} q^{L_0}
= y^{c} q^{n} \tr_{M_{\D-n}^{(k-c)}} b_n a_{-n} y^{h_0} q^{L_0}.
\end{align*}
If we now assume that $n > 0$ then we can iterate the manipulation just carried out to obtain the following
\begin{align*}
\tr_{M_\D^{(k)}} a_{-n} b_n y^{h_0} q^{L_0}
= \sum_{m=1}^\infty (y^c q^{n})^m \tr_{M_{\D-mn}^{(k-mc)}} [b_n, a_{-n}] y^{h_0} q^{L_0},
\end{align*}
which is well-defined as the sum on the right hand side truncates (and becomes finite) due to the positive energy condition on $M$. Summing over all $\D$ and $k$ gives
\begin{align*}
\tr_{M} a_{-n} b_n y^{h_0} q^{L_0}
= -\sum_{m=1}^\infty (y^c q^{n})^m \tr_{M} [a_{-n}, b_n] y^{h_0} q^{L_0},
\end{align*}
or more compactly
\begin{align}\label{ab.telescope}
\tr_{M} a_{-n} b_n y^{h_0} q^{L_0}
= -\frac{y^c q^n}{1 - y^c q^n} \tr_{M} [a_{-n}, b_n] y^{h_0} q^{L_0}.
\end{align}
Similarly, for $n > 0$ we have
\begin{align}\label{ba.telescope}
\tr_M b_{-n} a_n y^{h_0} q^{L_0} = +\frac{y^{-c}q^n}{1-y^{-c}q^n} \tr_M [a_n, b_{-n}] y^{h_0} q^{L_0}.
\end{align}

We now prove \eqref{eq:tr.formal.3+} and \eqref{eq:tr.formal.3-}. Let $a \in V^{(c)}$ and $b \in V^{(-c)}$. By combining \eqref{eq:bor.[]} and \eqref{eq:comm.[]}, and then \eqref{ab.telescope} and \eqref{ba.telescope} we see that
\begin{align}
&2\pi i \tr_M ( \res_z \left( - \frac{e^{2\pi i z}}{e^{2\pi i z}-1} + \sum_{n = 1}^\infty \left( \frac{y^{-c}q^n}{1-y^{-c}q^n} e^{2\pi i n z} - \frac{y^{c} q^n}{1-y^{c}q^n} e^{-2\pi i n z} \right) \right) Y[a, z] b \, dz )_0 y^{h_0} q^{L_0} \nonumber \\
= {} & -\sum_{n =0}^\infty \tr_M \left( a_{-n} b_n + b_{-1-n} a_{1+n} \right) y^{h_0} q^{L_0} \nonumber \\
&+ \sum_{n = 1}^\infty \left( \frac{y^{-c}q^n}{1-y^{-c}q^n} \tr_M [a_n, b_{-n}] - \frac{y^{c} q^n}{1-y^{c}q^n} \tr_M [a_{-n}, b_n] \right) y^{h_0} q^{L_0} \label{eq:zeta.to.2point} \\
= {} & -\sum_{n =0}^\infty \tr_M \left( a_{-n} b_n + b_{-1-n} a_{1+n} \right) q^{L_0} y^{h_0} + \sum_{n = 1}^\infty \tr_M \left( b_{-n} a_{n} + a_{-n} b_n \right) y^{h_0} q^{L_0} \nonumber \\
= {} & -\tr_M a_0 b_0 y^{h_0} q^{L_0}. \nonumber
\end{align}
If $c > 0$ then we have
\begin{align*}
\tr_{M^{(k)}} a_0 b_0 y^{h_0} q^{L_0}
&= \tr_{M^{(k)}} ([a_0, b_0] + b_0 a_0) y^{h_0} q^{L_0} \\
&= \tr_{M^{(k)}} [a_0, b_0] q^{L_0} y^{h_0} + y^{-c} \tr_{M^{(k+c)}} a_0 b_0 y^{h_0} q^{L_0} \\
&= \sum_{m=0}^\infty y^{-mc} \tr_{M^{(k+mc)}} [a_0, b_0] y^{h_0} q^{L_0}
\end{align*}
Summing on $k$ yields
\begin{align*}
\tr_{M} a_0 b_0 y^{h_0} q^{L_0}
= \frac{1}{1-y^{-c}} \tr_M [a_{0}, b_{0}] y^{h_0} q^{L_0}
\end{align*}
where the fraction is shorthand for its geometric series expansion. This completes the proof of \eqref{eq:tr.formal.3+}. If $c < 0$ then we similarly have
\begin{align*}
\tr_{M^{(k)}} a_0 b_0 y^{h_0} q^{L_0}
&= y^{c} \tr_{M^{(k-c)}} b_0 a_0 y^{h_0} q^{L_0} \\
&= y^c \tr_{M^{(k-c)}} (-[a_0, b_0] + a_0 b_0) y^{h_0} q^{L_0} \\
&= -\sum_{m=1}^\infty y^{mc} \tr_{M^{(k-mc)}} [a_0, b_0] y^{h_0} q^{L_0},
\end{align*}
which yields
\begin{align*}
\tr_{M} a_0 b_0 y^{h_0} q^{L_0}
= -\frac{y^c}{1-y^{c}} \tr_M [a_{0}, b_{0}] y^{h_0} q^{L_0}.
\end{align*}
This completes the proof of \eqref{eq:tr.formal.3-}. Notice that in both of these calculations, traces on certain charge eigenspaces are rewritten in terms of traces on spaces of more positive charge. Because of the stability condition on $M$, these sums truncate. This is one of the reasons we have distinguished the cases $c>0$ and $c<0$.

We now prove \eqref{eq:tr.formal.2}. Let $a, b \in V^{(0)}$. We note that \eqref{eq:zeta.to.2point} holds for $c=0$, in which case it reduces to
\begin{align}\label{eq:trace.master}
\tr_M a_0 b_0 y^{h_0} q^{L_0} = 2\pi i \tr_M (\res_z \wtil\wp_1(z, q) Y[a, z] b \, dz)_0 y^{h_0} q^{L_0}.
\end{align}

If we replace $a$ with $L_{[-1]}a$ in \eqref{eq:trace.master}, then on the left hand side we have $(L_{[-1]}a)_0 = 2\pi i (L_0 a + L_{-1} a)_0 = 0$ and on the right hand side we have $Y[L_{[-1]}a, z] b = \partial_z Y[a, z]$. By integration by parts the derivative can be moved to $\wtil\wp_1(z, q)$, which gives us
\begin{align*}
\tr_M (\res_z \wtil\wp_2(z, q) Y[a, z] b \, dz)_0 y^{h_0} q^{L_0} = 0.
\end{align*}
\end{proof}

In the course of the proof of Proposition \ref{prop:conformalblock} we proved the following useful identity.
\begin{prop}\label{prop:zeta.to.2point}
Let $M$ be a stable $V$-module and let $a, b \in V^{(0)}$. Then
\begin{align*}
&\tr_M a_0 b_0 y^{h_0} q^{L_0}
= \tr_M ( \res_z \wtil\wp_1(z, q) Y[a, z] b \, dz )_0 y^{h_0} q^{L_0}.
\end{align*}
\end{prop}

We now pass from unnormalised to normalised trace functions.
\begin{prop}\label{prop:trace.DE}
The normalised trace function
\[
S_M(u, y, q) = \tr_M u_0 y^{h_0} q^{L_0-c/24}
\]
satisfies
\begin{align*}
2\pi i y \frac{\partial}{\partial y} S_M(u, y, q) &= S_M(\res_t \wtil\wp_1(t, q) h[t] u, y, q ) \\
(2\pi i)^2 q \frac{\partial}{\partial q} S_M(u, y, q) &= S_M(\res_t \wtil\wp_1(t, q) L[t] u, y, q).
\end{align*}
\end{prop}

\begin{rem}
The second of these equations is essentially Zhu's differential equation {\cite[Lemma 5.2.1]{Z96}}.
\end{rem}

\begin{proof}
By setting $a = \omega$ in Proposition \ref{prop:zeta.to.2point} and $b = u$ we obtain
\begin{align*}
q \frac{\partial}{\partial q} \tr_M b_0 y^{h_0} q^{L_0}
&= \tr_M L_0 b_0 q^{L_0} y^{h_0} = \tr_M ( \res_t \wtil\wp_1(t, q) Y[\omega, t] b )_0 y^{h_0} q^{L_0}.
\end{align*}
It is straightforward to see that this implies the second of the claimed differential equations, upon normalisation. The proof of the other differential equation is very similar.
\end{proof}

\begin{rem}\label{rem:Ramond.OK}
Suppose the triple $(V, \omega, h)$ conforms to the definition of a charged conformal vertex algebra except that the eigenvalues of $L_0$ on $V$ might not all lie in $\Z$. We define a charged conformal block as a linear function $\Phi : V \rightarrow \C$ satisfying \eqref{eq:twisted.CB.analytic}, \eqref{eq:DE.y} and \eqref{eq:DE.q} as before. All the results of Section \ref{sec:finiteness} go through without modification. The results of the present section also hold as long as we take for $M$ a Ramond twisted $V$-module. We refer to \cite{De-Kac06, VE13} for details, but in brief this means that, for all $u \in V$, the field $Y^M(u, z)$ acts with integral conformally indexed modes \eqref{eq:conf.modes}, and in particular every $u \in V$ has a well defined zero mode $u_0$ in $M$.
\end{rem}

\section{Holonomicity}\label{sec:holonomic}

Let us denote $X = \{(y, q) \in \C^2 \mid 0 \leq |q| < |y^{-1}|^N < 1 \}$ and let $U$ be a connected and simply connected open subset of $\{(y, q) \in X \mid q \neq 0\}$. In this section we show that the differential equations \eqref{eq:DE.y} and \eqref{eq:DE.q} equip the sheaf $\CC|_U$ with the structure of an integrable connection $\nabla$. Furthermore the equations have a regular singularity at $q=0$.


\subsection{Integrability of the connection}\label{sec:integrable}

For notational convenience we shall introduce operators
\begin{align*}
[R_L S](a, y, q) &= -S(\res_z \wtil\wp_1(z, q) L[z] a, y, q), \\
\text{and} \quad [R_h S](a, y, q) &= -S(\res_z \wtil\wp_1(z, q) h[z] a, y, q).
\end{align*}
(The minus signs are introduced to make the commutation relations of $R_L$ and $R_h$ more cleanly reflect those of $L$ and $h$, as in the dual action of a Lie algebra, see below.) Then the differential equations \eqref{eq:DE.y} and \eqref{eq:DE.q} are expressed as
\begin{align*}
\left( y\frac{\partial}{\partial y} + \frac{1}{2\pi i} R_h \right)S &= 0, \\
\left( q\frac{\partial}{\partial q} + \frac{1}{(2\pi i)^2} R_L \right) S &= 0.
\end{align*}
Explicitly the connection $\nabla : \CC \rightarrow \CC \otimes \Omega_U$ becomes
\begin{align}\label{eq:ward-connection}
\nabla(S) = \left(\frac{\partial}{\partial q} + \frac{1}{(2\pi i)^2 q} R_L\right)S \otimes dq + \left(\frac{\partial}{\partial y} + \frac{1}{2\pi i y} R_h\right)S \otimes dy.
\end{align}

\begin{prop}
The connection $\nabla$ on $\CC|_U$ is integrable.
\end{prop}

\begin{proof}
We must verify $\nabla^2(S) = 0$. We compute directly
\begin{align*}
\nabla^2(S) &= \left[ \partial_q + \frac{1}{(2\pi i)^2 q}R_L, \partial_y + \frac{1}{2\pi iy}R_h \right]S \otimes dq \wedge dy \\
&= \left(\frac{1}{(2\pi i)^3 qy} [R_L, R_h] + \frac{1}{2\pi i y}(\partial_qR_h) \right)S \otimes dq \wedge dy,
\end{align*}
where the term $(\partial_qR_h)$ stands for the endomorphism
\[
[(\partial_q R_h) S](a, y, q) = -S\left(\res_z (\partial_q \wtil\wp_1(z, q)) h[z] a, y, q\right).
\]
We note that a similar term $(\partial_yR_L)$ vanishes because $\wtil\wp_1(z, q)$ is independent of $y$. For the purpose of computing $[R_L, R_h]$, we now do an auxiliary computation, in which we write $\wtil\wp_1(z)$ in place of $\wtil\wp_1(z, q)$. We have
\begin{align*}
\res_z \res_w \wtil\wp_1(z) \wtil\wp_1(w) [L[z], h[w]] &= \res_z \res_w \wtil\wp_1(z) \wtil\wp_1(w) \left( h'[w] \delta(z, w) + h[w] \partial_w \delta(z, w) \right) \\
&= \res_w \wtil\wp_1(w)^2 h'[w] + \res_w \wtil\wp_1(w) \wtil\wp_1'(w) h[w]  \\
&= -\res_w \wtil\wp_1(w) \wtil\wp_1'(w) h[w].
\end{align*}
Therefore
\begin{align*}
([R_L, R_h]S)(a, y, q) = S(\res_w \wtil\wp_1(w) \wtil\wp_1'(w) h[w] a, y, q).
\end{align*}
Now we recall the identity
\begin{align*}
(2\pi i)^2 q \partial_q \wtil\wp_1(w, q) - \wtil\wp_1(w, q) \wtil\wp'(w, q) = \wtil\wp_3(w, q).
\end{align*}
Notice that the right hand side is an elliptic function (see {\cite[Section 7]{E2017}}), so
\begin{align*}
\nabla^2(S) = - \frac{1}{(2\pi i)^3qy} A^*S \otimes dq \wedge dy
\end{align*}
where
\begin{align*}
{{[A^*S](u, y, q) = S(\res_w \wtil\wp_3(w, q) h[w] u, y, q),}}
\end{align*}
and this vanishes by definition of $\CC$.
\end{proof}

\subsection{Regularity}\label{sec:regularity}

We continue in the setup of Section \ref{sec:integrable}. We make use of the following standard result, in this case paraphrased from {\cite[Proposition 8.8]{1970-Katz}}.
\begin{thm}
Let $(\CF, \nabla)$ be a coherent sheaf equipped with integrable connection. Then $\CF$ is locally free, i.e., isomorphic to the sheaf of sections of a vector bundle.
\end{thm}
In particular, the restriction $\CC|_U$ is a vector bundle of finite rank $r$ with connection $\nabla$. Since we assume $U$ simply connected, $\CC|_U$ is trivial and we may choose a trivialisation $\CC|_U \rightarrow U \times \C^r$ with frame $\mathbf{e}=(e_1,\dots,e_r)$. Let $u = u(y, q)$ be the coordinate column vector of a horizontal section $S \in \Gamma(U,\CC)$ with respect to $e$. Then $u$ satisfies a first-order system
\begin{equation}\label{eq:system-AB}
q\frac{\partial}{\partial q} u = B(y,q)\,u,\qquad y\frac{\partial}{\partial y} u = A(y,q)\,u
\end{equation}
for some matrices $A,B \in \Mat_r(\OO_X)$. Indeed the entries of the matrices $A, B$ are expressed as elements of the ring $R^{(N)} \subset \C(\!(y^{-1})\!)[\![q]\!]$ and these series are convergent in $X$.

The intersection of $X$ with the locus $q=0$ is the set $L = \{y \in \C \mid |y| > 1\}$. The matrices $A$ and $B$ have series expansions
\[
A(y,q) = \sum_{n=0}^\infty A_n(y)\,q^n \quad \text{and} \quad B(y,q)=\sum_{n=0}^\infty B_n(y)\,q^n
\]
where $A_n(y), B_n(y) \in \C(\!(y^{-1})\!)$ are convergent on $L$.

We wish to use regularity of the connection at $q=0$ to prove that $u(y,q)$ has a series expansion of Frobenius type, in analogy with the classical theory of ordinary differential equations with regular singular point. To be able to do this, we need the following simple but key isospectrality result.
\begin{lemma}\label{lem:constant-q-exponents}
The characteristic polynomial $\det(\lambda I - B_0(y))$ of $B_0(y)$ is independent of $y$.
\end{lemma}

\begin{proof}
The integrability condition $\nabla^2=0$ implies the identity
\begin{equation}\label{eq:flatness-AB}
[y\partial_y-A,\;q\partial_q-B]=0,
\end{equation}
equivalently
\begin{equation}\label{eq:zero-curvature-AB}
(y\partial_y)B-(q\partial_q)A+[A,B]=0.
\end{equation}
Taking the coefficient of $q^0$ in \eqref{eq:zero-curvature-AB} gives
\begin{equation}\label{eq:lax-B0}
y\partial_y B_0=[A_0,B_0],
\end{equation}
where $A_0(y)$ is the $q^0$-term of $A(y,q)$.

For $k\ge 1$ we have, by Leibniz and \eqref{eq:lax-B0},
\[
y\partial_y(B_0^k)=\sum_{i=0}^{k-1}B_0^i\,(y\partial_y B_0)\,B_0^{k-1-i}
=\sum_{i=0}^{k-1}B_0^i\,[A_0,B_0]\,B_0^{k-1-i}=[A_0,B_0^k].
\]
Taking traces yields $y\partial_y\tr(B_0^k)=0$ for all $k\ge 1$. Hence all coefficients of
$\det(\lambda I-B_0)$ are $y$-independent (e.g. by Newton identities), proving the result.
\end{proof}

\begin{prop}\label{prop:s7-implies-s4-form}
Let $u$ be a solution of \eqref{eq:system-AB}. Then $u$ can be expressed by a convergent series of the form
\begin{align}\label{eq:u.Frob.0}
u(y,q) = \sum_{\al=1}^K q^{r_\al} u_\al(y,q),
\end{align}
where the $r_\al$ are constants, no two of which differ by an integer, and each function $u_\al(y, q)$ has a series expansion of the form
\begin{align}\label{eq:u.Frob.1}
u_\al(y,q) = \sum_{j=0}^\infty \sum_{\mu=0}^{M_j} \log(q)^\mu q^j u_{\al,m,j}(y).
\end{align}
Furthermore, each function $u_{\al,\mu,0}(y)$ has an expansion of the form
\begin{align}\label{eq:u.Frob.2}
u_{\al,\mu,0}(y) = \sum_{\beta=1}^{L} \sum_{\nu=0}^{N} y^{h_\beta} \log(y)^\nu \sum_{\ell=0}^\infty C_{\al,\beta,\mu,\nu,0,\ell} y^{-\ell},
\end{align}
where the $C_{\al,\beta,\mu,\nu,0,\ell}$ are $r$-component column vectors written relative to the frame $\textbf{e}$.
\end{prop}

\begin{proof}
The first order equation $q\partial_q u = B(y,q) u$ can be converted into equations of order $r$ on each of the components, in the usual way. By \cite[Theorem 3.3]{Mandai2000} it is possible to write $u(y, q)$ in the form \eqref{eq:u.Frob.0} and \eqref{eq:u.Frob.1}, where for each $\al = 1, \ldots, K$ the function $r_\al = r_\al(y)$ potentially has nontrivial dependence on $y$.

In particular each $u_\al(y, q)$ takes the form (here $M = M_0$)
\begin{equation}\label{eq:frob-log-q-compact}
q^{r(y)} \Big( \sum_{\mu=0}^{M} \log(q)^\mu v_m(y) \Big) + O(q^{r(y)+1}),
\qquad v_{M}(y)\neq 0.
\end{equation}
We apply $(q\partial_q-B)$ to \eqref{eq:frob-log-q-compact} and, since
$q\partial_q(q^{r(y)})=r(y)\,q^{r(y)}$ and $q\partial_q \log (q)^{M} = M \log(q)^{M-1}$ lowers the logarithmic degree, we see that the coefficient of $q^{r(y)}\log(q)^M$ in $(q\partial_q - B)u$ is
\[
r(y) v_{M}(y) - B_0(y) v_{M}(y).
\]
Thus $(q\partial_q-B)u=0$ implies $B_0(y)v_{M}(y)=r(y)v_M(y)$, so $r(y)$ is an eigenvalue of $B_0(y)$. By Lemma \ref{lem:constant-q-exponents}, the eigenvalues of $B_0(y)$ are independent of $y$, hence $r(y)$ is constant. 

If we now substitute \eqref{eq:u.Frob.1} into the equation $y\partial_y u = A(y,q) u$, it is clear that the leading coefficients $u_{\al,\mu,0}(y)$ satisfy
\[
y\partial_y u = A(y,0) u.
\]
We recall that the entries of $A(y, q)$, considered as power series, lie in $R^{(N)}$ and in particular are polynomial in the functions $P_k^{\pm}(y^c, q)$, and at $q=0$ these simplify to
\[
Q_k(y) = \sum_{n=0}^{\infty} n^{k-1} y^{-n}.
\]
In particular the coefficients of $A(y,0)$ are holomorphic at $y=\infty$. It follows that $u_{\al,\mu,0}(y)$ has a Frobenius expansion of the type \eqref{eq:u.Frob.2}.
\end{proof}

\section{Stable rationality and the stable Zhu algebra}\label{sec:stable.rat}

\begin{defn}
Let $(V, \omega, h)$ be a charged conformal vertex algebra, and let us denote by $V\mod^{\omega, h}$ the category of finitely generated stable $V$-modules. We say $V$ is stably rational if $V\mod^{\omega, h}$ is a semisimple category.
\end{defn}

Let $(V, \omega, h)$ be a charged conformal vertex algebra, and let $M = \bigoplus_{\D, k} M_\D^{(k)}$ be a nonzero stable $V$-module. If $\D_0$ is the minimal value of $\D$ for which $M_\D = \bigoplus_k M_\D^{(k)}$ is nonzero, then we denote $M_{\D_0}$ by $M_{\Top}$. If $k_0$ is the maximal value of $k$ for which $M_{\D_0}^{(k_0)}$ is nonzero, then we denote $M_{\D_0}^{(k_0)}$ by $M_{\peak}$.

\begin{rem}
The explanation for the terminology, in which $M_{\Top}$ refers the eigenspace with lowest eigenvalue of $L_0$, traces back to the fact that modules of ``highest weight'' in representation theory correspond in quantum theory to those whose energy spectrum is bounded below.
\end{rem}

\subsection{The stable Zhu algebra}\label{sec:stable.zhu}

We begin by recalling the definition of the Zhu algebra $A(V)$ of a conformal vertex algebra $(V, \omega)$ \cite{Z96}. Let $J \subset V$ denote the vector subspace spanned by elements of the form
\[
\res_w w^{-2}(1+w)^{\Delta_a}Y(a,w)b \, dw, \quad \text{$a, b \in V$}.
\]
By the Borcherds identity \eqref{eq:Borcherds_identity} we have
\begin{align*}
\bigl(\res_w w^{-2}(1+w)^{\Delta_a}Y(a,w)b \bigr)_0
&=
\sum_{j\in \Z_+}\binom{\Delta(a)}{j}(a_{(-2+j)}b)_0 \\
&=
\sum_{j\in \Z_+}(-1)^j\bigl(a_{-1-j}b_{1+j}-b_{-1-j}a_{1+j}\bigr).
\end{align*}
Hence, if $M$ is a positive-energy $V$-module and $M_{\Top}$ denotes its lowest conformal-weight subspace, then the assignment
\[
V \to \en(M_{\Top}), \qquad a \mapsto a_0,
\]
induces a linear map
\[
V/J \to \rm{End}(M_{\Top}).
\]
The remarkable observation of Zhu is that $V/J$ acquires the structure of an associative algebra with the product
\[
a * b = \res_w w^{-1}(1+w)^{\Delta(a)}Y(a,w)b \, dw, \quad \text{$a, b \in V$},
\]
unital with unit $\vac$, and the map above is a homomorphism of associative algebras. The algebra $(V/J, *)$ is denoted $A(V)$. For any irreducible left $A(V)$-module $N$, there exists an irreducible positive energy $V$-module $M$ such that $M_{\Top} = N$ as $A(V)$-modules.

We record the following useful result \cite[Equation (2.1.4)]{Z96}.
\begin{lemma}\label{lem:Zhu.comm}
For all $a, b \in A(V)$ the relation holds
\[
a * b - b * a = 2\pi i a_{([0])}b.
\]
\end{lemma}

Assume now that $(V, \omega, h)$ is a charged conformal vertex algebra. Then $A(V)$ inherits a $\Z$-grading by charge from that on $V$, we write $A(V)=\bigoplus_{c\in\mathbb Z} A(V)^{(c)}$.

Let $M$ be a stable $V$-module, so that $M_{\Top}$ is a left $A(V)$-module. Then since $a_0$ raises the $h_0$-eigenvalue by $c$ for $a\in V^{(c)}$, the space $M_{\peak}$ is annihilated by $a_0$ for all $a\in V^{(c)}$ with $c>0$. This motivates us to define:
\begin{defn}
The stable Zhu algebra of the charged conformal vertex algebra $(V, \omega, h)$ is the quotient algebra
\[
A^{\stab}(V) = A(V)^{(0)} \big/ \bigl(A(V)A(V)^{(>0)}\bigr)^{(0)}.
\]
\end{defn}

By construction, the action of $A(V)$ on $M_{\Top}$ induces an action of $A^{\stab}(V)$ on $M_{\peak}$.


\begin{thm}\label{thm:stable.induction}
Let $P$ be an irreducible left $A^{\stab}(V)$-module. Then there exists an irreducible stable $V$-module $M$ for which $M_{\peak} = P$ as $A^{\stab}(V)$-modules.
\end{thm}

\begin{proof}
Through the surjection $A(V)^{(0)} \rightarrow A^{\stab}(V)$ we obtain a left action of $A(V)^{(0)}$ on $P$. We extend the action to an action of $A(V)^{(\geq 0)}$ by setting $A(V)^{(> 0)} P = 0$. We now consider the $A(V)$-module
\[
\wtil{N} = A(V) \otimes_{A(V)^{(\geq 0)}} P.
\]
It is clear that $\wtil N$ inherits a $\Z$-grading from the charge grading on $A(V)$. Furthermore under this grading we have $\wtil N = \bigoplus_{k \in \Z_{\leq 0}} \wtil N^{(k)}$ and $\wtil N^{(0)} = P$.

Since $h$ lies in the centre of $A^{\stab}(V)$, it acts as a scalar in $P$, and so $h$ acts in the graded pieces of $\wtil N$ by distinct scalars. It follows that $A(V)$-submodules of $\wtil N$ are graded. Let $S \subset \wtil N$ be a proper $A(V)$-submodule, then we claim $S \subset \bigoplus_{k < 0} \wtil N$. Indeed it is clear that $S^{(0)} \subset \wtil N^{(0)}$ is stable under the action of $A(V)^{(0)}$ hence also of $A^{\stab}(V)$, hence trivial by irreducibility of $P$.

We now define $I \subset \wtil N$ to be the sum of all proper $A(V)$-submodules, which as we have seen above is contained in $\bigoplus_{k < 0} \wtil N$, so we may form the irreducible $A(V)$-module $N = \wtil N / I$ which is $\Z_{\leq 0}$-graded with $N^{(0)} = P$.

Finally we apply the Zhu induction to $N$ to obtain an irreducible $V$-module $M$ with $M_{\Top} = N$ and hence $M_{\peak} = P$.
\end{proof}

\begin{lemma}\label{lem:peak-simple}
Let $M$ be an irreducible stable $V$-module. Then $M_{\peak}$ is an irreducible
left $A^{\stab}(V)$-module.
\end{lemma}
\begin{proof}
Let $0\neq P\subset M_{\peak}$ be an $A^{\stab}(V)$-submodule, and let $N\subset M$ be the $V$-submodule generated by $P$. Since $M$ is irreducible, one has $N=M$.

By the usual Zhu algebra argument, $N_{\Top}$ is generated
from $P$ by the action of $A(V)$. Passing to the highest $h_0$-eigenspace, this action
factors through $A^{\stab}(V)$. Hence
\[
N_{\peak}=P.
\]
Since $N=M$, it follows that $P=M_{\peak}$.
\end{proof}

\begin{prop}\label{stablezhu-semisimple}
If $V$ is stably rational then $A^{\stab}(V)$ is finite dimensional and semisimple.
\end{prop}

\begin{proof}
Let $P$ be a finite-dimensional left $A^{\stab}(V)$-module. Without loss of generality $P$ is a sum of generalised eigenspaces of $h$ of some fixed eigenvalue $k$, and likewise $L$ of eigenvalue $\D$. By a similar argument as in the proof of Theorem \ref{thm:stable.induction}, there exists a stable $V$-module $M$ such that
\[
M_{\peak}\cong P
\]
as $A^{\stab}(V)$-modules. Indeed we form $\wtil N$ as in that proof, and then apply Zhu induction directly to obtain a stable $V$-module.

Since $V$ is stably rational, the category $V\mod^{\omega, h}$ is semisimple. Hence
\[
M\cong \bigoplus_{i=1}^k M^i
\]
for irreducible stable $V$-modules $M^1,\dots,M^k$. Taking peak spaces, we recover $P$ as the intersection of $M_{\D, k}$ with
\[
M_{\peak}\cong \bigoplus_{i=1}^k M^i_{\peak}.
\]
Each $M^i_{\peak}$ is an irreducible $A^{\stab}(V)$-module by Lemma \ref{lem:peak-simple}. Therefore $P$ is semisimple. Since $P$ was arbitrary, every
finite-dimensional left $A^{\stab}(V)$-module is semisimple. Thus $A^{\stab}(V)$ is semisimple.

Assume now that $V\mod^{\omega, h}$ has only finitely many simple objects up to isomorphism. By
Theorem~\ref{thm:stable.induction}, every irreducible left $A^{\stab}(V)$-module arises as
$M_{\peak}$ for some irreducible stable $V$-module $M$. Hence $A^{\stab}(V)$ has only finitely
many irreducible left modules up to isomorphism. Moreover, each such irreducible module is
finite-dimensional, since it is the peak space of a stable module. Therefore
$A^{\stab}(V)$ is a finite-dimensional semisimple algebra.
\end{proof}

In the rest of this section we discuss stable rationality of affine vertex algebras and $W$-algebras.
\begin{prop}\label{prop:admissible-affine-stable}
Let $V=L_k(\g)$ be a simple affine vertex algebra at an admissible level, and take $h=\rho$. Then the category of finitely generated stable $V$-modules coincides with category $\mathcal O$. In particular, $V$ is stably rational. Hence, by Proposition~\ref{stablezhu-semisimple}, the stable Zhu algebra $A^{\stab}(V)$ is finite-dimensional and semisimple.
\end{prop}
  
\begin{proof}
Let
\[
M=\bigoplus_{\Delta,r} M_\Delta^{(r)}
\]
be a finitely generated stable \(V\)-module. Thus each \(M_\Delta^{(r)}\) is
finite-dimensional, the \(L_0\)-spectrum of \(M\) is bounded below, and, for each fixed
\(\Delta\), the \(\rho_0\)-eigenvalues occurring in \(M_\Delta\) are bounded above.

Let \(a\in \g\). Since \(a\) has conformal weight \(1\), the mode \(a_{(n)}\) lowers
conformal weight by \(n\). Hence, for \(n>0\), the operator \(a_{(n)}\) acts locally
nilpotently on \(M\). It follows that \(\g\otimes t\C[t]\) acts locally nilpotently on
\(M\).

Next, let \(x_\alpha\in\g_\alpha\) be a positive root vector. Then \((x_\alpha)_{(0)}\)
preserves conformal weight and satisfies
\[
[\rho_0,(x_\alpha)_{(0)}]=\alpha(\rho)(x_\alpha)_{(0)},
\qquad \alpha(\rho)>0.
\]
Since the \(\rho_0\)-spectrum on each \(M_\Delta\) is bounded above, \(x_\alpha(0)\)
acts locally nilpotently on \(M_\Delta\), and hence on \(M\). Thus \(\n_+\) acts
locally nilpotently on \(M\).

As \(M\) is finitely generated and \(L_0,\rho_0\) act semisimply, we may choose a finite
set of homogeneous generators. The preceding local nilpotence then implies that the
weights of \(M\) are contained in a finite union of sets of the form
\[
\lambda_i-\sum_{\beta\in\widehat{\Delta}_+}\Z_{\geq0}\beta .
\]
Moreover, for each fixed pair \((\Delta,r)\), only finitely many \(\mathfrak h\)-weights
can occur in such a finite union of downward affine cones. Since \(M_\Delta^{(r)}\) is
finite-dimensional by stability, it follows that the affine weight spaces of \(M\) are
finite-dimensional. Therefore \(M\) belongs to category \(\mathcal O\).

Conversely, let \(M\in\mathcal O\). Then \(M\) has finite-dimensional
\(\widehat{\mathfrak h}\)-weight spaces and its weights are contained in a finite
union of downward affine cones. This implies that the \(L_0\)-spectrum is bounded
below. Moreover, for each fixed conformal weight \(\Delta\) and each \(r\), only
finitely many finite \(\mathfrak h\)-weights \(\mu\) with \(\mu(\rho)=r\) can occur in
this finite union of cones. Hence \(M_\Delta^{(r)}\) is a finite direct sum of
finite-dimensional \(\widehat{\mathfrak h}\)-weight spaces, and is therefore
finite-dimensional. Finally, for each fixed \(\Delta\), the \(\rho_0\)-eigenvalues are
bounded above. Thus \(M\) is stable.

Therefore the category of finitely generated stable \(V\)-modules coincides with the
category of \(L_k(\g)\)-modules in \(\mathcal O\). Since \(k\) is admissible, this category
is semisimple by \cite[Main Theorem]{A12-2}. Hence \(V\) is stably rational. The final
assertion follows from Proposition~\ref{stablezhu-semisimple}.
\end{proof}

Recall that a nilpotent element $f$ of $\g$ is said to be of standard Levi type if it is a principal nilpotent element in some Levi subalgebra of $\g$. Note that if $\g$ is of type $A$ then all nilpotent elements are of standard Levi type.
\begin{thm}\label{thm:W-algebra}
Let $V = L_k(\g)$ be a simple affine vertex algebra at an admissible level $k$, and let $f$ be a nilpotent element of standard Levi type that is contained in $X_{L_k(\g)}$. Then the simple $W$-algebra $\mathscr{W}_k(\g,f)$ is stably rational with respect to $\rho_f$ \textup{(}defined in \eqref{eq:rho_f}\textup{)}.
\end{thm}

\begin{proof}
Recall the decomposition \eqref{eq:rho_f}. By hypothesis, $f$ is a principal nilpotent element in $\mf{l}$.
By \cite{Ara07,De-Kac06}, we have
\begin{align*}
A(\W^k(\g,f))\cong U(\g,f),
\end{align*}
where $U(\g,f)$ is the finite $W$-algebra associated with $(\g,f)$ \cite{Pre02}. Here $A(\W^k(\g,f))$ denotes the Ramond twisted Zhu algebra of $\W^k(\g,f)$. It follows from {\cite[Theorem 4.3]{BruGooKle08}} that \begin{align*}
A^{\on{stab}}(\W^k(\g,f))\cong U(\mf{l},f),
\end{align*}
the finite $W$-algebra associated with $(\mf{l},f)$. Since $f\in \mf{l}$ is principal, we have
\begin{align*}
U(\mf{l},f) \cong \mc{Z}(\mf{l}),
\end{align*}
where $\mc{Z}(\mf{l})$ is the center of the enveloping algebra $U(\mf{l})$ of $\mf{l}$ \cite{Kos78}. In particular, the stable Zhu algebra $A^{\on{stab}}(\W^k(\g,f))$ is commutative. On the other hand, one knows that the natural map
$\mc{Z}(\g)\to U(\g,f)$ is injective and the image is identified with the center of $U(\g,f)$ \cite{Pre07}.

By hypothesis, $H_{DS,f}^0(L_k(\g))$ is a nonzero quotient of $\W^k(\g,f)$,
and by {\cite[Theorem 8.1]{A2012Dec}}
we have
\begin{align}
A(H_{DS,f}^0(L_k(\g))\cong H_{f}^0(A(L_k(\g))),
\label{eq:commutativity-Zhu-BRST}
\end{align}
where on the right hand side $H_{f}^0(?)$ denotes the finite dimensional analogue of the quantized Drinfeld-Sokolov reduction, see \cite[Section 3]{A2012Dec} for details.

Let $\mc{Z}\subset A(H_{DS,f}^0(L_k(\g))$ denote the image of the center $\mc{Z}(\g)$ of $U(\g)$ under the composition
\begin{align*}
\mc{Z}(\g) \to U(\g,f)=A(H_{DS,f}^0(V^k(\g))\to A(H_{DS,f}^0(L_k(\g))).
\end{align*}
If we write also $\mc{Z}'$ for $\on{im}(\mc{Z}(\g)\to A(\W_k(\g,f)))$ then $\mc{Z}'$ is the image of $\mc{Z}$ by \eqref{eq:commutativity-Zhu-BRST}.

By \cite[Main Theorem]{A12-2} and \cite[Theorem 3.4]{AEkeren19}, we have that $\on{Spec}\mc{Z}'$ is reduced and $0$-dimensional, and is contained in the set of regular elements of $\h^*/W$ under the Harish-Chandra isomorphism $\mc{Z}(\g) \cong \C[\h^*]^W$. Therefore $\on{Spec}\mc{Z}$ is reduced and $0$-dimensional, and is contained in the set of regular elements in $\h^*/W$.

By \cite[Theorem 4.7]{BruGooKle08},
we have a commutative diagram
\begin{center}
\begin{tikzcd}
\C[\h^*]^{W}  \arrow[d]& \mc{Z}(\g)\arrow[l,"\sim"'] \arrow{d}{c}\arrow[r, tail, twoheadrightarrow]& \mc{Z}\arrow[d]\\%
\C[\h^*]^{W_L} &\mc{Z}(\mf{l})=A^{\on{stab}}(\W^k(\g,f))\arrow[l,"\sim"']\arrow[r, tail, twoheadrightarrow]&A^{\on{stab}}(H_{DS,f}^0(L_k(\g))),
\end{tikzcd}
\end{center}
where the left horizontal arrows are the respective Harish-Chandra homomorphisms and the map $c : \mc{Z}(\g)\to \mc{Z}(\mf{l})$ is as in \cite{BruGooKle08}. Hence
\begin{align*}
\spec A^{\on{stab}}(H_{DS,f}^0(L_k(\g)))\cong \spec{\mc{Z}}\times_{\mf{h}^*/W}\mf{h}^*/W_L.
\end{align*}

Since the morphism $\h^*/W_L\to \h^*/W$ is finite and unramified over the locus of regular elements, {{$\spec A(H_{DS,f}^0(L_k(\g)))$ is reduced and $0$-dimensional.}} Therefore $A^{\on{stab}}(H_{DS,f}^0(L_k(\g)))$ is semisimple, and thus, so is $ A^{\on{stab}}(\W_k(\g,f))$.
\end{proof}
Theorem \ref{thm:W-algebra} provides a large supply of examples to which our main results apply, beyond admissible level affine vertex algebras.
\begin{rem}\label{rem:SXY.prediction}
It would be interesting to compare the dimension of $A^{\on{stab}}(\W_k(\g,f))$ with the prediction {\cite[Equation (3.5)]{SXY2024}} for the number of irreducible $\W_k(\g, f)$-algebra modules in the case of boundary admissible level $k$, obtained from counting fixed points in affine Springer fibres.
\end{rem}

\begin{rem}\label{rem:non.even.W}
If $f$ is an odd nilpotent element then $V = \W^k(\g,f)$ and its quotients may contain elements of conformal weight in $1/2 + \Z$. By the symbol $A(V)$ we refer, still, to the associative algebra defined exactly as at the start of this section, it is the Ramond twisted Zhu algebra of $V$. The association $M \mapsto M_{\Top}$ induces a bijection between irreducible positive energy Ramond twisted $V$-modules and irreducible left $A(V)$-modules. See Remark \ref{rem:Ramond.OK} and \cite{De-Kac06}.
\end{rem}

\begin{rem}
There exist quasi-lisse simple $W$-algebras with trivial weight-one space, such algebras will typically not be stably rational and the modular invariance properties of their trace functions is an interesting question. A specific class of examples is provided by the simple subregular $W$-algebra $W_k(D_4(a_1))$ at level $k = -6 + p/u$ where $p, u \geq 6$ are coprime. Further examples come from any distinguished nilpotent orbit.
\end{rem}

\subsection{Exhaustion by trace functions}\label{sec:exhaustion}

In this subsection we complete the proof of Theorem~\ref{thm:intro.2}. Let $S$ be a genus-one charged conformal block, i.e.\ a flat section of the connection $\nabla$ on $\CC$ in the sense of Section~\ref{sec:holonomic}. Our aim is to show that, when $A^{\stab}(V)$ is finite-dimensional and semisimple, $S$ is a finite linear combination of trace functions on irreducible stable $V$-modules. The argument follows from Zhu's exhaustion method \cite{Z96}, applied first to the $y$-expansion of each $q$-sector and then to powers of $\log(q)$. If $V$ is stably rational, the hypothesis on $A^{\stab}(V)$ holds by Proposition~\ref{stablezhu-semisimple}.

By Proposition~\ref{prop:s7-implies-s4-form}, applied to a local trivialisation of $\CC$, the block $S(a,y,q)$ admits a Frobenius expansion. We use the conventions of \eqref{eq:u.Frob.0}--\eqref{eq:u.Frob.2}. Thus $\al$ indexes a $q$-sector $q^{r_\al}$ ($r_\al$ constant, no two differing by an integer), $j \in \Z_{\geq 0}$ is an extra power of $q$, $\mu \leq M_{\al,j}$ is the power of $\log(q)$ in the slice $q^{r_\al+j}$, $\beta$ indexes a $y$-sector $y^{h_\beta}$, and $\ell$ indexes the tail $\sum_\ell y^{-\ell}$ at $y=\infty$. We have
\begin{align}
S(a,y,q)
&= \sum_{\al=1}^K \sum_{j=0}^\infty \sum_{\mu=0}^{M_{\al,j}}
q^{r_\al} \log(q)^\mu\, q^j\, S_{\al,\mu,j}(a,y),
\label{eq:Frob} \\
\text{in which} \quad S_{\al,\mu,0}(a,y)
&= \sum_{\beta=1}^L \sum_{\nu=0}^{N_{\al,\mu,\beta}}
y^{h_\beta} \log(y)^\nu
\sum_{\ell=0}^\infty C_{\al,\beta,\mu,\nu,0,\ell}(a)\, y^{-\ell}.
\label{eq:FrobFrob}
\end{align}
Each $C_{\al,\beta,\mu,\nu,0,\ell} : V \to \C$ is linear, as in \eqref{eq:u.Frob.2}. The subscripts $(\al,\mu,0)$ on $S_{\al,\mu,0}$ and the fourth entry $0$ in $C_{\al,\beta,\mu,\nu,0,\ell}$ record the $q$-index $j=0$ in \eqref{eq:Frob}. When $\ell=0$ we set $C_{\al,\beta,\mu,\nu,0} := C_{\al,\beta,\mu,\nu,0,0}$ for the \emph{head coefficient} of $y^{h_\beta}\log(y)^\nu$. For $j \geq 1$, the coefficient $S_{\al,\mu,j}$ has the same $y$-expansion with $0$ replaced by $j$. The $y$-direction arguments below use the slice $j=0$; a uniform bound on the integers $M_{\al,j}$ is proved later in Lemma~\ref{lem:ex}.

For $a,b \in V^{(0)}$ we write $a *_q b = \res_z \wtil{\wp}_1(z,q)\,Y[a,z]b\,dz$. The Ward identities \eqref{eq:DE.y} and \eqref{eq:DE.q} become
\begin{align}
S(h *_q a,y,q) &= y\frac{d}{dy}S(a,y,q), \label{ward2} \\
S((\omega-\tfrac{c}{24}\vac) *_q a,y,q) &= q\frac{d}{dq}S(a,y,q). \label{ward1}
\end{align}

The first step is to show that the head coefficients descend to linear functionals on the stable Zhu algebra.

\begin{prop}\label{prop:coef.stab}
For all $\al,\beta,\mu,\nu$, the head coefficient $C_{\al,\beta,\mu,\nu,0,0}$ (i.e.\ $C_{\al,\beta,\mu,\nu,0}$ in \eqref{eq:FrobFrob}) descends to a symmetric linear functional on $A^{\on{stab}}(V)$.
\end{prop}

\begin{proof}
The argument is analogous to \cite[Lemma~5.3.2]{Z96}. At $q=0$,
\[
\wtil{\wp}_2(z,q=0) = (2\pi i)^2 \frac{e^{2\pi i z}}{(e^{2\pi i z}-1)^2},\]
and so
\[
\res_t \wtil{\wp}_2(t,q=0)\,Y[a,t]b\,dt
= 2\pi i \res_w w^{-2}(1+w)^{\Delta_a}Y(a,w)b\,dw.
\]
Since $S$ annihilates the left-hand side for $a,b \in V^{(0)}$, each $C_{\al,\beta,\mu,\nu,0}$ annihilates $\res_w w^{-2}(1+w)^{\Delta_a}Y(a,w)b\,dw$.

Similarly, at $q=0$,
\begin{align}\label{eq:Ppm.at.q=0}
P_1^+(z,y,q=0) &= \sum_{n\geq 0} y^{-n} - \frac{e^{2\pi i z}}{e^{2\pi i z}-1},
\qquad
P_1^-(z,y,q=0) = -\sum_{n\geq 1} y^{n} - \frac{e^{2\pi i z}}{e^{2\pi i z}-1},
\end{align}
and $S$ annihilates
\begin{align}
\res_t \Bigl(1 - \frac{e^{2\pi i t}}{e^{2\pi i t}-1}\Bigr) Y[a,t]b\,dt
\quad &\text{for $a \in V^{(c)}$, $b \in V^{(-c)}$, $c>0$}, \label{eq:Pp.limit} \\
\res_t \Bigl(- \frac{e^{2\pi i t}}{e^{2\pi i t}-1}\Bigr) Y[a,t]b\,dt
\quad &\text{for $a \in V^{(c)}$, $b \in V^{(-c)}$, $c<0$}. \label{eq:Pm.limit}
\end{align}
Replacing $a$ by $Ta$ in \eqref{eq:Pp.limit}--\eqref{eq:Pm.limit}, we see that $C_{\al,\beta,\mu,\nu,0}$ annihilates the term $\res_w w^{-2}(1+w)^{\Delta_a}Y(a,w)b\,dw$ for all $a \in V^{(c)}$, $b \in V^{(-c)}$, hence factors through $A(V)$ and vanishes off charge $0$.

Since $2\pi i \res_t \frac{e^{2\pi i t}}{e^{2\pi i t}-1} Y[a,t]b\,dt = a*b$, the expressions in \eqref{eq:Pp.limit} and \eqref{eq:Pm.limit} equal $a_{[0]}b - \frac{1}{2\pi i}a*b$ and $-\frac{1}{2\pi i}a*b$, respectively.  Therefore, by Lemma~\ref{lem:Zhu.comm}, we have $C_{\al,\beta,\mu,\nu,0}$ annihilates $a*b$ when $a \in V^{(c)}$, $b \in V^{(-c)}$, $c<0$. Lemma~\ref{lem:Zhu.comm} and \eqref{eq:twisted.CB.analytic} imply that $C_{\al,\beta,\mu,\nu,0}$ annihilates $a*b-b*a$ for $a,b \in V^{(0)}$. Thus $C_{\al,\beta,\mu,\nu,0}$ descends to a symmetric linear functional on $A^{\stab}(V)$.
\end{proof}

We next verify two head-coefficient identities used in the exhaustion arguments below (in the $j=0$ slice of \eqref{eq:Frob}; here $N_{\al,\mu,\beta}$ is the top $\log(y)$-power in the sector $y^{h_\beta}$ of $S_{\al,\mu,0}$, and $M_{\al,0}$ is the top $\log(q)$-power in $q^{r_\al}$):
\begin{align}
C_{\al,\beta,\mu,\nu,0}\bigl((h-h_{\beta})*a\bigr) &= 0 && \text{for $\nu=N_{\al,\mu,\beta}$,} \label{eq:h.annihilate} \\
C_{\al,\beta,M_{\al,0},\nu,0}\bigl((\omega-\tfrac{c}{24}-r_{\al})*a\bigr) &= 0 && \text{for each $\nu \geq 0$.} \label{eq:omega.annihilate}
\end{align}
Identity \eqref{eq:omega.annihilate} is the $q$-identity of \cite[Theorem~5.3.1, (5.3.16)]{Z96}; \eqref{eq:h.annihilate} is the analogous $y$-identity from the current $h$ and \eqref{ward2}. At $q=0$,
\[
\wtil{\wp}_1(z,q=0) = 2\pi i\, \frac{e^{2\pi i z}}{e^{2\pi i z}-1},
\]
so the constant term in $q$ of $h*_q a$ is $h*a$.

\emph{Proof of \eqref{eq:h.annihilate}.}
Fix $\al$, $\mu$, and $\beta$, and set $\nu=N_{\al,\mu,\beta}$. Then $C_{\al,\beta,\mu,\nu+1,0}\equiv 0$ on the sector $y^{h_\beta}$, by maximality of $\nu$.
Since \eqref{ward2} involves only $y\partial_y$, it restricts to each slice $S_{\al,\mu,0}$ in \eqref{eq:FrobFrob}.
Substituting into \eqref{ward2} and comparing the head coefficient of $y^{h_\beta}(\log y)^\nu$, the left-hand side gives $C_{\al,\beta,\mu,\nu,0}(h*a)$.
On the right-hand side, only $(\log y)^\nu$ and $(\log y)^{\nu+1}$ contribute, since $y\partial_y(\log y)^\nu$ contains $(\log y)^{\nu-1}$ and therefore lies in the comparison at $\nu-1$; the $(\log y)^{\nu+1}$ term vanishes by maximality of $\nu$. Thus
\[
C_{\al,\beta,\mu,\nu,0}(h*a)=h_\beta\,C_{\al,\beta,\mu,\nu,0}(a),
\]
which is \eqref{eq:h.annihilate}.

\emph{Proof of \eqref{eq:omega.annihilate}.}
Fix $\al$, $\beta$, and $\nu \geq 0$.
Since \eqref{ward1} involves only $q\partial_q$, it restricts to the sector $q^{r_\al}$ in the $j=0$ slice, and $\nu$ may be arbitrary in \eqref{eq:FrobFrob}.
Applying $q\partial_q$ to this sector as in \eqref{eq:frob-log-q-compact}, only the top power $\log(q)^{M_{\al,0}}$ contributes to the coefficient of $q^{r_\al}(\log q)^{M_{\al,0}}$ on the right of \eqref{ward1}; comparing with the left-hand side yields \eqref{eq:omega.annihilate}.

The following lemma is the common exhaustion step; it will be applied separately in the $y$- and $q$-directions.

\begin{lemma}\label{lem:coef.exhaust}
Suppose $A^{\stab}(V)$ is finite-dimensional and semisimple. Let $C : V \to \C$ satisfy
\begin{align*}
C(a) &= 0 && \text{for $a \in J$ and for $a \in V^{(c)}$ for $c \neq 0$}, \\
C(a*b-b*a) &= 0 && \text{for $a,b \in V^{(0)}$}, \\
C(a*b) &= 0 && \text{for $a \in V^{(-c)}$, $b \in V^{(c)}$ for $c > 0$},
\end{align*}
where $J \subset V$ is as in Section \ref{sec:stable.zhu}. Then:
\begin{enumerate}[label=(\roman*), ref=\roman*]
\item\label{lem:coef.exhaust:h}
Fix $h_\beta \in \R$ and an integer $n \geq 1$. If $C((h-h_\beta)^n * a)=0$ for all $a \in V$, there exist irreducible stable $V$-modules $W^1,\ldots,W^k$ and constants $c_i \in \C$ such that
\[
C(a) = \sum_{i=1}^k c_i\, \tr_{W^i_{\peak}}(\bar a)
\qquad \text{for all $a \in V$},
\]
where $\bar a \in A^{\stab}(V)$ is the image of $a$ and $h$ acts on $W^i_{\peak}$ with eigenvalue $h_\beta$.
\item\label{lem:coef.exhaust:omega}
Fix $r \in \R$ and an integer $n \geq 1$. If $C((\omega-\tfrac{c}{24}-r)^n * a)=0$ for all $a \in V$, there exist irreducible stable $V$-modules $W^1,\ldots,W^k$ and constants $c_i \in \C$ such that
\[
C(a) = \sum_{i=1}^k c_i\, \tr_{W^i_{\peak}}(\bar a)
\qquad \text{for all $a \in V$},
\]
where $\bar a \in A^{\stab}(V)$ is the image of $a$ and $(\omega-\tfrac{c}{24})$ acts as $r$ on $W^i_{\peak}$.
\end{enumerate}
\end{lemma}

\begin{proof}
By Proposition~\ref{prop:coef.stab}, $C$ descends to a symmetric linear functional on $A^{\stab}(V)$, again denoted $C$. Semisimplicity implies
\[
C(b)=\sum_{i=1}^k c_i\, \tr_{P^i}(b)
\qquad (b \in A^{\stab}(V)),
\]
for simple $A^{\stab}(V)$-modules $P^1,\ldots,P^k$ and constants $c_i \in \C$ (cf.\ \cite[Lemma~5.3.3]{Z96}).

For \itemref{lem:coef.exhaust:h}, $h$ is central, so $h$ acts on each $P^i$ as a scalar $h_i$. If $c_i \neq 0$ and $h_i \neq h_\beta$, choose $a \in V^{(0)}$ whose image $\bar a \in A^{\stab}(V)$ acts nontrivially on $P^i$; then $C((h-h_\beta)^n * a)=(h_i-h_\beta)^n C(a) \neq 0$ for $n \geq 1$, a contradiction. Hence $h|_{P^i}=h_\beta$ whenever $c_i \neq 0$.

For \itemref{lem:coef.exhaust:omega}, $(\omega-\tfrac{c}{24})$ acts on each $P^i$ as a scalar $r_i$. The same argument with $(\omega-\tfrac{c}{24}-r)^n$ in place of $(h-h_\beta)^n$ shows $r_i=r$ whenever $c_i \neq 0$.

By Theorem~\ref{thm:stable.induction}, $P^i=W^i_{\peak}$ for irreducible stable $W^i$, and $C(a)=\sum_i c_i \tr_{W^i_{\peak}}(\bar a)$ for all $a \in V$, since both sides vanish on $J$ and on $V^{(c)}$ for $c \neq 0$.
\end{proof}

We now treat the $y$-expansion of the leading $q$-coefficient $S_{\al,\mu,0}$ in \eqref{eq:Frob}. Fix $\al$ and $\mu$ and decompose
\begin{equation}\label{eq:T.logy.decomp}
S_{\al,\mu,0}(a,y)=\sum_{\nu\geq 0}\log(y)^\nu T_{\al,\mu,\nu}(a,y).
\end{equation}
Here $T_{\al,\mu,\nu}$ is the coefficient of $(\log y)^\nu$ in the slice $S_{\al,\mu,0}$ (the subscript $\nu$ indexes $\log y$, whereas the subscript $0$ on $S_{\al,\mu,0}$ records $j=0$ in \eqref{eq:Frob}). Set
\[
\nu_{\max} := \max\bigl\{\nu \in \Z_{\geq 0} \mid \text{some head coefficient $C_{\al,\beta,\mu,\nu,0}$ is nonzero}\bigr\},
\]
with the convention that if no such $\nu$ exists, then $\nu_{\max}=0$ and $T_{\al,\mu,\nu_{\max}}=0$.

\begin{lemma}\label{lem:top.ex}
Suppose $A^{\stab}(V)$ is finite-dimensional and semisimple, and let $S$ be a genus-one charged conformal block with expansion \eqref{eq:Frob}. Fix $\al$ and $\mu$. Then the top $\log(y)$-coefficient $T_{\al,\mu,\nu_{\max}}$ in \eqref{eq:T.logy.decomp} satisfies
\[
T_{\al,\mu,\nu_{\max}}(a,y) = \sum_{i=1}^k c_i\, \tr_{W^i_{\Top}}(a_0 y^{h_0})
\]
for irreducible stable $V$-modules $W^1,\ldots,W^k$ and constants $c_i \in \C$.
\end{lemma}

\begin{proof}
By \eqref{eq:FrobFrob} and \eqref{eq:T.logy.decomp}, $T_{\al,\mu,\nu_{\max}}$ is the coefficient of $\log(y)^{\nu_{\max}}$ in $S_{\al,\mu,0}$. For each $\beta$ with $N_{\al,\mu,\beta}=\nu_{\max}$, the $\beta$-sector of \eqref{eq:FrobFrob} contributes only powers $\log(y)^\nu$ with $\nu\leq\nu_{\max}$, and its leading piece is
\[
y^{h_\beta}\log(y)^{\nu_{\max}}\,C_{\al,\beta,\mu,\nu_{\max},0}(a)
\]
plus terms with $y^{h_\beta-\ell}$ for $\ell\geq 1$ (the tail at $y=\infty$). In particular, the factor multiplying $y^{h_\beta}\log(y)^{\nu_{\max}}$ in $T_{\al,\mu,\nu_{\max}}$ is the head functional $C_{\al,\beta,\mu,\nu_{\max},0}$.

Apply part~\itemref{lem:coef.exhaust:h} of Lemma~\ref{lem:coef.exhaust} to $C=C_{\al,\beta,\mu,\nu_{\max},0}$, using \eqref{eq:h.annihilate} (with $\nu_{\max}=N_{\al,\mu,\beta}$ on each contributing sector) and $n=1$. This gives constants $c_{\beta,i}$ and irreducible stable $W^i$ with $h|_{W^i_{\peak}}=h_\beta$ such that
\[
C_{\al,\beta,\mu,\nu_{\max},0}(a)=\sum_i c_{\beta,i}\,\tr_{W^i_{\peak}}(\bar a).
\]
If $h_\beta$ is not minimal among the $h$-exponents in $T_{\al,\mu,\nu_{\max}}$, set
\[
T'_{\al,\mu,\nu_{\max}}(a,y):=T_{\al,\mu,\nu_{\max}}(a,y)-\sum_i c_{\beta,i}\,\tr_{W^i_{\Top}}(a_0 y^{h_0}).
\]
By \eqref{ward2}, $T'_{\al,\mu,\nu_{\max}}$ satisfies the same setup with a strictly smaller leading $h$-exponent. Since $A^{\stab}(V)$ has only finitely many irreducible modules, iterating downward in $h_\beta$ yields the formula.
\end{proof}

\begin{prop}\label{prop:top.ex.N=0}
Under the same hypotheses as in Lemma \ref{lem:top.ex}, we have $N_{\al,\mu,\beta}=0$ for all $\beta$. In particular,
\[
S_{\al,\mu,0}(a,y) = \sum_{i=1}^k c_i\, \tr_{W^i_{\Top}}(a_0 y^{h_0})
\]
for all $\al$ and $\mu$.
\end{prop}

\begin{proof}
Suppose $N_{\al,\mu,\beta}>0$ for some $\beta$, and choose $\nu=\nu_{\max}$ and $a \in V$ with $C_{\al,\beta,\mu,\nu,0}(a)\neq 0$.
Decomposing \eqref{ward2} as in \eqref{eq:T.logy.decomp} gives
\begin{equation}\label{eq:ward2.logy}
T_{\al,\mu,\nu-1}(h*a,y)
= \nu\, T_{\al,\mu,\nu}(a,y)
+ y\frac{\partial}{\partial y} T_{\al,\mu,\nu-1}(a,y).
\end{equation}
Applying \eqref{ward2} again and comparing head coefficients at $y^{h_\beta}$, using \eqref{eq:h.annihilate}, we obtain
\[
C_{\al,\beta,\mu,\nu-1,0}((h-h_\beta)^2*a)=0.
\]
By part~\itemref{lem:coef.exhaust:h} of Lemma~\ref{lem:coef.exhaust} with $n=2$, $T_{\al,\mu,\nu-1}$ is a linear combination of $\tr_{W^i_{\Top}}(a_0 y^{h_0})$; each such trace function satisfies \eqref{ward2}, so $T_{\al,\mu,\nu-1}(h*a,y)=y\partial_y T_{\al,\mu,\nu-1}(a,y)$. Substituting into \eqref{eq:ward2.logy} gives $\nu\, T_{\al,\mu,\nu}(a,y)=0$, contradicting the choice of $\nu$. Hence $N_{\al,\mu,\beta}=0$ for all $\beta$, so $\nu_{\max}=0$ and Lemma~\ref{lem:top.ex} gives the trace expansion for $S_{\al,\mu,0}$.
\end{proof}

We turn to the $q$-expansion. Write
\[
S(a,y,q)=\sum_{\mu=0}^\infty \log(q)^\mu S^{(q)}_\mu(a,y,q),
\]
and for an irreducible stable $V$-module $W$ set
\[
S_W(a,y,q):=\tr_W a_0 y^{h_0} q^{L_0-c/24}.
\]

\begin{lemma}\label{lem:ex}
Suppose $A^{\stab}(V)$ is finite-dimensional and semisimple, and let $S$ be a genus-one charged conformal block with expansion \eqref{eq:Frob}. Then the integers $M_{\al,j}$ occurring in \eqref{eq:Frob} are bounded above uniformly in $\al$ and $j$. Set $M := \max_{\al,j} M_{\al,j}$. Then
\[
S^{(q)}_M(a,y,q) = \sum_{i=1}^k c_i\, S_{W^i}(a,y,q)
\]
for irreducible stable $V$-modules $W^1,\ldots,W^k$ and constants $c_i \in \C$.
\end{lemma}

\begin{proof}
List the irreducible stable $V$-modules as $W^1,\ldots,W^k$, with conformal dimensions $r^{(i)}$. Let $J$ be the largest integer among the differences $r^{(i)}-r^{(j)}$, and set
\[
M_* := \max\{ M_{\al,j} \mid q^{r_\al+j} \text{ occurs in \eqref{eq:Frob} with $0 \leq j \leq J$} \}.
\]
Suppose $S^{(q)}_{\wtil M}\neq 0$ for some $\wtil M > M_*$. Then the leading $q$-terms of $S^{(q)}_{\wtil M}$ lie in sectors $q^{r_\al}$ with $r_\al > r^{(i)}$ for every $i$. The corresponding head coefficients satisfy \eqref{eq:h.annihilate} and the analogue of \eqref{eq:omega.annihilate} with $M_{\al,0}$ replaced by $\wtil M$ in the relevant $q$-sectors. By part~\itemref{lem:coef.exhaust:omega} of Lemma~\ref{lem:coef.exhaust}, they would be linear combinations of traces on modules with $(\omega-\tfrac{c}{24})$-eigenvalue $r_\al$ on the peak space. No such irreducible stable module exists, since $r_\al$ exceeds every $r^{(i)}$. This contradiction shows $M_{\al,j} \leq M_*$ for all $\al$ and $j$, so $M=\max_{\al,j} M_{\al,j}$ is finite.
By maximality of this $M$, the coefficient $S^{(q)}_M$ satisfies \eqref{ward1} and \eqref{ward2}. By Proposition~\ref{prop:top.ex.N=0}, its $y$-expansion involves no $\log(y)$ and each head coefficient is a linear combination of traces on $W^i_{\peak}$. Applying Lemma~\ref{lem:coef.exhaust} and subtracting the corresponding $S_{W^i}$, we obtain a new block satisfying \eqref{ward1} and \eqref{ward2} whose leading $q$-term has strictly larger exponent $r_\al$ or smaller $\log(q)$-degree. Since $A^{\stab}(V)$ has only finitely many irreducible modules, this exhaustion procedure terminates with the stated formula for $S^{(q)}_M$.
\end{proof}

\begin{prop}\label{prop:ex.N=0}
Under the same hypotheses as in Lemma \ref{lem:ex}, every genus-one charged conformal block satisfies
\[
S(a,y,q) = \sum_{i=1}^k c_i\, S_{W^i}(a,y,q).
\]
\end{prop}

\begin{proof}
By Lemma~\ref{lem:ex}, $S=\sum_{\mu=0}^M \log(q)^\mu S^{(q)}_\mu$ for some $M \geq 0$. Suppose $M>0$ and $S^{(q)}_M \neq 0$. By \eqref{ward1},
\begin{equation}\label{eq:ward1.logq}
S^{(q)}_{M-1}((\omega-\tfrac{c}{24}\vac)*_q a,y,q)
= M\, S^{(q)}_M(a,y,q)
+ q\frac{\partial}{\partial q} S^{(q)}_{M-1}(a,y,q).
\end{equation}
Applying \eqref{ward1} again and comparing head coefficients in the sector $q^{r_\al}$ (the $j=0$ slice of \eqref{eq:FrobFrob}) yields $C_{\al,\beta,M-1,0}((\omega-\tfrac{c}{24}-r_\al)^2*a)=0$ for all $\beta$. By part~\itemref{lem:coef.exhaust:omega} of Lemma~\ref{lem:coef.exhaust} with $n=2$, each such head coefficient is a linear combination of traces on $W^i_{\peak}$ with $(\omega-\tfrac{c}{24})$-eigenvalue $r_\al$. By Proposition~\ref{prop:top.ex.N=0}, $S^{(q)}_{M-1}$ has no positive powers of $\log(y)$, and the exhaustion step in the proof of Lemma~\ref{lem:ex} therefore gives
\[
S^{(q)}_{M-1}(a,y,q) = \sum_{i=1}^k c_i\, S_{W^i}(a,y,q),
\]
and therefore $S^{(q)}_{M-1}((\omega-\tfrac{c}{24}\vac)*_q a,y,q)=q\partial_q S^{(q)}_{M-1}(a,y,q)$. Substituting into \eqref{eq:ward1.logq} forces $S^{(q)}_M=0$, a contradiction. Hence $M=0$ and the proposition follows from Lemma~\ref{lem:ex}.
\end{proof}

\begin{thm}\label{thm:main}
Let $(V, \omega, h)$ be a charged conformal vertex algebra, which we assume to be strongly finitely generated, stably quasi-lisse and stably rational. Let the irreducible stable $V$-modules be $W^1, \ldots, W^k$. Then the trace functions
\[
S_{W^i}(u, \al, \tau) = \tr_{W^i} u_0 y^{h_0} q^{L_0-c/24}
\]
converge on the domain $0 < -N \im(\al) < \im(\tau)$, span the space $\CC((E, \CL)_{(\al, \tau)}, V)$ at each point of this domain, and satisfy \eqref{eq:DE.y} and \eqref{eq:DE.q}.
\end{thm}

\begin{proof}
By Proposition~\ref{stablezhu-semisimple}, stable rationality implies that $A^{\stab}(V)$ is finite-dimensional and semisimple. Proposition~\ref{prop:ex.N=0} shows that every genus-one charged conformal block is a linear combination of the $S_{W^i}$. Convergence and the differential equations were established in Section~\ref{sec:circulating.trace}; finite-dimensionality of $\CC((E, \CL)_{(\al, \tau)}, V)$ follows from Theorem~\ref{thm:intro.1}.
\end{proof}

\begin{rem}\label{rem:Ramond.Zhu}
The results of this section continue to hold (verbatim) if we allow $V$ to be graded by non integral conformal weights, as in Remark \ref{rem:Ramond.OK}, as long as we replace the category of $V$-modules by the category of Ramond twisted $V$-modules.
\end{rem}

\section{Modularity of the connection}\label{sec:modularity}

In Section \ref{sec:curves.bundles} we introduced a family $T \rightarrow B$ of pairs $(E, \CL)_{(\al, \tau)}$ consisting of $E = E_\tau$ a smooth elliptic curve and $\CL = \CL_\al$ a holomorphic line bundle over $E$ of degree $0$. In Section \ref{sec:chiral} we described the bundle $\CC$ of charged conformal blocks over $B$, associated with a charged conformal vertex algebra $(V, \omega, h)$ and the connection $\nabla$ on $\CC$ corresponding to the differential equations \eqref{eq:DE.y} and \eqref{eq:DE.q}.

In Sections \ref{sec:finiteness} and \ref{sec:holonomic} we saw that, for suitable $V$, the pair $(\CC, \nabla)$ is an integrable connection in the domain
\begin{align*}
X  = \{(\al, \tau) \in B \mid N \al \notin \Z + \Z \tau\} \subset B
\end{align*}
(where $N$ is a positive integer which depends on $V$), and the trace function $S_{M}(u, \al, \tau)$ associated with a stable $V$-module is convergent for $(\al, \tau)$ in a certain domain $Y \subset X$, and converges in $Y$ to a flat section of $(\CC, \nabla)$. Explicitly
\begin{align*}
Y  = \{(\al, \tau) \in B \mid -\im(\tau) < N \im(\al) < 0 \} \subset X.
\end{align*}
Furthermore, if $V$ is stably rational, the finite set of trace functions $S_M$ associated with irreducible stable $V$-modules forms a basis of the space of flat sections of $(\CC, \nabla)$ on $Y$.

Although the connection $\nabla$ develops singularities on the lattice $N \al \in \Z + \Z \tau$, we now prove that these singularities are regular. The proof is straightforward, and just requires some careful book keeping with respect to pole orders.
\begin{prop}
Let $(V, \omega, h)$ be a charged conformal vertex algebra, which we assume to be finitely strongly generated and stably quasi-lisse. Then the connection $\nabla$ on $\CC$ exhibits regular singularities at $N \al \in \Z + \Z \tau$.
\end{prop}

\begin{proof}
Let us take $\tau \in \CH$ as fixed. Since $\CC$ is of finite rank, the equation $\nabla S = 0$ reduces as in Section \ref{sec:integrable} to a first order system of ODEs $S' = A S$, where here and throughout the rest of this proof we denote $\partial_\al = 2\pi i y \partial_y$ by $(-)'$ for convenience. Here $A$ is a matrix whose entries lie in the ring $R^{(N)}$, and $S$ is a column vector with entries $S_i = S(u_i, \al, \tau)$ where $u_1, u_2, \ldots, u_r$ is a choice of basis of $V^{(0)} / C(V)$. We assume without loss of generality that $u_i$ is homogeneous of conformal weight $\D(u_i)$. The ODE is obtained explicitly by reduction of \eqref{eq:DE.y} modulo the relations \eqref{eq:twisted.CB.analytic}. We notice that the ring $R^{(N)}$ as well as the differential equation and the relations are homogeneous if we assign degrees as follows
\begin{align*}
\deg(\wtil G_k(q)) &= k & \deg(u) &= \D_u \\
\deg(P_k(y^c, q)) &= k, & \deg(\partial_\al ) & = 1.
\end{align*}
It follows that the entries $A_{ij}$ of $A$ are homogeneous in the sense that $\deg(A_{ij}) = 1 + \D(u_j) - \D(u_i)$ for each $i, j$. It is convenient, then, for us to say that a matrix $B \in \Mat_{r \times r}(R^{(N)})$ has degree $a$ if $\deg(B_{ij}) = a + \D(u_j) - \D(u_i)$ for each $i, j$. Then degree is clearly multiplicative. The equation $S' = A S$ is homogeneous, as are the equations $S'' = (A' + A^2) S$, $S^{(3)} = (A'' + 3AA' + A^3) S$, etc., obtained by successive differentiation. Here we have used that $\partial_\al$ acts in $R^{(N)}$ with degree $+1$. Let us write these equations in general as $S^{(i)} = P_i S$. For fixed $j \in \{1, \ldots, r\}$ we may eliminate variables to obtain an ODE of order $N$ for $S_j$, namely
\begin{align*}
\sum_{i=0}^N f_i S_j^{(i)} = 0.
\end{align*}
Later we will normalise so that $f_{N} = 1$, but for now we may clear denominators and assume $f_i \in R^{(N)}$ for all $i$. The coefficients $f_i$ must then satisfy
\[
\sum_{j=0}^N (P_i)_{jk} f_i = 0, \quad k = 1,\ldots, r.
\]
Since $\deg((P_i)_{jk}) = i + \D(u_k) - \D(u_j)$, we have $\deg(f_i) = \varepsilon - i$ for some $\varepsilon$.

The final observation is that for $f \in R^{(N)}$ of degree $k$, the poles of $f(\al, \tau)$ on the lattice $N\al \in \Z + \Z \tau$ are of order at most $k$. It follows, then, that after normalisation so that $f_N = 1$, the function $f_i$ has poles of order at most $N-i$. The ODE is thus Fuchsian.
\end{proof}

The spaces $B$ and $T$ carry compatible actions of the groups $SL_2(\Z)$ and $\Z^2$, and thus of their semidirect product the Jacobi group $J = SL_2(\Z) \ltimes \Z^2$. The fibres over points in the same $J$-orbit are isomorphic. These actions are well known and are as follows. For $A = \left(\begin{smallmatrix} a & b \\ c & d \\ \end{smallmatrix}\right) \in SL_2(\Z)$ and $(\al, \tau) \in B$ we have
\[
A \cdot (\al, \tau) = \left( \frac{\al}{c\tau+d}, \frac{a\tau+b}{c\tau+d} \right).
\]
Writing $(\al', \tau') = A \cdot (\al, \tau)$, the action of $A$ on $T$ is
\[
A \cdot (\al, \tau, v, z) = \left( \al', \tau', e^{2\pi i \frac{c}{c\tau+d} \alpha z} v, \frac{z}{c\tau+d} \right).
\]
For $(m, n) \in \Z^2$ the action is
\begin{align*}
(m, n) \cdot (\al, \tau, v, z) = (\al + m\tau + n, \tau, e^{-2\pi i m z} v, z).
\end{align*}
It is straightforward to verify that
\[
(E, \CL)_{A \cdot (\al, \tau)} \cong (E, \CL)_{(\al, \tau)} \quad \text{and} \quad (E, \CL)_{(m, n) \cdot (\al, \tau)} \cong (E, \CL)_{(\al, \tau)},
\]
for all $A \in SL_2(\Z)$ and $(m, n) \in \Z^2$.

The main result of this section is that the connection $\nabla$ is $J$-equivariant, and that flat sections of $(\CC, \nabla)$ transform as Jacobi forms of a certain index $\kappa$ which is determined by the charged conformal vertex algebra $(V, \omega, h)$. In particular if $u \in V$ is primary of conformal weight $k$, then the flat sections $S_M(u, \al, \tau)$ of $\CC$ transform as a vector-valued Jacobi form of weight $k$ and $\kappa$ is the index defined as in Section \ref{sec:vertex}, namely $h_{1}h = 2\kappa \vac$. The precise statement is now as follows.

\begin{prop}\label{prop:connec.is}
Let $S(u, \al, \tau)$ be a flat section of $(\CC, \nabla)$.
\begin{enumerate}[label={(\arabic*)},ref={\thecor~(\arabic*)}]
\item\label{prop:connec.is.modular} Let $A \in SL_2(\Z)$ and $(\al', \tau') = A \cdot (\al, \tau)$. If we define $S'(u, \al, \tau)$ by
\begin{align*}
S'(u, \al', \tau') = e^{2\pi i \kappa \frac{c}{c\tau+d} \al^2} S( e^{2\pi i\frac{c\al}{c\tau+d} h_{[1]}}(c\tau+d)^{L_{[0]}} u, \al, \tau),
\end{align*}
then $S'(u, \al, \tau)$ is also a flat section of $(\CC, \nabla)$.

\item\label{prop:connec.is.elliptic} Let $(m, n) \in \Z^2$ and $(\al', \tau') = (m, n) \cdot (\al, \tau)$. If we define $S'(u, \al, \tau)$ by
\begin{align*}
S'(u, \al', \tau') = e^{ -2\pi i \kappa (m^2 \tau + 2m \al) } S(e^{-2\pi i m h_{[1]}} u, \al, \tau),
\end{align*}
then $S'(u, \al, \tau)$ is also a flat section of $(\CC, \nabla)$.
\end{enumerate}
\end{prop}
The proof of this proposition is essentially by direct calculation, given below in Appendices \ref{app.mod} and \ref{app.ell}. We now indicate the idea behind the calculations. Let $A \in SL_2(\Z)$ and let us consider the isomorphism $f : (E, \CL)_{(\al, \tau)} \rightarrow (E, \CL)_{A \cdot (\al, \tau)}$ given by
\[
f(v, z) = \left( e^{2\pi i \frac{c}{c\tau+d} \alpha z} v, \frac{z}{c\tau+d} \right).
\]
This isomorphism induces a nontrivial identification between the fibres $\CV_0$ on the two curves $E_{\tau}$ and $E_{A \cdot \tau}$, as per the associated bundle construction described in Section \ref{sec:chiral}. Explicitly, for each $u \in V$, we have in $\CV_0$ the equality
\[
\left( e^{2\pi i \frac{c\al}{c\tau+d} z} v, \frac{z}{c\tau+d}, u \right) = \left( v, z, e^{2\pi i \frac{c\al}{c\tau+d} h_{[1]}} (c\tau+d)^{L_{[0]}} u \right).
\]

Let us now consider a flat local section $S(u, \al, \tau)$ of $(\CC, \nabla)$, more precisely $S((v, z, u), \al, \tau)$. We define a new section $S'((v, z, u), \al, \tau)$ by
\begin{align}\label{eq:mod.xform}
\begin{split}
S'((v', z', u), \al', \tau')
&= e^{2\pi i \kappa \frac{c}{c\tau+d} \al^2} S((v', z', u), \al, \tau) \\
&= e^{2\pi i \kappa \frac{c}{c\tau+d} \al^2} S((v, z, e^{2\pi i\frac{c\al}{c\tau+d} h_{[1]}}(c\tau+d)^{L_{[0]}} u), \al, \tau),
\end{split}
\end{align}
where $(\al', \tau', v', z') = A \cdot (\al, \tau, v, z)$. The proof of Proposition~\ref{prop:connec.is.modular} consists in using \eqref{eq:DE.y} and \eqref{eq:DE.q} to prove the following:
\begin{align}
\frac{\partial}{\partial\al'} S((v', z', u), \al', \tau') = S((v', z', \res_t \wp_1(t, \tau) h[t] u), \al', \tau' ) \label{eq:DE.xform.y} \\
\text{and} \quad 2\pi i \frac{\partial}{\partial\tau'} S((v', z', u), \al', \tau') = S((v', z', \res_t \wp_1(t, \tau) L[t] u), \al', \tau' ). \label{eq:DE.xform.q}
\end{align}
The verification is straightforward but lengthy and we relegate it to Appendix \ref{app.mod}.

Similarly, for a flat local section $S((v, z, u), \al, \tau)$ of $(\CC, \nabla)$ and $(m, n) \in \Z^2$, we define a new section $S'((v, z, u), \al, \tau)$ by
\begin{align}\label{eq:ell.xform}
\begin{split}
S'((v', z', u), \al', \tau') &= e^{ -2\pi i \kappa (m^2 \tau + 2m \al) } S((v', z', u), \al, \tau) \\
&= e^{ -2\pi i \kappa (m^2 \tau + 2m \al) } S((v, z, e^{-2\pi i m h_{[1]}} u), \al, \tau)
\end{split}
\end{align}
where $(\al', \tau', v', z') = (m, n) \cdot (\al, \tau, v, z)$. The proof of Proposition \ref{prop:connec.is.elliptic} consists again in using \eqref{eq:DE.y} and \eqref{eq:DE.q} to prove \eqref{eq:DE.xform.y} and \eqref{eq:DE.xform.q}. The computations are significantly easier in this case, and are given in Appendix \ref{app.ell}.

Using the action of $SL_2(\Z)$ on $X$, we now show that the trace functions can be analytically continued along paths from $Y$ to any point of $X$.
\begin{lemma}\label{lem:G.action}
Every $SL_2(\Z)$-orbit in $X$ meets $Y$. Furthermore for every $x \in X$ there exists a sequence of elements $A_0 = I, A_1, \ldots, A_k \in SL_2(\Z)$ such that $A_i Y$ intersects $A_{i-1}Y$ nontrivially, and $x \in A_k Y$.
\end{lemma}

In fact for almost all points we can take $k=1$, and for the remainder $k=2$ suffices.

\begin{proof}
Let $x = (\al, \tau) \in X$. Let $(\al', \tau') = A \cdot x = (\frac{\al}{c\tau+d}, \frac{a\tau+b}{c\tau+d})$. An easy computation, using $\det(A)=1$, shows that
\[
\im \al' = \im \frac{\al}{c\tau+d} = \frac{\im(\al(c\overline\tau+d))}{|c\tau+d|^2} \quad \text{and} \quad \im \tau' = \im \frac{a\tau+b}{c\tau+d} = \frac{\im \tau}{|c\tau+d|^2}.
\]
We would like to impose conditions on $(c, d)$ so that
\[
\frac{\im \al'}{\im \tau'} = c \frac{\im(\al \overline\tau)}{\im \tau} + d \frac{\im \al}{\im \tau} \in (-1/N, 0).
\]
In fact it is possible to find appropriate $c, d \in \Z$ unless
\begin{align}\label{eq:bad.alpha.condition}
\frac{\im(\al \overline\tau)}{\im \tau}, \frac{\im \al}{\im \tau} \in \frac{1}{N} \Z.
\end{align}
If we now write $\tau = u + iv$ and $\al = x + iy$ (where $u, v, x, y \in \R$), then \eqref{eq:bad.alpha.condition} becomes
\[
\frac{y}{v} = \frac{n}{N} \quad \text{and} \quad \frac{uy-xv}{v} = \frac{-m}{N}
\]
for some $m, n \in \Z$. Rearranging gives
\[
\al = x + iy = \frac{m+nu}{N} + i\frac{nv}{N} = \frac{1}{N}(m + n\tau).
\]
So indeed for all $x \in X$, there is some $A \in SL_2(\Z)$ such that $A \cdot x$ lies in $Y$ the domain of convergence of the trace functions.

We pass to the second claim now. Whenever $(c, d)$ can be chosen with $c \neq 0$ then $A \cdot Y$ and $Y$ intersect nontrivially. When $c=0$ we have $d = \pm 1$ and $(\al', \tau') = (\pm \al, \tau + n)$ for some $n$. But then any other image of $Y$ (under an element of $SL_2(\Z)$ for which $c \neq 0$) meets both $Y$ and $A \cdot Y$ nontrivially.
\end{proof}
Let $(E, \CL)$ be the fibre of $T \rightarrow B$ over a point $(\al, \tau) \in X$. Since, by Lemma \ref{lem:G.action}, $(\al, \tau)$ has an image $(\al', \tau') \in Y$, we can conclude that charged conformal blocks on $(E, \CL)$ are expressible in terms of trace functions on stable $V$-modules. More precisely we have the following corollary.
\begin{cor}\label{cor:finite.everywhere}
Let $V$ be a charged conformal vertex algebra.
\begin{enumerate}
\item Suppose that $V$ is finitely strongly generated and stably quasi-lisse. Then the vector space $\CC((E, \CL)_{(\al, \tau)}, V)$ is finite dimensional whenever $N\al \notin \Z + \Z\tau$.

\item Suppose furthermore that $V$ is stably rational. Then $\CC((E, \CL)_{(\al, \tau)}, V)$, is spanned by the trace functions $S_{W^i}(u, \al', \tau')$ on the irreducible stable $V$-modules $W^1, \ldots, W^r$, for some choice of $(\al', \tau') = A \cdot (\al, \tau)$.
\end{enumerate}
\end{cor}

\begin{proof}
By Lemma \ref{lem:G.action} any point $(\al, \tau) \in X$ is in the orbit under $SL_2(\Z)$ of a point $(\al', \tau') \in Y$. {The corresponding fibres $(E, \CL)_{(\al, \tau)}$ and $(E, \CL)_{(\al', \tau')}$ are isomorphic.} The first claim then follows immediately from Corollary \ref{cor:informal.finiteness} and the second claim from Theorem \ref{thm:main}.
\end{proof}


We now show the consequences of these results for the example of simple affine vertex algebras at admissible level. As before we consider the charged conformal vertex algebra $V = (L_k(\g), \omega, \rho)$, as in Proposition {\ref{prop:quasi-lisse.enough}} where we showed that $V$ is stably quasi-lisse, and Proposition \ref{prop:admissible-affine-stable}, where we showed that $V$ is stably rational also. By Proposition \ref{prop:admissible-affine-stable} and the main result of \cite{A12-2}, we see that the category of stable $V$-modules is semisimple, with simple objects precisely the admissible $\what\g$-modules $L_k(\la)$ as $\la$ ranges over $\textup{Prin}^k$ the set of principal/coprincipal admissible weights of level $k$.

We now obtain the following as a consequence of our main results. We write $L_k(\g)$ for $(L_k(\g), \omega, \rho)$.
\begin{prop}
Let $k$ be an admissible level for the simple Lie algebra $\g$, then:
\begin{enumerate}
\item The trace functions $\tr_{L(\la)} u_0 y^{h_0} q^{L_0 - c/24}$, as $\la$ ranges over the set $\on{Adm}^k$ of admissible weights of level $k$, span the space $\CC((E, \CL)_{(\al, \tau)}, L_k(\g))$ for all $\al \notin \Z + \Z\tau$;

\item for any smooth elliptic curve $E$ and degree $0$ holomorphic line bundle $\CL$ over $E$, the dimension of $\CC((E, \CL), L_k(\g))$ equals the cardinality of $\on{Adm}^k$, namely $|\check{P}/u\check{P}| \cdot |P_+^{p, \on{reg}}|$ if $k$ is principal admissible, and $|\check{P}/uP| \cdot |{}^{L}\check{P}_+^{p, \on{reg}}|$ if coprincipal;

\item the characters $\tr_{L(\la)} y^{h_0} q^{L_0 - c/24}$, as $\la$ ranges over $\on{Adm}^k$, constitute a vector-valued Jacobi form of weight $0$ and index $\kappa = k h^\vee \dim(\g) / 24$.
\end{enumerate}
\end{prop}

{{Here $P$ and $\check{P}$ are the weight and coweight lattices of $\g$, and $p,u$ are as in the definition of admissible level $k=-h^\vee+p/u$ above. We recall that $\nu\in P$ is \emph{dominant integral} if $\left<\nu,\al^\vee\right>\in\Z_{\geq 0}$ for all simple roots $\al$, and is \emph{regular} if $\left<\la,\al^\vee\right>\neq 0$ for all such $\al$. With this convention we set
\begin{equation*}
\begin{aligned}
P_+^{p}:=\{\nu \in P\mid{}& \text{$\nu$ is dominant integral and $(\nu,\theta)=p$}\}.
\end{aligned}
\end{equation*}
and $P_+^{p,\on{reg}}$ its subset of regular weights. If $k$ is coprincipal, we write ${}^{L}\check{P}_+^{p,\on{reg}}$ for the analogous set of regular dominant integral coweights for the Langlands dual root system ${}^{L}\Delta$. The parametrisation of admissible weights \cite{KW89} yields bijections $\on{Adm}^k\cong (\check{P}/u\check{P})\times P_+^{p,\on{reg}}$ in the principal case {\cite[Section 1.5]{FKW}}, and $\on{Adm}^k\cong (\check{P}/uP)\times {}^{L}\check{P}_+^{p,\on{reg}}$ in the coprincipal case {\cite[Proposition 3.3]{A12-2}}.}}

\begin{proof}
By Propositions \ref{prop:quasi-lisse.enough} and \ref{prop:admissible-affine-stable} the charged conformal vertex algebra $(L_k(\g), \omega, \rho)$ is stably quasi-lisse and stably rational, and its irreducible stable modules are precisely $L(\la)$ as $\la$ ranges over $\on{Adm}^k$. The first claim follows from Theorem \ref{thm:main} (with $N=1$). The claim about the dimension of $\CC((E, \CL), L_k(\g))$ then follows from the parametrisation of admissible weights. The Jacobi invariance follows from Proposition \ref{prop:connec.is}, and the specific value for the index comes from the Freudenthal-de Vries strange formula $|\rho|^2 = h^\vee \dim(\g) / 12$.
\end{proof}
In particular the index is generally fractional and can be either positive or negative, according to the sign of the level $k$.

\section{Examples of MLDEs}\label{sec:examples}

We have discussed a number of examples already throughout the text, in particular simple affine vertex algebras at admissible levels, and quasi-lisse $W$-algebras obtained as quantum Drinfeld-Sokolov reductions. We have not explicitly discussed modular linear differential equations (MLDEs), though there is interest in determining them explicitly in examples. In this section we briefly examine the explicit form of the connection $\nabla$ on $\CC$ and the associated MLDEs in two simple cases.


%

\subsection{The rational vertex algebra $L_1(\sll_2)$}

We take $V = L_1(\sll_2)$ with $\omega$ the Sugawara conformal vector and $h = \rho = \al/2$ the Weyl vector. In particular $[\rho, E] = E$ and $[\rho, F] = -F$. The central charge of $(V, \omega, \rho)$ is $c=1$ and the index is $\kappa = 1/4$. The vertex algebra $V$ is rational and has two irreducible modules $W^0 = V$ and another which we denote $W^1$. The normalised characters of these $V$-modules, defined by
\[
\chi_i(\al, \tau) = \tr_{W^i} y^{h_0} q^{L_0-c/24},
\]
are given by the formulas $\chi_i(\al, \tau) = \theta_{i0}(\al, \tau) / \eta(\tau)$ where {\cite[p. 17]{Mumford.Tata.1}}
\begin{align*}
\theta_{i0}(\al, \tau) = \sum_{n \in \Z} q^{(n+i/2)^2/2} e^{2\pi i (n+i/2)z}.
\end{align*}
These characters are the restrictions to $u=\vac$ of two charged conformal blocks $S_{W^i}(u, y, q)$.

On the other hand let $S(u, y, q)$ be a charged conformal block for $(V, \omega, h)$, and write $S_{0}$ and $S_1$ for its restrictions to $u = \vac$ and $u = h$, respectively. In this example $V^{(0)} / C(V)$ is $2$ dimensional, spanned by $\vac$ and $h$. The equations \eqref{eq:DE.y} and \eqref{eq:DE.q} therefore reduce to a $2 \times 2$ holonomic system on $S_{0}$ and $S_1$. We find in fact that
\begin{align}\label{eq:sl2.y.eq}
y \frac{\partial}{\partial y} \left[ \begin{array}{c} S_0 \\ S_1 \\ \end{array} \right] = \left[ \begin{array}{cc} 0 & 1 \\ \tfrac{1}{16}(4Q_2-4Q_1^2-E_2) & Q_1 \\ \end{array} \right] \left[ \begin{array}{c} S_0 \\ S_1 \\ \end{array} \right]
\end{align}
and
\begin{align}\label{eq:sl2.q.eq}
q \frac{\partial}{\partial q} \left[ \begin{array}{c} S_0 \\ S_1 \\ \end{array} \right] = \left[ \begin{array}{cc} \tfrac{1}{48}(12Q_2-12Q_1^2-E_2) & -Q_1 \\ \tfrac{1}{16}(4Q_1^3-12Q_1Q_2+8Q_3-E_2Q_1) & \tfrac{1}{48}(36Q_1^2-36Q_2+E_2) \\ \end{array} \right] \left[ \begin{array}{c} S_0 \\ S_1 \\ \end{array} \right]
\end{align}
{{To avoid cumbersome factors of $2\pi i$ in these equations, we have set
\[
Q_k(y, q) = \frac{1}{2}\delta_{k,1} + (2\pi i)^{-k} P^+_k(y, q),
\]
where $P^+_k(y,q)$ is as in \eqref{eq:zhu.3.9} with $w=y$ (a series in $\C[y][[q]]$, not the three-variable series $P_1^+(z,y,q)$ of \eqref{eq:Ppm.def}). Here $E_k(q) = -(2\pi i)^{-k}(k!/B_k)$ is the Eisenstein series normalised with constant coefficient $1$.}}


In particular we deduce from \eqref{eq:sl2.y.eq} the following MLDE in $y$ for $u(y, q) = S_0(\al, \tau)$:
\begin{align}\label{eq:MLDE.level.1}
\begin{split}
\left(y\frac{\partial}{\partial y}\right)^2 u
- Q_1 \left(y\frac{\partial}{\partial y}\right) u
&-\left(\frac{1}{4} Q_2 - \frac{1}{4} Q_1^2 - \frac{1}{16} E_2 \right) u = 0.
\end{split}
\end{align}

\subsection{The admissible level vertex algebra $L_{-4/3}(\sll_2)$}

We take $V = L_{-4/3}(\sll_2)$ with $\omega$ and $h$ as above. The central charge of $(V, \omega, \rho)$ is $c=-6$ and the index is $\kappa = -1/3$. This vertex algebra is stably rational with three stable modules $L_k(\la)$, where $\la$ takes the values $0$, $- \frac{2}{3}\varpi$ and $- \frac{4}{3}\varpi$, where $\varpi = \al/2$ is the fundamental weight. (In fact $\rho = \varpi$ since $\g = \sll_2$.)

The characters of these modules are well known from the work of Kac and Wakimoto \cite{KW89, KW2017}. To express them easily we may introduce the theta functions
\begin{align*}
\theta_j(y, q) = \sum_{n \in \Z} (-1)^n y^{n+1/2-j/3} q^{(n+1/2-j/3)^2/2}, \qquad j=0,1,2.
\end{align*}
Then the three normalised characters are $\chi_i(y,q) = {\theta_i(y, q^3)}/{\theta_0(y,q)}$.

In this example $V^{(0)} / C(V)$ is $3$-dimensional, spanned by the images of the vectors $\vac$, $h$ and $h_{(-1)}h$. If we write $S_i(\al, \tau)$ for $i=0,1,2$ for the evaluations of a charged conformal block $S(u, \al, \tau)$ on $\vac$, $h$ and $h_{(-1)}h$, respectively, then the equations \eqref{eq:DE.y} and \eqref{eq:DE.q} reduce to a $3 \times 3$ holonomic system on $(S_0, S_1, S_2)$. The explicit form for this system is quite tedious to write out in full, but we obtain from it the following MLDE in $y$ for $u(y, q) = S_0(\al, \tau)$:
\begin{align}\label{eq:MLDE.level.-4/3}
\begin{split}
\left(y\frac{\partial}{\partial y}\right)^3 u
- 2Q_1 \left(y\frac{\partial}{\partial y}\right)^2 u
&-\left(\frac{10}{3} Q_2 - \frac{4}{3} Q_1^2 + \frac{1}{9} E_2 \right) \left(y\frac{\partial}{\partial y}\right) u \\
&+ \left(\frac{2}{27} E_2 Q_1 - \frac{52}{27} Q_3 + \frac{20}{9} Q_1 Q_2 - \frac{8}{27} Q_1^3 \right) u = 0.
\end{split}
\end{align}
The three normalised characters $\chi_i(\alpha, \tau)$ satisfy this MLDE.

\appendix

\section{Verification of invariance under the Jacobi group}

Before beginning the calculations, we record some commutator relations that will be used.
\begin{lemma}\label{lem:ABcomm}
For endomorphisms $A$ and $B$ of a vector space and any $k \in \Z_+$, we have the equality
\[
B A^k = \sum_{j \in \Z_+} (-1)^j \binom{k}{j} A^{k-j} \ad_A(B).
\]
\end{lemma}
Applying Lemma \ref{lem:ABcomm} to $A = h_{[1]}$ and $B = L_{([m])}$ and using the relations
\begin{align*}
\ad_{h_{[1]}}(L_{([m])}) &= h_m \\
\ad_{h_{[1]}}^2(L_{([m])}) &= \ad_{h_{[1]}}(h_{[m]}) = \delta_{m,-1} \left<h, h\right> = \delta_{m,-1} \left<h, h\right> \\
\ad_{h_{[1]}}^{\geq 3}(L_{([m])}) &= 0,
\end{align*}
yields
\begin{align*}
L_{([m])} h_{[1]}^k = h_{[1]}^k L_{([m])} - k h_{[1]}^{k-1} h_{[m]} + \delta_{m,-1} \binom{k}{2} \left<h, h\right> h_{[1]}^{k-2}.
\end{align*}
From this, in turn, we obtain
\begin{align*}
\sum_{k \in \Z_+} \frac{\la^k}{k!} L_{([m])} h_{[1]}^k
&= \sum_{k \in \Z_+} \frac{\la^k}{k!} h_{[1]}^k L_{([m])} - \sum_{k \in \Z_+} \frac{\la^k}{k!} k h_{[1]}^{k-1} h_{[m]} + \delta_{m,-1} \sum_{k \in \Z_+} \frac{\la^k}{k!} \binom{k}{2} \left<h, h\right> h_{[1]}^{k-2} \\
&= e^{\la h_{[1]}} L_{([m])} - \la e^{\la h_{[1]}} h_{[m]} + \delta_{m,-1}  \la^2\frac{\left<h, h\right>}{2} e^{\la h_{[1]}}.
\end{align*}
Rewriting in terms of the field $L[t] = \sum_m L_{([m])} t^{-m-1}$ we get
\begin{align}
L[t] e^{\la h_{[1]}}
&= \sum_{m \in \Z} L_{([m])} e^{\la h_{[1]}} t^{-m-1} \nonumber \\
&= e^{\la h_{[1]}} \sum_{m \in \Z} \left( L_{([m])} - \la h_{[m]} + \delta_{m,-1} \la^2\frac{\left<h, h\right> }{2} \right) t^{-m-1} \nonumber \\
&= e^{\la h_{[1]}} L[t] - \la e^{\la h_{[1]}} h[t] + \la^2\frac{\left<h, h\right> }{2}e^{\la h_{[1]}}. \label{eq:L[t].exp.comm}
\end{align}

Applying Lemma \ref{lem:ABcomm} to $A = h_{[1]}$ and $B = h_{[m]}$ and using the relations
\begin{align*}
\ad_{h_{[1]}}(h_{[m]}) &= \delta_{m,-1} \left<h, h\right> \\
\ad_{h_{[1]}}^{\geq 2}(h_{[m]}) &= 0,
\end{align*}
yields
\begin{align*}
h_{[m]} h_{[1]}^k = h_{[1]}^k h_{[m]} - \delta_{m,-1} k \left<h, h\right> h_{[1]}^{k-1}.
\end{align*}
we can say that
\begin{align}
h[t] e^{ \la h_{[1]} } &= e^{ \la h_{[1]} } h[t] - \left<h, h\right> \sum_{k=0}^\infty \frac{1}{(k-1)!} \la^{k} h_{[1]}^{k-1} \nonumber \\
&= e^{ \la h_{[1]} } h[t] - \la \left<h, h\right> e^{ \la h_{[1]} }.  \label{eq:h[t].exp.comm}
\end{align}

\subsection{Verification of invariance under modular transformations}\label{app.mod}

Let $A \in SL_2(\Z)$ and $(\al', \tau') = A \cdot (\al, \tau)$ as in Section \ref{sec:modularity}. The partial derivatives in $\tau$ and $\al$ transform as follows: If $(\al', \tau') = A(\al, \tau)$ then
\begin{align*}
\frac{\partial}{\partial \al'} &= (c\tau+d) \frac{\partial}{\partial \al} \\
\frac{\partial}{\partial \tau'} &= (c\tau+d)^2 \frac{\partial}{\partial \tau} + c\al(c\tau+d) \frac{\partial}{\partial \al}
\end{align*}

\subsubsection{Modular invariance of the $y$-equation}\label{sec:inv.DE.y}

Let $S((v, z, u), \al, \tau)$ be a flat section of $\CC$ and let $S'((v', z', u), \al', \tau')$ be given by equation \eqref{eq:mod.xform}. We need to show \eqref{eq:DE.xform.y}. We compute
\begin{align*}
& \frac{\partial}{\partial\al'}S'((v', z', u), \al', \tau') \\
= {} & (c\tau+d) \frac{\partial}{\partial\al}S'((v', z', u), \al', \tau') \\
= {} & (c\tau+d) \frac{\partial}{\partial\al}\left(e^{2\pi i \kappa \frac{c}{c\tau+d} \al^2} S((v, z, e^{2\pi i\frac{c\al}{c\tau+d} h_{[1]}}(c\tau+d)^{L_{[0]}} u), \al, \tau)\right) \\
= {} & (c\tau+d) \frac{\partial}{\partial\al}\left(e^{2\pi i \kappa \frac{c}{c\tau+d} \al^2} (c\tau+d)^{\nabla_u} \sum_{j=0}^\infty \frac{1}{j!} \left(2\pi i\frac{c\al}{c\tau+d}\right)^j S((v, z,  h_{[1]}^j u), \al, \tau)\right) \\
= {} & (c\tau+d) \sum_{j=0}^\infty \frac{1}{j!} \frac{\partial}{\partial\al}\left(e^{2\pi i \kappa \frac{c}{c\tau+d} \al^2} (c\tau+d)^{\nabla_u} \left(2\pi i\frac{c\al}{c\tau+d}\right)^j \right) S((v, z,  h_{[1]}^j u), \al, \tau) \tag{D}\label{eqn:y.D} \\
& + e^{2\pi i \kappa \frac{c}{c\tau+d} \al^2}(c\tau+d)^{\nabla_u+1} \sum_{j=0}^\infty \frac{1}{j!} \left(2\pi i\frac{c\al}{c\tau+d}\right)^j \frac{\partial}{\partial\al} S((v, z,  h_{[1]}^j u), \al, \tau). \tag{C}\label{eqn:y.C}
\end{align*}
We simplify \eqref{eqn:y.D} first and then return to \eqref{eqn:y.C} later. We have
\begin{align*}
\text{\eqref{eqn:y.D}} = {} & (c\tau+d)^{\nabla_u+1} \sum_{j=0}^\infty \frac{1}{j!} \frac{d}{d\al}\left(e^{2\pi i \kappa \frac{c}{c\tau+d} \al^2} \left(2\pi i\frac{c\al}{c\tau+d}\right)^j \right) S((v, z,  h_{[1]}^j u), \al, \tau) \\
= {} & (c\tau+d)^{\nabla_u+1} e^{2\pi i \kappa \frac{c}{c\tau+d} \al^2} \sum_{j=0}^\infty \frac{1}{j!} \frac{d}{d\al}\left( 2\pi i\frac{c\al}{c\tau+d} \right)^j S((v, z,  h_{[1]}^j u), \al, \tau) \tag{D1}\label{eqn:y.D1} \\
&+ (c\tau+d)^{\nabla_u+1} \frac{d}{d\al} (e^{2\pi i \kappa \frac{c}{c\tau+d} \al^2}) S((v, z, e^{2\pi i\frac{c\al}{c\tau+d} h_{[1]} } u), \al, \tau). \tag{D2}\label{eqn:y.D2}
\end{align*}
We have
\begin{align*}
\text{\eqref{eqn:y.D2}} = {} & (c\tau+d) 2\pi i \kappa \frac{c}{c\tau+d} (2\al) e^{2\pi i \kappa \frac{c}{c\tau+d} \al^2} S((v, z, e^{2\pi i\frac{c\al}{c\tau+d} h_{[1]} } (c\tau+d)^{L_{[0]}} u), \al, \tau) \\
= {} & 2\pi i \kappa {c} (2\al) S'((v', z', u), \al', \tau').
\end{align*}
And we have
\begin{align*}
\text{\eqref{eqn:y.D1}}
= {} & (c\tau+d)^{\nabla_u+1} e^{2\pi i \kappa \frac{c}{c\tau+d} \al^2} \sum_{j=0}^\infty \frac{1}{j!} \frac{2\pi i c}{c\tau+d} j (2\pi i\frac{c\al}{c\tau+d})^{j-1} S((v, z,  h_{[1]}^j u), \al, \tau) \\
= {} & {2\pi i c} (c\tau+d)^{\nabla_u} e^{2\pi i \kappa \frac{c}{c\tau+d} \al^2} \sum_{j=0}^\infty \frac{1}{(j-1)!} \left(2\pi i\frac{c\al}{c\tau+d}\right)^{j-1} S((v, z,  h_{[1]}^j u), \al, \tau) \\
= {} & {2\pi i c} (c\tau+d)^{\nabla_u} e^{2\pi i \kappa \frac{c}{c\tau+d} \al^2} S((v, z, e^{2\pi i\frac{c\al}{c\tau+d} h_{[1]} } h_{[1]} u), \al, \tau) \\
= {} & {2\pi i c} (c\tau+d) e^{2\pi i \kappa \frac{c}{c\tau+d} \al^2} S((v, z, e^{2\pi i\frac{c\al}{c\tau+d} h_{[1]} }(c\tau+d)^{L_{[0]}} h_{[1]} u), \al, \tau) \\
= {} & {2\pi i c} (c\tau+d) S'((v', z', h_{[1]} u), \al', \tau').
\end{align*}

Now we turn to \eqref{eqn:y.C}, rewriting it using \eqref{eq:DE.y} as
\begin{align*}
\text{\eqref{eqn:y.C}}
= {} & e^{2\pi i \kappa \frac{c}{c\tau+d} \al^2}(c\tau+d)^{\nabla_u+1} \sum_{j=0}^\infty \frac{1}{j!} \left(2\pi i\frac{c\al}{c\tau+d}\right)^j S((v, z, \res_t \wp_1(t, \tau) h[t] h_{[1]}^j u), \al, \tau) \\
= {} &  e^{2\pi i \kappa \frac{c}{c\tau+d} \al^2}(c\tau+d)^{\nabla_u+1} S((v, z, \res_t \wp_1(t, \tau) h[t] e^{2\pi i\frac{c\al}{c\tau+d} h_{[1]}} u), \al, \tau)
\end{align*}
To proceed further we commute $h[t]$ with $e^{h_{[1]}}$ using equation \eqref{eq:h[t].exp.comm} (with $\la = 2\pi i\frac{c\al}{c\tau+d}$). We have
\begin{align*}
\text{\eqref{eqn:y.C}} = {} & e^{2\pi i \kappa \frac{c}{c\tau+d} \al^2}(c\tau+d)^{\nabla_u+1} S((v, z, \res_t \wp_1(t, \tau) h[t] e^{2\pi i\frac{c\al}{c\tau+d} h_{[1]}} u), \al, \tau) \\
= {} & e^{2\pi i \kappa \frac{c}{c\tau+d} \al^2}(c\tau+d)^{\nabla_u+1} S((v, z, \res_t \wp_1(t, \tau) e^{2\pi i\frac{c\al}{c\tau+d} h_{[1]}} h[t] u), \al, \tau) \tag{C1}\label{eqn:y.C1} \\
&- 2\pi i\frac{c\al \left<h, h\right>}{c\tau+d} e^{2\pi i \kappa \frac{c}{c\tau+d} \al^2}(c\tau+d)^{\nabla_u+1} S((v, z, \res_t \wp_1(t, \tau)e^{ 2\pi i\frac{c\al}{c\tau+d} h_{([1])}} u), \al, \tau). \tag{C2}\label{eqn:y.C2}
\end{align*}
We continue
\begin{align*}
\text{\eqref{eqn:y.C2}} = {} & -2\pi i{c\al \left<h, h\right>} e^{2\pi i \kappa \frac{c}{c\tau+d} \al^2}(c\tau+d)^{\nabla_u} S((v, z, e^{ 2\pi i\frac{c\al}{c\tau+d} h_{[1]}} u), \al, \tau) \\
= {} & -2\pi i{c\al \left<h, h\right>} e^{2\pi i \kappa \frac{c}{c\tau+d} \al^2} S((v, z, e^{ 2\pi i\frac{c\al}{c\tau+d} h_{[1]}} (c\tau+d)^{L_{[0]}} u), \al, \tau) \\
= {} & -2\pi i{c\al \left<h, h\right>} S'((v', z', u), \al', \tau').
\end{align*}
Before simplifying \eqref{eqn:y.C1} we recall that $\wp_1(t, \tau) = \sum_{i=0}^\infty P^{i+1}(\tau) t^{i}$, where for all $k$ we have
\[
P^k(\tau') = (c\tau+d)^k P^k(\tau),
\]
except when $k=2$ in which case $P^2(\tau) = -G_2(\tau)$ and
\[
G_2(\tau') = (c\tau+d)^2 G_2(\tau) - 2\pi i c (c\tau+d).
\]
We apply this in the following simplifications:
\begin{align*}
\text{\eqref{eqn:y.C1}} = {} & e^{2\pi i \kappa \frac{c}{c\tau+d} \al^2} (c\tau+d)^{\nabla_u+1} \sum_{i} P^{i+1}(\tau) S((v, z, e^{2\pi i\frac{c\al}{c\tau+d} h_{[1]}} h_{[i]}u), \al, \tau) \\
= {} & e^{2\pi i \kappa \frac{c}{c\tau+d} \al^2} (c\tau+d)^{\nabla_u+1} \sum_{i} (c\tau+d)^{-i-1} P^{i+1}(\tau') S((v, z, e^{2\pi i\frac{c\al}{c\tau+d} h_{[1]}} h_{[i]}u), \al, \tau) \\
&- e^{2\pi i \kappa \frac{c}{c\tau+d} \al^2} (c\tau+d)^{\nabla_u+1} 2\pi i c (c\tau+d)^{-1} S((v, z,  e^{2\pi i\frac{c\al}{c\tau+d} h_{[1]}} h_{[1]}u), \al, \tau) \\
= {} & e^{2\pi i \kappa \frac{c}{c\tau+d} \al^2} \sum_{i} P^{i+1}(\tau') S((v, z, e^{2\pi i\frac{c\al}{c\tau+d} h_{[1]}} (c\tau+d)^{L_{[0]}} h_{[i]}u), \al, \tau) \\
&- 2\pi i c (c\tau+d) e^{2\pi i \kappa \frac{c}{c\tau+d} \al^2} S((v, z, e^{2\pi i\frac{c\al}{c\tau+d} h_{[1]}} (c\tau+d)^{L_{[0]}} h_{[1]}u), \al, \tau) \\
= {} & \sum_{i} P^{i+1}(\tau') S'((v', z', h_{[i]}u), \al', \tau')
- 2\pi i c (c\tau+d) S'((v', z', h_{[1]}u), \al', \tau') \\
= {} & S'((v', z', \res_t \wp_1(t, \tau') h[t] u), \al', \tau') \tag{C0}\label{eqn:y.C0} \\
&- 2\pi i c (c\tau+d) S'((v', z', h_{[1]}u), \al', \tau'). \tag{C'}\label{eqn:y.C'}
\end{align*}
It remains to check that \eqref{eqn:y.C'} + \eqref{eqn:y.C2} + \eqref{eqn:y.D1} + \eqref{eqn:y.D2} = 0. But \eqref{eqn:y.C'} + \eqref{eqn:y.D1} = 0, and once we set
\[
\kappa = \frac{\left<h, h\right>}{2}
\]
we get \eqref{eqn:y.C2} + \eqref{eqn:y.D2} = 0.

\subsubsection{Modular invariance of the $q$-equation}\label{sec:inv.DE.q}

Let $S((v, z, u), \al, \tau)$ be a flat section of $\CC$ and let $S'((v', z', u), \al', \tau')$ be given by equation \eqref{eq:mod.xform}. We need to show \eqref{eq:DE.xform.q}. We compute
\begin{align*}
2\pi i\frac{\partial}{\partial\tau'}S'((v', z', u), \al', \tau')
= {} & 2\pi i(c\tau+d)^2 \frac{\partial}{\partial\tau}S'((v', z', u), \al', \tau') \tag{I1}\label{eqn:I1} \\
&+ 2\pi i c\al (c\tau+d) \frac{\partial}{\partial\al} S'((v', z', u), \al', \tau'). \tag{I2}\label{eqn:I2}
\end{align*}
Taking advantage of the calculations done above in Section \ref{sec:inv.DE.y}, we have
\begin{align*}
\text{\eqref{eqn:I2}} = 2\pi i c\al S'((v', z', \res_t \wp_1(t, \tau') h[t] u), \al', \tau').
\end{align*}
We now analyse \eqref{eqn:I1}.
\begin{align*}
\text{\eqref{eqn:I1}} = {} & 2\pi i(c\tau+d)^2 \frac{\partial}{\partial\tau}\left(e^{2\pi i \kappa \frac{c}{c\tau+d} \al^2} S((v, z, e^{2\pi i\frac{c\al}{c\tau+d} h_{[1]}}(c\tau+d)^{L_{[0]}} u), \al, \tau)\right) \\
= {} & (2\pi i)(c\tau+d)^2 \frac{\partial}{\partial\tau}\left(e^{2\pi i \kappa \frac{c}{c\tau+d} \al^2} (c\tau+d)^{\nabla_u} \sum_{j=0}^\infty \frac{1}{j!} \left(2\pi i\frac{c\al}{c\tau+d}\right)^j S((v, z,  h_{[1]}^j u), \al, \tau)\right) \\
= {} & (2\pi i)(c\tau+d)^2 \sum_{j=0}^\infty \frac{1}{j!} \frac{\partial}{\partial\tau}\left(e^{2\pi i \kappa \frac{c}{c\tau+d} \al^2} (c\tau+d)^{\nabla_u} \left(2\pi i\frac{c\al}{c\tau+d}\right)^j \right) S((v, z,  h_{[1]}^j u), \al, \tau) \tag{D}\label{eqn:D} \\
&+ (2\pi i)e^{2\pi i \kappa \frac{c}{c\tau+d} \al^2}(c\tau+d)^{\nabla_u+2} \sum_{j=0}^\infty \frac{1}{j!} \left(2\pi i\frac{c\al}{c\tau+d}\right)^j \frac{\partial}{\partial\tau} S((v, z,  h_{[1]}^j u), \al, \tau). \tag{C}\label{eqn:C}
\end{align*}
We leave \eqref{eqn:C} until later, focusing for now on \eqref{eqn:D}. We have
\begin{align*}
\text{\eqref{eqn:D}} = {} &
2\pi i (c\tau+d)^2 \frac{\partial}{\partial \tau}\left(e^{2\pi i \kappa \frac{c}{c\tau+d} \al^2} \right) \sum_{j=0}^\infty \frac{1}{j!} (c\tau+d)^{\nabla_u} \left(2\pi i\frac{c\al}{c\tau+d}\right)^j S((v, z,  h_{[1]}^j u), \al, \tau)  \tag{D1}\label{eqn:D1} \\
&+ 2\pi i (c\tau+d)^2 \frac{d}{d\tau}\left( (c\tau+d)^{\nabla_u}  \right) e^{2\pi i \kappa \frac{c}{c\tau+d} \al^2} \sum_{j=0}^\infty \frac{1}{j!} \left(2\pi i\frac{c\al}{c\tau+d}\right)^j S((v, z,  h_{[1]}^j u), \al, \tau)  \tag{D2}\label{eqn:D2} \\
&+ 2\pi i (c\tau+d)^{\nabla_u+2} e^{2\pi i \kappa \frac{c}{c\tau+d} \al^2} \sum_{j=0}^\infty \frac{1}{j!} \frac{\partial}{\partial\tau} \left(2\pi i\frac{c\al}{c\tau+d}\right)^j S((v, z,  h_{[1]}^j u), \al, \tau).  \tag{D3}\label{eqn:D3}
\end{align*}
We simplify each of these three terms in turn. We have
\begin{align*}
\text{\eqref{eqn:D1}} = {} & 2\pi i (c\tau+d)^2 \frac{\partial}{\partial\tau}\left(e^{2\pi i \kappa \frac{c}{c\tau+d} \al^2} \right) (c\tau+d)^{\nabla_u} S((v, z, e^{2\pi i\frac{c\al}{c\tau+d} h_{([1])} } u), \al, \tau) \\
= {} & 2\pi i (c\tau+d)^2 \frac{\partial}{\partial\tau}\left(e^{2\pi i \kappa \frac{c}{c\tau+d} \al^2} \right) S((v, z, e^{2\pi i\frac{c\al}{c\tau+d} h_{([1])} } (c\tau+d)^{L_{[0]}} u), \al, \tau) \\
= {} & -2\pi i (c\tau+d)^2 (2\pi i) \kappa \frac{c^2}{(c\tau+d)^2}\al^2  e^{2\pi i \kappa \frac{c}{c\tau+d} \al^2} S((v, z, e^{2\pi i\frac{c\al}{c\tau+d} h_{([1])} } (c\tau+d)^{L_{[0]}} u), \al, \tau) \\
= {} & -2\pi i (c\tau+d)^2 (2\pi i) \kappa \frac{c^2}{(c\tau+d)^2}\al^2 S'((v', z', u), \al', \tau') \\
= {} & -(2\pi i c \al)^2 \kappa S'((v', z', u), \al', \tau').
\end{align*}
Next
\begin{align*}
\text{\eqref{eqn:D2}} = {} & 2\pi i (c\tau+d)^2 \frac{d}{d\tau}\left( (c\tau+d)^{\nabla_u} \right) e^{2\pi i \kappa \frac{c}{c\tau+d} \al^2} S((v, z,  e^{2\pi i\frac{c\al}{c\tau+d} h_{([1])} } u), \al, \tau) \\
= {} & 2\pi i c \nabla_u (c\tau+d)^{\nabla_u+1} e^{2\pi i \kappa \frac{c}{c\tau+d} \al^2} S((v, z,  e^{2\pi i\frac{c\al}{c\tau+d} h_{([1])} } u), \al, \tau) \\
= {} & 2\pi i c \nabla_u (c\tau+d) e^{2\pi i \kappa \frac{c}{c\tau+d} \al^2} S((v, z,  e^{2\pi i\frac{c\al}{c\tau+d} h_{([1])} }(c\tau+d)^{L_{[0]}} u), \al, \tau) \\
= {} & 2\pi i c \nabla_u (c\tau+d) S'((v', z', u), \al', \tau').
\end{align*}
And
\begin{align*}
\text{\eqref{eqn:D3}} = {} & 2\pi i (c\tau+d)^{\nabla_u+2} e^{2\pi i \kappa \frac{c}{c\tau+d} \al^2} \sum_{j=0}^\infty \frac{1}{j!} \frac{\partial}{\partial\tau}\left(  \left(2\pi i\frac{c\al}{c\tau+d}\right)^j \right) S((v, z,  h_{[1]}^j u), \al, \tau) \\
= {} & 2\pi i (c\tau+d)^{\nabla_u+2} e^{2\pi i \kappa \frac{c}{c\tau+d} \al^2} \sum_{j=0}^\infty \frac{1}{j!} \frac{-2\pi i c^2\al}{(c\tau+d)^2} j \left(2\pi i\frac{c\al}{c\tau+d}\right)^{j-1}  S((v, z, h_{[1]}^j u), \al, \tau) \\
= {} & -(2\pi i)^2 c^2\al (c\tau+d)^{\nabla_u} e^{2\pi i \kappa \frac{c}{c\tau+d} \al^2} \sum_{j=0}^\infty \frac{1}{(j-1)!} \left(2\pi i\frac{c\al}{c\tau+d}\right)^{j-1}  S((v, z, h_{[1]}^j u), \al, \tau) \\
= {} & -(2\pi i)^2 c^2\al (c\tau+d)^{\nabla_u} e^{2\pi i \kappa \frac{c}{c\tau+d} \al^2} S((v, z, e^{2\pi i\frac{c\al}{c\tau+d}h_{[1]}} h_{([1])} u), \al, \tau) \\
= {} & -(2\pi i)^2 c^2\al (c\tau+d) e^{2\pi i \kappa \frac{c}{c\tau+d} \al^2} S((v, z, e^{2\pi i\frac{c\al}{c\tau+d}h_{[1]}} (c\tau+d)^{L_{[0]}} h_{([1])} u), \al, \tau) \\
= {} & -(2\pi i)^2 c^2\al (c\tau+d) S'((v', z', h_{[1]} u), \al', \tau').
\end{align*}

We use the DE to simplify \eqref{eqn:C}, obtaining
\begin{align*}
\text{\eqref{eqn:C}}
= {} & e^{2\pi i \kappa \frac{c}{c\tau+d} \al^2}(c\tau+d)^{\nabla_u+2} \sum_{j=0}^\infty \frac{1}{j!} \left(2\pi i\frac{c\al}{c\tau+d}\right)^j S((v, z, \res_t \wp_1(t, \tau) L[t] h_{[1]}^j u), \al, \tau) \\
= {} & e^{2\pi i \kappa \frac{c}{c\tau+d} \al^2}(c\tau+d)^{\nabla_u+2} S((v, z, \res_t \wp_1(t, \tau) L[t] e^{2\pi i\frac{c\al}{c\tau+d} h_{[1]}} u), \al, \tau).
\end{align*}
To proceed further we need to commute $L[t]$ with $e^{h_{[1]}}$, using equation \eqref{eq:L[t].exp.comm}. The term (C) equals
\begin{align*}
e^{2\pi i \kappa \frac{c}{c\tau+d} \al^2}(c\tau+d)^{\nabla_u+2} S\left(\left(v, z, \res_t \wp_1(t, \tau) e^{\la h_{[1]}} \left( L[t] - \la h[t] + \la^2\frac{\left<h, h\right> }{2} \right) u\right), \al, \tau \right)
\end{align*}
(where $\la = 2\pi i\frac{c\al}{c\tau+d}$). Let's identify the three terms with $L[t]$, $- \la h[t]$ and $\la^2 \frac{\left<h, h\right>}{2}$ as (C1), (C2) and (C3), respectively. First let's focus on
\begin{align*}
\text{(C1)} = {} & e^{2\pi i \kappa \frac{c}{c\tau+d} \al^2}(c\tau+d)^{\nabla_u+2} S((v, z, \res_t \wp_1(t, \tau) e^{2\pi i\frac{c\al}{c\tau+d} h_{[1]}} L[t] u), \al, \tau) \\
= {} & e^{2\pi i \kappa \frac{c}{c\tau+d} \al^2}  (c\tau+d)^{\nabla_u+2} \sum_i P^{i+2}(\tau) S((v, z, e^{2\pi i\frac{c\al}{c\tau+d} h_{[1]}} L_{[i]}u), \al, \tau) \\
= {} & e^{2\pi i \kappa \frac{c}{c\tau+d} \al^2}\sum_i (c\tau+d)^{\nabla_u-i} P^{i+2}(\tau') S((v, z, e^{2\pi i\frac{c\al}{c\tau+d} h_{[1]}} L_{[i]}u), \al, \tau) \tag{C10}\label{eqn:C10} \\
&- 2\pi i c e^{2\pi i \kappa \frac{c}{c\tau+d} \al^2} (c\tau+d)^{\nabla_u+1} S((v, z, e^{2\pi i\frac{c\al}{c\tau+d} h_{[1]}} L_{[0]}u), \al, \tau). \tag{C1'}\label{eqn:C1'}
\end{align*}
Notice that
\begin{align*}
\text{\eqref{eqn:C1'}} = {} & - 2\pi i c \nabla_u e^{2\pi i \kappa \frac{c}{c\tau+d} \al^2} (c\tau+d)^{\nabla_u+1} S((v, z, e^{2\pi i\frac{c\al}{c\tau+d} h_{[1]}} u), \al, \tau) \\
= {} & - 2\pi i c \nabla_u e^{2\pi i \kappa \frac{c}{c\tau+d} \al^2} (c\tau+d) S((v, z, e^{2\pi i\frac{c\al}{c\tau+d} h_{[1]}}(c\tau+d)^{L_{[0]}} u), \al, \tau) \\
= {} & - 2\pi i c \nabla_u (c\tau+d) S'((v', z', u), \al', \tau').
\end{align*}
The main term \eqref{eqn:C10} continues as
\begin{align*}
\text{\eqref{eqn:C10}} = {} & e^{2\pi i \kappa \frac{c}{c\tau+d} \al^2}\sum_i P^{i+2}(\tau') S((v, z, e^{2\pi i\frac{c\al}{c\tau+d} h_{[1]}} (c\tau+d)^{L_{[0]}} L_{[i]}u), \al, \tau) \\
= {} & \sum_i P^{i+2}(\tau') S'((v', z', L_{[i]}u), \al', \tau') \\
= {} & S'((v', z', \res_t \wp_1(t, \tau') L[t] u), \al', \tau').
\end{align*}

%

Next is
\begin{align*}
\text{(C2)} = {} & -2\pi i\frac{c\al}{c\tau+d} e^{2\pi i \kappa \frac{c}{c\tau+d} \al^2}(c\tau+d)^{\nabla_u+2} S((v, z, \res_t \wp_1(t, \tau) e^{2\pi i\frac{c\al}{c\tau+d} h_{[1]}} h[t]  u), \al, \tau) \\
= {} & -2\pi i {c\al} e^{2\pi i \kappa \frac{c}{c\tau+d} \al^2}(c\tau+d)^{\nabla_u+1} \sum_i P^{i+1}(\tau) S((v, z, e^{2\pi i\frac{c\al}{c\tau+d} h_{[1]}} h_{[i]} u), \al, \tau) \\
= {} & -2\pi i {c\al} e^{2\pi i \kappa \frac{c}{c\tau+d} \al^2} \sum_i (c\tau+d)^{\nabla_u-i} P^{i+1}(\tau') S((v, z, e^{2\pi i\frac{c\al}{c\tau+d} h_{[1]}} h_{[i]} u), \al, \tau) \tag{C20}\label{eqn:C20} \\
&+ (2\pi i c)^2\al e^{2\pi i \kappa \frac{c}{c\tau+d} \al^2}(c\tau+d)^{\nabla_u} S((v, z, e^{2\pi i\frac{c\al}{c\tau+d} h_{[1]}} h_{[1]} u), \al, \tau). \tag{C2'}\label{eqn:C2'}
\end{align*}
Note that
\begin{align*}
\text{\eqref{eqn:C2'}} = {} & (2\pi i c)^2\al e^{2\pi i \kappa \frac{c}{c\tau+d} \al^2} (c\tau+d) S((v, z, e^{2\pi i\frac{c\al}{c\tau+d} h_{[1]}} (c\tau+d)^{L_{[0]}} h_{[1]} u), \al, \tau) \\
= {} & (2\pi i c)^2\al (c\tau+d) S'((v', z', h_{[1]} u), \al', \tau').
\end{align*}
The main term \eqref{eqn:C20} continues as
\begin{align*}
\text{\eqref{eqn:C20}} = {} & -2\pi i {c\al} e^{2\pi i \kappa \frac{c}{c\tau+d} \al^2} \sum_i (c\tau+d)^{\nabla_u-i} P^{i+1}(\tau') S((v, z, e^{2\pi i\frac{c\al}{c\tau+d} h_{[1]}} h_{[i]} u), \al, \tau) \\
= {} & -2\pi i {c\al} e^{2\pi i \kappa \frac{c}{c\tau+d} \al^2}  S((v, z, \res_t \wp_1(t, \tau') e^{2\pi i\frac{c\al}{c\tau+d} h_{[1]}} (c\tau+d)^{L_{[0]}} h[t] u), \al, \tau) \\
= {} & -2\pi i {c\al} S'((v', z', \res_t \wp_1(t, \tau') h[t] u), \al', \tau').
\end{align*}
Next is (C3) which simplifies as:
\begin{align*}
\text{(C3)} = {} & \frac{\left<h,h\right>}{2} (2\pi i\frac{c\al}{c\tau+d})^2 e^{2\pi i \kappa \frac{c}{c\tau+d} \al^2}(c\tau+d)^{\nabla_u+2} S((v, z, e^{2\pi i\frac{c\al}{c\tau+d} h_{[1]}} u), \al, \tau) \\
= {} & \frac{\left<h,h\right>}{2} (2\pi i {c\al})^2 e^{2\pi i \kappa \frac{c}{c\tau+d} \al^2} S((v, z, e^{2\pi i\frac{c\al}{c\tau+d} h_{[1]}} (c\tau+d)^{L_{[0]}} u), \al, \tau) \\
= {} & \frac{\left<h,h\right>}{2} (2\pi i {c\al})^2 S'((v', z', u), \al', \tau').
\end{align*}
Next we note that \eqref{eqn:I2} + (C20) = 0, (D1) + (C3) = 0 (assuming $\kappa = \left<h,h\right>/2$), (D2) + (C1') = 0, (D3) + (C2') = 0. The only remaining term is (C10), so we have proved \eqref{eq:DE.xform.q} at last.

\subsection{Verification of invariance under elliptic transformations}\label{app.ell}

Let $(m, n) \in \Z^2$ and $(\al', \tau') = (m, n) \cdot (\al, \tau)$ as in Section \ref{sec:modularity}. The partial derivatives in $\tau$ and $\al$ transform as follows:
\begin{align*}
\frac{\partial}{\partial \al'} &= \frac{\partial}{\partial \al} \\
\frac{\partial}{\partial \tau'} &= \frac{\partial}{\partial \tau} - m \frac{\partial}{\partial \al}.
\end{align*}

\subsubsection{Elliptic invariance of the $y$-equation}\label{sec:inv.DE.y.ell}

Let $S((v, z, u), \al, \tau)$ be a flat section of $\CC$ and let $S'((v', z', u), \al', \tau')$ be given by equation \eqref{eq:ell.xform}. We need to show \eqref{eq:DE.xform.y}. We compute
\begin{align*}
\frac{\partial}{\partial\al'}S'((v', z', u), \al', \tau')
&= \frac{\partial}{\partial\al}S'((v', z', u), \al', \tau') \\
= {} & \frac{\partial}{\partial\al}\left( e^{- 2\pi i \kappa (m^2 \tau + 2m \al) } S( ( v, z, e^{2\pi i m h_{[1]}} u), \al, \tau) \right) \\
= {} & -2\pi i (2m \kappa) e^{- 2\pi i \kappa (m^2 \tau + 2m \al) } S( ( v, z, e^{-2\pi i m h_{[1]}} u), \al, \tau) \tag{D}\label{eqn:y.D.ell} \\
& + e^{- 2\pi i \kappa (m^2 \tau + 2m \al) } \frac{\partial}{\partial\al} S( ( v, z, e^{-2\pi i m h_{[1]}} u), \al, \tau). \tag{C}\label{eqn:y.C.ell}
\end{align*}

We simplify \eqref{eqn:y.C.ell}, rewriting it using \eqref{eq:DE.y} as
\begin{align*}
\text{\eqref{eqn:y.C.ell}}
= {} & e^{- 2\pi i \kappa (m^2 \tau + 2m \al) } S((v, z, \res_t \wp_1(t, \tau) h[t] e^{-2\pi i m h_{[1]}} u), \al, \tau).
\end{align*}
To proceed further we commute $h[t]$ with $e^{h_{[1]}}$ using equation \eqref{eq:h[t].exp.comm} (with $\la = -2\pi i m$). We have
\begin{align*}
\text{\eqref{eqn:y.C.ell}} = {} & e^{- 2\pi i \kappa (m^2 \tau + 2m \al) } S((v, z, \res_t \wp_1(t, \tau) e^{-2\pi i m h_{[1]}} h[t] u), \al, \tau) \tag{C1}\label{eqn:y.C1.ell} \\
&+ 2\pi i{m \left<h, h\right>} e^{- 2\pi i \kappa (m^2 \tau + 2m \al) } S((v, z, \res_t \wp_1(t, \tau)e^{-2\pi i m h_{[1]}} u), \al, \tau). \tag{C2}\label{eqn:y.C2.ell}
\end{align*}
We continue
\begin{align*}
\text{\eqref{eqn:y.C2.ell}} = {} & 2\pi i{m \left<h, h\right>} e^{- 2\pi i \kappa (m^2 \tau + 2m \al) } S((v, z, e^{-2\pi i m h_{[1]}} u), \al, \tau) \\
= {} & 2\pi i{m \left<h, h\right>} S'((v', z', u), \al', \tau').
\end{align*}
Since $\tau' = \tau$ the transformation of $\wp_1(t, \tau)$ is trivial and we have
\begin{align*}
\text{\eqref{eqn:y.C1.ell}} = {} & e^{- 2\pi i \kappa (m^2 \tau + 2m \al) } S((v, z, \res_t \wp_1(t, \tau) e^{-2\pi i m h_{[1]}} h[t] u), \al, \tau) \\
= {} & e^{- 2\pi i \kappa (m^2 \tau + 2m \al) } S((v', z', \res_t \wp_1(t, \tau') h[t] u), \al', \tau')
\end{align*}
as required.

\subsubsection{Elliptic invariance of the $q$-equation}\label{sec:inv.DE.q.ell}

Let $S((v, z, u), \al, \tau)$ be a flat section of $\CC$ and let $S'((v', z', u), \al', \tau')$ be given by equation \eqref{eq:ell.xform}. We need to show \eqref{eq:DE.xform.q}. We compute
\begin{align*}
2\pi i\frac{\partial}{\partial\tau'}S'((v', z', u), \al', \tau')
= {} & 2\pi i \frac{\partial}{\partial\tau}S((v', z', u), \al', \tau') \tag{I1}\label{eqn:I1.ell} \\
&- 2\pi i m \frac{\partial}{\partial\al} S'((v', z', u), \al', \tau'). \tag{I2}\label{eqn:I2.ell}
\end{align*}
Taking advantage of the calculations done above in Section \ref{sec:inv.DE.y.ell}, we have
\begin{align*}
\text{\eqref{eqn:I2.ell}} = -2\pi i m S'((v', z', \res_t \wp_1(t, \tau') h[t] u), \tau', \al' ).
\end{align*}
We now analyse \eqref{eqn:I1.ell}.

We compute
\begin{align*}
\frac{\partial}{\partial\tau}S'((v', z', u), \al', \tau')
= {} & \frac{\partial}{\partial\tau}\left( e^{- 2\pi i \kappa (m^2 \tau + 2m \al) } S( ( v, z, e^{2\pi i m h_{[1]}} u), \al, \tau) \right) \\
= {} & -2\pi i (m^2 \kappa) e^{- 2\pi i \kappa (m^2 \tau + 2m \al) } S( ( v, z, e^{-2\pi i m h_{[1]}} u), \al, \tau) \tag{D}\label{eqn:q.D.ell} \\
& + e^{- 2\pi i \kappa (m^2 \tau + 2m \al) } \frac{\partial}{\partial\tau} S( ( v, z, e^{-2\pi i m h_{[1]}} u), \al, \tau). \tag{C}\label{eqn:q.C.ell}
\end{align*}

We simplify \eqref{eqn:y.C.ell}, rewriting it using \eqref{eq:DE.y} as
\begin{align*}
\text{\eqref{eqn:y.C.ell}}
= {} & e^{- 2\pi i \kappa (m^2 \tau + 2m \al) } S((v, z, \res_t \wp_1(t, \tau) h[t] e^{-2\pi i m h_{[1]}} u), \al, \tau).
\end{align*}
To proceed further we commute $L[t]$ with $e^{h_{[1]}}$ using equation \eqref{eq:L[t].exp.comm} (with $\la = -2\pi i m$). The term (C) equals
\begin{align*}
e^{- 2\pi i \kappa (m^2 \tau + 2m \al) } S((v, z, \res_t \wp_1(t, \tau) e^{-2\pi i m h_{[1]}} \left(L[t] + (2\pi i m) h[t] + (2\pi i m)^2\frac{\left<h,h\right>}{2} \right) u), \al, \tau).
\end{align*}
Let's identify the three terms with $L[t]$, $(2\pi i m) h[t]$ and $(2\pi i m)^2\frac{\left<h,h\right>}{2}$ as (C1), (C2) and (C3), respectively. First we have
\begin{align*}
\text{(C1)} = {} & e^{- 2\pi i \kappa (m^2 \tau + 2m \al) } S((v, z, \res_t \wp_1(t, \tau) e^{-2\pi i m h_{[1]}} L[t] u), \al, \tau) \\
= {} & e^{- 2\pi i \kappa (m^2 \tau + 2m \al) } S'((v', z', \res_t \wp_1(t, \tau') L[t] u), \al', \tau').
\end{align*}
Now (C2) cancels (I2) and (C3) cancels (D).

\bigskip

\small

\noindent
T.A.:\newline
Okinawa Institute of Science and Technology (OIST)\\
Onna, Kunigami District, Okinawa 904-0495, Japan\\
{\tt tomoyuki.arakawa@oist.jp}

\vspace{5 mm}

\noindent
J.vE.:\newline
Instituto de Matem\'{a}tica Pura e Aplicada (IMPA)\newline
Jardim Bot\^{a}nico, Rio de Janeiro, RJ 22.460-320, Brazil\newline
{\tt jethro@impa.br}

\vspace{5 mm}

\noindent
H.L:\newline
Okinawa Institute of Science and Technology (OIST)\\
Onna, Kunigami District, Okinawa 904-0495, Japan\\
{\tt hao-li@oist.jp}

\end{document}